\documentclass[12pt]{article}
\newcommand{\function}[2]{\colon #1 \rightarrow #2}
\newcommand{\injection}[2]{\colon #1 \rightarrowtail #2}
\newcommand{\interpret}[2]{\colon #1 \leadsto #2}
\newcommand{\of}[1]{\left( #1 \right)}
\newcommand{\set}[2]{\left\{\hspace{0.2ex} #1 \;\middle\vert\; #2 \right\}}
\usepackage{bm}
\newcommand{\rn}{\bm}
\newcommand{\df}{\stackrel{\rm def}{=}}
\newcommand{\given}{\;\middle\vert\;}

\usepackage{theorem}
\newtheorem{theorem}{Theorem}[section]
\newtheorem{lemma}[theorem]{Lemma}
\newtheorem{corollary}[theorem]{Corollary}
\newtheorem{proposition}[theorem]{Proposition}

{\theorembodyfont{\rmfamily}
\newtheorem{definition}[theorem]{Definition}
\newtheorem{remark}{Remark}

\newtheorem{example}{Example}
}

\newcommand{\beginproof}{\medskip\noindent{\bf Proof.~}}
\newcommand{\beginproofof}[1]{\medskip\noindent{\bf Proof of #1.~}}
\newcommand{\beginproofdotless}{\medskip\noindent{\bf Proof}}
\newcommand{\finishproof}{\hspace{0.2ex}\rule{1ex}{1ex}}
\newenvironment{proof}{\beginproof}{\unskip\nolinebreak\finishproof\par\medskip}
\newenvironment{proofof}[1]{\beginproofof{#1}}{\unskip\nolinebreak
\finishproof\par\medskip}

\usepackage{amsfonts}
\usepackage{amsmath,amscd}
\usepackage{amssymb}

\usepackage{mathtools} 
\usepackage{array} 

\usepackage{tikz} 
\usetikzlibrary{calc} 
\usetikzlibrary{arrows} 
\usetikzlibrary{cd} 
\usepackage{subcaption} 

\usepackage{enumitem}
\usepackage{microtype}

\def\Dopen{D_{\operatorname{open}}}
\def\PDopen{\operatorname{PD}_{\operatorname{open}}}
\def\ind{\operatorname{ind}}
\def\inj{\operatorname{inj}}
\def\tind{t_{\ind}}
\def\tinj{t_{\inj}}
\def\Tind{T_{\ind}}
\def\Tinj{T_{\inj}}
\def\comp{\mathbin{\circ}}
\def\symmdiff{\mathbin{\triangle}}
\def\Aut{\operatorname{Aut}}
\def\Tr{\operatorname{Tr}}
\def\im{\operatorname{im}}

\newenvironment{functiondef}[1][1]%
{\begin{tabular}{*{2}{>{$\displaystyle}r<{$}}*{#1}{>{$\displaystyle}c<{$}}>{$\displaystyle}l<{$}}}%
{\end{tabular}}

\newcommand{\prob}[1]{\mathbb{P}\left[#1\right]}
\newcommand{\condprob}[2]{\mathbb{P} \left[#1 \middle\vert \hspace{0.5mm} #2 \right]}
\newcommand{\indprob}[2]{\mathbb{P}_{#1} \left[#2\right]}
\newcommand{\expect}[1]{\mathbb{E}\left[#1\right]}
\newcommand{\indexpect}[2]{\mathbb{E}_{#1}\left[#2\right]}

\newcommand{\plainHom}[1]{\operatorname{Hom}^+(#1, \mathbb{R})}
\newcommand{\Hom}{\plainHom{\mathcal{A}}}
\newcommand{\HomT}[1]{\plainHom{\mathcal{A}[#1]}}

\newcommand{\place}{\mathord{-}}
\newcommand{\Forb}[2]{\operatorname{Forb}_{#1}(#2)}
\newcommand{\Forbp}[2]{\operatorname{Forb}^+_{#1}(#2)}

\title{Semantic Limits of Dense Combinatorial Objects}
\author{Leonardo Nagami Coregliano\thanks{University of Chicago, {\tt lenacore@uchicago.edu}} \and Alexander A.~Razborov\thanks{University of Chicago,
{\tt razborov@math.uchicago.edu}, Steklov Mathematical Institute, Moscow, {\tt razborov@mi-ras.ru} and Toyota Technological Institute, Chicago}
}
\begin{document}
\maketitle

\begin{abstract}
The theory of limits of discrete combinatorial objects has been thriving for the last decade or so. The syntactic, algebraic
approach to the subject is popularly known as ``flag algebras'', while the semantic, geometric one is often associated with
the name ``graph limits''. The language of graph limits is generally more intuitive and expressible, but a price that one
has to pay for it is that it is better suited for the case of ordinary graphs than for more general combinatorial objects.
Accordingly, there have been several attempts in the literature, of varying degree of generality, to define limit objects
for more complicated combinatorial structures.

This paper is another attempt at a workable
general theory of dense limit objects.
Unlike previous efforts in this direction (with notable exception of~\cite{ArC}),
we base our account on the same concepts from the first-order logic and the model
theory as in the theory of flag algebras.

We show how our definition naturally encompasses a host of previously considered
cases (graphons, hypergraphons, digraphons, permutons, posetons, colored graphs, etc.), and we extend the
fundamental properties of existence and uniqueness to this more
general case. We also give an intuitive general proof of the continuous version
of the Induced Removal Lemma based on the completeness theorem for propositional
calculus. We capitalize on the notion of an open
interpretation that often allows to transfer methods and results from one
situation to another. Again, we show that some previous arguments can be quite
naturally framed using this language.
\end{abstract}

\section{Introduction}
Extremely large objects and structures have greatly expanded from their natural
habitat in mathematics (asymptotic constructions in analysis, combinatorics,
etc.), statistics and statistical physics, and these days they are
ubiquitous. Many of them naturally include or even are entirely
comprised of continuous data; those are not considered here. However,
even after taking numerical data structures out of the picture, there remains
a significant body of huge objects that are completely discrete and
combinatorial in their nature, at least a priori. We refer the reader
to~\cite[Part~1]{Lov4} for a lovely discussion of this paradigm accompanied by
many examples for ordinary graphs.

The philosophical question that has given rise to a virtually new discipline
that can be provisionally called ``continuous combinatorics'' is this. If we
begin with a very large {\em combinatorial} structure, can it still be studied
using analytical tools or is its ``discreteness'' a natural inhibition to it?
It should be noted here, of course, that there is a plethora of numerical
characteristics associated with combinatorial objects that are important and
interesting for their own sake, those have been studied in combinatorics
for centuries. So, we would like to stress that we mean a
reasonably general and coherent theory in good mathematical sense; with its internal logic
and structure, natural and preferably unexpected connections between different
parts, connections to other mathematical disciplines and theoretical computer
science, etc.

\smallskip
By now it has become reasonably clear that we have a satisfactory answer to this question, at least in the dense setting to
which the forthcoming discussion is confined. It has turned out that the ``primary'', ``basic'' set of numerical characteristics responsible for many a priori unrelated properties of a combinatorial structure $M$ is made by densities with which small ``templates'', that is, fixed size structures of the same kind, occur in $M$.
Note that it is absolutely not obvious a priori why this set of parameters should be any better and any more universal than,
say, the chromatic number of a graph or the dominance number in a tournament. But this claim is strongly supported by
ample ``empirical'' evidence gathered in this young area.

Once we know what are the basic properties of huge combinatorial objects making the backbone of the theory, there are,
as it often happens, two complementary approaches to the task.

The ``semantical'' or ``geometric'' approach asks if it is possible to find the actual (limit) object on which these numerical parameters are ``imprinted'', and then it naturally proceeds to studying these objects. This approach is collectively known as ``graph limits'', and it has achieved quite a spectacular success in the case of ordinary graphs, with very beautiful, deep and elegant structural results involving many ideas and concepts from other areas of mathematics. We refer the reader to Parts~2 and~3 of the monograph~\cite{Lov4} for a comprehensive (that is, at the moment of its release) account of the subject.

One drawback of this theory, however, is that it tends to be tied to ordinary graphs. Extensions of graph limits to several other kinds of combinatorial structures are known and, let us note in the brackets, have been very inspirational for this work. Still, it would be fair to say (cf.~\cite[Part~5]{Lov4}) that, in contrast to the elaborated theory for graphons, most of them tend to be on somewhat ad hoc side. The only attempts at a completely general theory we are
aware of were undertaken in~\cite{Aus,ArC}.

The approach often called ``flag algebras''~\cite{flag} has precisely the opposite set of features. It is manifestly minimalist: we do not even try to define limit objects, but instead argue about the densities of occurrences of small templates in purely syntactic, algebraic way. The immediate advantage is the generality of the theory: its abstract techniques apply to arbitrary combinatorial structures in the same uniform way. Also, the lighter and in a sense single-purpose notational system makes it much better suited for proving \emph{concrete} results in extremal combinatorics; we refer the reader to the survey~\cite{flag_survey} for a comprehensive (again, at the time of its release) list of such results. The disadvantages of flag algebras are also clear, of course. Limit objects certainly {\em are} extremely interesting and natural entities to study in their own right, and the theory that does not even address their existence is necessarily single-minded. Another issue is that even when a structural result or a construction can be formulated in the restricted language of flag algebras, its purely syntactic proofs can often be awkward, see e.g.\ most of~\cite[\S 3-4]{flag}. A workable semantics would have made these proofs
straightforward or, as the very least, more natural.

\medskip
\noindent
{\bf Our contributions.}
In this paper we attempt to develop a general definition (and a vocabulary) of limit objects suitable for combinatorial structures of
arbitrary type based on the same formalism that was used in flag algebras. More precisely, we adapt models of universal first-order theories as our preferred language of working with general combinatorial structures. An added benefit is that we can use  the
concept of {\em interpretation} well established in the mathematical logic for relating structures of different (or the same) kinds. It should
be noted that a principal possibility of this approach was briefly sketched in~\cite[\S 4.3]{Aus} and~\cite{ArC} but we make it significantly more
systematic, and we do emphasize the importance of working with arbitrary universal first-order theories (as opposed to pure first-order
calculus) and their open interpretations.

\smallskip
We begin in Section~\ref{sec:prel} with a brief overview of those parts of mathematical logic, graph limits and flag algebras that
are needed for our purposes.

In Section~\ref{sec:peonstheons} we first define limit objects for the case
when our theory consists of a single predicate $P$ without any additional
constraints; following the well-established pattern, we call such objects
{\em $P$-ons} ({\em pe-ons}, Definition~\ref{def:peons}). In analogy with model theory, we define the notion of an Euclidean structure in a language $\mathcal{L}$ as a list of $P$-ons with $P$ ranging in $\mathcal{L}$. Then we generalize
this definition to arbitrary universal theories $T$, that results in our main
definition of {\em $T$-ons} ({\em the-ons}, Definition~\ref{def:theons}).
There are two reasonable ways to define $T$-ons: by requiring that additional
axioms of the theory $T$ are satisfied almost everywhere ({\em weak theons}), or by
demanding that they are satisfied everywhere ({\em strong theons}). The equivalence
of these two alternatives is the content of the {\em Induced Euclidean
Removal Lemma} (Theorem~\ref{thm:ierl}). Also, in this section we state our
main results: the existence of theons (Theorem~\ref{thm:theoncryptomorphism})
and their uniqueness (Theorems~\ref{thm:TheonUniqueness} and~\ref{thm:TheonUniquenessSecond}). Collectively, these
results deliver what is the main technical contribution of our work: for an arbitrary
universal theory $T$ without constants and function symbols, $T$-ons are
categorically equivalent (in Lov\'{a}sz's terminology, ``cryptomorphic'')
to convergent sequences and flag-algebraic homomorphisms
$\HomT{T}$.

The next three sections are devoted to proofs. We start with the Induced Euclidean Removal Lemma in Section~\ref{sec:removal}, primarily because it is simpler. Its proof essentially uses the axiom of choice, although we prefer to disguise this usage as the completeness theorem for propositional calculus (with uncountably many variables). Without the axiom of choice we can only show the corresponding result for the case of Horn theories (Theorem~\ref{thm:herl}), which in particular implies a non-induced version (Corollary~\ref{cor:erl}) of the general theorem. We also include an ad hoc constructive (that is, Borel and without the axiom of choice) proof for the theory of linear orders that, arguably, is the most prominent non-Horn theory (Theorem~\ref{thm:strongLinOrder}).

In Section~\ref{sec:existence-uniqueness} we prove the existence and
uniqueness of $T$-ons. This is the most difficult part of the paper, and we
base our proofs on the Aldous--Hoover--Kallenberg theory of exchangeable
arrays, a connection that was apparently for the first time pointed out
independently by Diaconis and Janson~\cite{DiJa} and~Austin~\cite{Aus}. For
an ultraproduct version of these proofs see~\cite{ArC}.

In Section~\ref{sec:othercrypto} we give an ergodic characterization of
theons that generalizes a corresponding result for graphons.

In Section~\ref{sec:otherobjects} we develop a
formalism that allows us to represent in the unique ``theonic'' framework
several ad hoc limit objects previously considered in the literature. We exemplify this
approach with permutons, posetons, limits of interval graphs and limits of
subsets in finite vector spaces (all other dense limit objects previously
considered in the literature fit the framework straightforwardly).

The paper is concluded with a brief discussion and a few open problems
in Section~\ref{sec:conclusion}.

\medskip
Before we begin the technical part, three more remarks are in place. Firstly, this text, like virtually everything about
``continuous combinatorics'', touches, even if sometimes marginally, upon many areas of mathematics and theoretical
computer science. We cannot assume our readers to be versatile in all of them (neither we presumptuously deem ourselves
qualifying for this position). For this reason we try to go at a rather low speed and interlace the formal account with as many examples and references and as much intuition, informal explanations, etc.\ as possible.

Another remark, somewhat derivative from the first, concerns the novelty of our results. As turned out in the course of this work (and had been absolutely unclear before we had started), for a proper generalization we basically need to look at the existing literature under an
appropriate angle, and properly combine different pieces in it together. While we will try to give proper credit as we go along, we also feel it would be appropriate already at this early point to list some particularly inspirational sources, in no particular order and without attempting to be
comprehensive:
\begin{itemize}
\item Graph Limits~\cite{Lov4};

\item Flag Algebras~\cite{flag};

\item Exchangeable Arrays~\cite{Ald,Kall,Aus};

\item Quasi-Randomness~\cite{Tho2,CGW,ChGr2};

\item Hypergraphons~\cite{ElS};

\item Digraphons~\cite{DiJa};

\item Permutons~\cite{HKM*};

\item Removal Lemmas~\cite{AuT,Petr,ArC}.
\end{itemize}

Finally, let us remark that the theory developed in this work concerns only the dense setting: in our
framework, a limit object of a sparse sequence is trivial. In the sparse setting, it is
not even clear that convergence of densities of small templates is the correct notion of
convergence. Alternate notions of convergence such as convergence in cut-norm or convergence of quotients seem
to be much more aligned with applications of sparse graph limit theory, and, unlike in the dense setting,
these notions are distinct from convergence of densities. However, let us also note that concentration results
about densities in the Erd\H{o}s--R\'{e}nyi graph in the sparse setting, such as~\cite{ChDe}, suggest that theories
of quasi-randomness and of limits with respect to density convergence are also possible. For further details we refer
the interested reader to~\cite{BoRi09, BoRi11, BCCZ14, BCCZ18}.

\section{Preliminaries} \label{sec:prel}

We use throughout the standard combinatorial notation $[n]\df \{1,2,\ldots,n\}$ and $(n)_m\df n(n-1)\cdots (n-m+1)$.
Also, for a set $X$ we let $2^X$ be the collection of all its subsets. The notation $\rightarrowtail$ will always presume
that the mapping in question is injective. For a set $V$, $(V)_k$ will denote the set of all injective functions $\alpha$
of the form $\alpha\injection{[k]}{V}$. Random variables will be typed in the $\rn{math\ bold\ face}$. For two random
variables with values in the same $\sigma$-algebra, $\rn X\sim \rn Y$ will mean that $\rn X$ and $\rn Y$ are equidistributed.
We let $\mathbb N\df \{0,1,2,3,\ldots\}$ and $\mathbb N_+\df \{1,2,3,\ldots\}$. $S_n$ is the group of permutations on $n$
elements.

\subsection{Theories and models} \label{sec:model_theory}

As we noted in the introduction, our preferred way to represent combinatorial objects is based on rudimentary notions from first-order logic and model theory; we will try to stick to the notation of~\cite{ChKr,Barw} as much as possible.

A (first-order) {\em language}\footnote{Sometimes also called {\em signatures} or {\em vocabularies}, first-order languages may in general contain individual constants and function symbols. Those are not considered here.} is a finite set
$\mathcal L$ of {\em predicate symbols}. Each symbol $P\in\mathcal L$ comes along with a positive integer $k(P)$ that is called its {\em arity} and designates the number of variables $P$ depends on. Given our restrictions on the language $\mathcal L$ (no constants or
function symbols), {\em atomic formulas} may only have the form $P(x_{i_1},\ldots,x_{i_k})$ or $x_{i_1}=x_{i_2}$ (we do allow equality), and {\em open formulas} are made from atomic formulas using standard propositional connectives $\neg,\lor,\land,\to,\equiv$, etc. A {\em universal formula} is a formula of the form $\forall x_1\cdots\forall x_n F(x_1,\ldots,x_n)$, where $F$ is open.

\begin{remark} \label{rem:locally_finite} All or almost all notions and results in this text readily generalize to the case when the language $\mathcal L$ is {\em locally finite}, by which we mean that it contains only finitely many symbols in every fixed arity. But since we have only one interesting motivating example for this generalization (see Section~\ref{sec:lineons}), we prefer to keep things simpler.
\end{remark}

A {\em universal first-order theory $T$} in a language $\mathcal L$ is a set of universal formulas called {\em axioms}; universal quantifiers in front of the axioms are usually omitted. In most cases the set of axioms will also be finite, but it is not formally required in our framework. Universal first-order theories will be often called simply {\em theories} as we do not consider any others in this paper.

A {\em structure} $M$ in a language $\mathcal L$ is a set $V(M)$ (whose elements, in the recognition of the combinatorial nature of our work, will be usually called {\em vertices}), equipped with a mapping that assigns to every $P\in\mathcal L$ a $k(P)$-ary relation $R_{P,M}\subseteq V(M)^{k(P)}$. A structure $M$ is a {\em model} of a theory $T$ in the language $\mathcal L$ iff all axioms of $T$ are universally true on $M$ (see any textbook in the mathematical logic for a formal definition). It is extremely important to us that for any model $M$ of $T$ and any set of its vertices $V\subseteq V(M)$, after restricting all relations $R_{P,M}$ to $V$ we again obtain a model of $T$. It is called the {\em induced submodel} and denoted by $M\vert_V$. One important consequence is this. Let us say that $T$ {\em proves} or {\em entails} a universal formula $\forall\vec x F(\vec x)$, denoted by $T\vdash \forall\vec x F(\vec x)$ if it does so in the first-order logic. By the Completeness Theorem, it is equivalent to saying that $\forall \vec x F(\vec x)$ is true in any model of $T$. The submodel property allows us to conclude also that $T\vdash \forall x F(\vec x)$ if and only if this formula is true in any {\em finite} model of $T$. This does not hold in general.

Our restrictions on the language $\mathcal L$ and the theory $T$ are quite severe from the point of view of mathematical logic. Nonetheless, they turn out to be precisely right to capture the kind of combinatorial structures to which much of the previously developed machinery applies. The rest of this section is devoted to various examples.

\begin{example}[graphs, etc.] \label{ex:graphs}
The language $\mathcal L$ consists of a single binary predicate $E$. The theory $T_{\operatorname{Graph}}$ has the axioms
\begin{align*}
\neg E(x,x); & & E(x,y) & \equiv E(y,x).
\end{align*}
Thus, it is the theory of {\em simple} graphs: we forbid loops, all edges are undirected and have multiplicity one. Removing the axiom $E(x,y)\equiv E(y,x)$, we arrive at the theory of {\em directed} graphs $T_{\operatorname{Digraph}}$, while replacing it with
\begin{equation} \label{eq:orgraph}
E(x,y) \to \neg E(y,x),
\end{equation}
we get the theory of oriented graphs\footnote{In~\cite{flag}, it was called $T_{\operatorname{Digraph}}$.} $T_{\operatorname{Orgraph}}$.
Strengthening the axiom~\eqref{eq:orgraph} to
$$
x\neq y\to (E(x,y)\equiv \neg E(y,x)),
$$
we arrive at the theory of tournaments $T_{\operatorname{Tournament}}$.
\end{example}

\begin{example}[uniform hypergraphs]
Let $k\in\mathbb{N}_+$ be a fixed constant. The language $\mathcal L$ consists of a single $k$-ary predicate $E$, and the theory $T_{k\operatorname{-Hypergraph}}$ has the axioms
\begin{gather}
  \neg E(x,y,\ldots, t) \quad (\text{the tuple}\ (x,y,\ldots, t)\ \text{contains repeated variables}) \notag\\
  E(x_1,\ldots,x_k) \equiv E(x_{\sigma(1)},\ldots,x_{\sigma(k)}) \quad (\sigma\in S_k).\label{eq:symmetryaxiomhypergraph}
\end{gather}

Of course, the theory of graphs $T_{\operatorname{Graph}}$ is the same as $T_{2\operatorname{-Hypergraph}}$.
\end{example}

\begin{example}[colorings] \label{ex:colorings}
Assume that $c\in\mathbb{N}_+$ is a fixed constant and that the language $\mathcal L$ consists of $c$ unary predicates $\chi_0,\ldots,\chi_{c-1}$. The theory $T_{c\operatorname{-Coloring}}$ of vertex colorings in $c$ colors has the axioms
\begin{align*}
  \neg \chi_i(x) \lor \neg \chi_j(x) & \quad (0\leq i<j\leq c-1); &
  \chi_0(x)\lor\cdots \lor \chi_{c-1}(x). &
\end{align*}
Likewise, the theory $T_{c\operatorname{-ColoredGraph}}$ has $c$ {\em binary} predicates $E_0,E_1,\ldots,E_{c-1}$ in its language, and it has the axioms
\begin{align*}
  \neg E_i(x,x); & &
  E_i(x,y)\equiv E_i(y,x); & &
  \neg E_i(x,y) \lor \neg E_j(x,y) & \quad\! (0\leq i<j\leq c-1).
\end{align*}
The theory $T_{c\operatorname{-ColoredComplete}}$ is the theory of colorings of the edges of a {\em complete} graph in
$c$ colors. It is obtained from $T_{c\operatorname{-ColoredGraph}}$ by adding the axiom
$$
x\neq y \to (E_0(x,y)\lor\cdots\lor E_{c-1}(x,y)).
$$

Note that in our framework these definitions are only valid when the number of
colors $c$ is finite and known in advance. The language of coloring into an unbounded number of distinguishable colors is not even {\em locally} finite (see Remark~\ref{rem:locally_finite}), so some of our conclusions will hold for it, but most will fall apart. What we, however, {\em can} do is to mimic the intended coloring by the associated equivalence relation that still allows us to capture many properties we are interested in. See Example~\ref{ex:extra_axioms} below for more details.
\end{example}

\begin{example}[orders] \label{ex:orders}
The language of the theory of {\em partial orders} $T_{\operatorname{Order}}$ has
only one binary predicate that will be denoted by\footnote{For
technical reasons that will become transparent soon, it is more convenient to work with strict order.} $x\prec y$, and it has the axioms
\begin{align}
\neg (x & \prec x); \notag\\
x\prec y & \to \neg (y\prec x); \label{eq:asymmetry}\\
(x\prec y \land y\prec z) & \to x\prec z. \notag
\end{align}

Strengthening~\eqref{eq:asymmetry} to $x\neq y \to (x\prec y \equiv \neg
(y\prec x))$, we get the theory $T_{\operatorname{LinOrder}}$ of linear
orders. A unique feature of this theory is that it has only one model (up
to isomorphism) in any given finite cardinality $n$; in model-theoretical terms this means that it is {\em $n$-categorical}.

The theory of \emph{cyclic orders} $T_{\operatorname{CycOrder}}$ has only one ternary predicate $C$ and has the following axioms.
\begin{align*}
  \neg C(x,y,z) & \quad \text{if the tuple}\ (x,y,z)\ \text{contains repeated variables};\\
  C(x,y,z) & \to C(y,z,x);\\
  x\neq y\land x\neq z\land y\neq z & \to (C(x,y,z)\equiv \neg C(x,z,y));\\
  C(x,w,y)\land C(x,y,z) & \to C(x,w,z).
\end{align*}
\end{example}

\medskip
Let us now review two general constructions allowing us to obtain new theories
from already existing ones.

\begin{example}[extra axioms] \label{ex:extra_axioms}
For a theory $T$ in a language $\mathcal L$, we can always obtain a stronger
theory $T'$ in the same language by adding extra axioms. Viewed this way, our
Examples~\ref{ex:graphs} and~\ref{ex:orders} lead to Figure~\ref{fig:containments}. More
generally, given a finite
model $M$ of a theory $T$ with, say, vertex set $V(M)=\{v_1,v_2,\ldots,v_m\}$, its {\em
  open diagram} $\Dopen(M)$ is the conjunction of all formulas of the form
\begin{eqnarray}
\nonumber x_i\neq x_j && (i\neq j);\\
 \nonumber P(x_{i_1},\ldots,x_{i_k}) &\text{if}& (v_{i_1},\ldots,v_{i_k}) \in R_{P,M};\\
\label{eq:negative} \neg P(x_{i_1},\ldots,x_{i_k}) &\text{if}& (v_{i_1},\ldots,v_{i_k})
\notin R_{P,M};
\end{eqnarray}
 where $P\in\mathcal L$ runs over all $k$-ary predicate symbols in our language, $i_1,\ldots,i_k\in
[n]$, and $R_{P,M}$ is the interpretation of $P$ in $M$. Adding to $T$ the
axiom $\neg \Dopen(M)$ we get the theory obtained from $T$ by
forbidding {\em induced} submodels isomorphic to $M$. In many situations
(particularly when working with graphs) it is also natural to forbid
submodels $M$ that are {\em not} necessarily induced. In our logical setting
this is achieved simply by leaving out negated atomic formulas~\eqref{eq:negative}, which leads to the notion of the {\em positive open
diagram} $\PDopen(M)$.

\begin{figure}[ht]
  \begin{center}
    \begin{tikzpicture}
\def\shorten{0.4cm}

\begin{scope}[scale=1.5]
\coordinate (Digraph) at (0cm,0cm);
\coordinate (Graph) at (-1cm,-1cm);
\coordinate (Orgraph) at (1cm,-1cm);
\coordinate (Order) at (0cm,-2cm);
\coordinate (Tournament) at (2cm,-2cm);
\coordinate (LinOrder) at (1cm,-3cm);
\end{scope}

\foreach \name in {%
  Digraph,
  Graph,
  Orgraph,
  Order,
  Tournament,
  LinOrder%
}
\node at (\name) {$T_{\text{\name}}$};

\foreach \a/\b in {%
  Digraph/Graph,
  Digraph/Orgraph,
  Orgraph/Order,
  Orgraph/Tournament,
  Order/LinOrder,
  Tournament/LinOrder%
}
\draw[shorten <=\shorten, shorten >=\shorten] (\a) -- (\b);
\end{tikzpicture}

    \caption{Poset of theories presented in Examples~\ref{ex:graphs} and~\ref{ex:orders}. The theories that are further below are stronger in the sense that they can be obtained by adding more axioms.}
    \label{fig:containments}
  \end{center}
\end{figure}
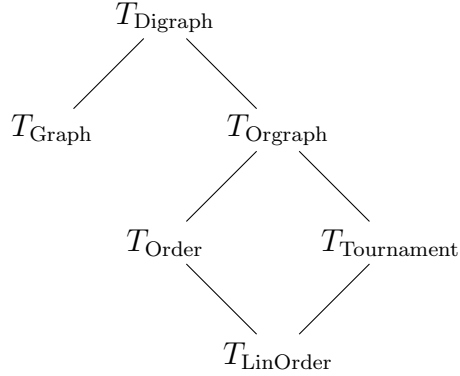

A host of natural examples of this sort is provided by the field of extremal
combinatorics and, in particular, so-called {\em Tur\'an density problems};
here we list only a few of them. $T_{\operatorname{TF-Graph}}$ is the theory of
triangle-free graphs, and forbidding in $T_{\operatorname{Graph}}$ induced copies of
$P_3$, a path on three vertices, we arrive at the theory $T_{\operatorname{EqRel}}$ of equivalence relations\footnote{Like in Example~\ref{ex:orders}, we
replace the reflexivity axiom by its negation and replace the transitivity axiom by
$(x\neq z\land x\prec y \land y\prec z\to x\prec z)$ (this may seem strange at first,
but for technical reasons, we want to keep the anti-reflexiveness).}. In graph-theoretic terms,
its models can be viewed as unions of vertex-disjoint cliques {\em without} any a priori bound on their number. This is practically the same as the theory of vertex colorings into an unspecified number of colors, cf.~Example~\ref{ex:colorings}. The theory
$T_{\operatorname{\text{Tur\'an}}}$ (named after~\cite{Tur}) is the extension of
$T_{3\operatorname{-Hypergraph}}$ with the axiom forbidding independent sets of size
four, and $T_{\operatorname{CH}}$ (named after~\cite{CaH}) is the extension of
$T_{\operatorname{Orgraph}}$ asserting that the oriented graph in question has girth at least
$4$, or, in other words, forbidding oriented cycles $\vec C_3$. The
theory $T_{\operatorname{FDF}}$ (named after~\cite{Fon}) is the extension of
$T_{\operatorname{Orgraph}}$ forbidding {\em induced} copies of $\vec C_4$.

We need not restrict ourselves to just one extra axiom, of course. Given a
theory $T$ and a set $\mathcal F$ of structures in its language, let
$\Forb{T}{\mathcal F}$ [$\Forbp{T}{\mathcal F}$] be the theory obtained from
$T$ by appending axioms $\neg\Dopen(M)$ [$\neg\PDopen(M)$, respectively] for
all $M\in \mathcal F$. It should be clear at this point that in principle
{\em every} theory can be obtained along these lines, namely by taking as its
axioms the formulas $\neg \Dopen(M)$ for all those structures $M$ that are
{\em not} its models. What is not clear, however, is whether the resulting
``brute-force'' description would be necessarily instructive. Also, there is
no easy way to tell in advance whether the theory in question is finitely
axiomatizable or even ``reasonably'' axiomatizable. For example, the theory
$T_{\operatorname{Bipartite}} =
\Forbp{T_{\operatorname{Graph}}}{\{C_{2\ell+1} \mid \ell\in\mathbb{N}_+\}}$
of bipartite graphs is obtained from the theory of graphs by simultaneously
forbidding all odd cycles (induced or not), and, by the same token, the
theory of directed {\em acyclic} graphs $T_{\operatorname{DAG}} =
\Forbp{T_{\operatorname{Digraph}}}{\{\vec C_\ell \mid \ell\geq 2\}}$ can be
obtained from $T_{\operatorname{Digraph}}$ by forbidding all finite directed
cycles: on Figure~\ref{fig:containments}, the latter theory would be located
between $T_{\operatorname{Order}}$ and $T_{\operatorname{Orgraph}}$. These
theories are not finitely axiomatizable. On the other hand, the theory
$T_{\operatorname{ThreshGraph}}$ of induced subgraphs of the graph
$\set{(v,w)\in\mathbb R^2}{v\neq w\land v+w>0}$, called {\em threshold
graphs}, turns out to be axiomatizable by adding to
$T_{\operatorname{Graph}}$ just one axiom:
\begin{align*}
  \bigl(E(x,y) & \land E(u,z)\bigr) \to
  \Bigl(\bigl(E(x,u)\land E(x,z)\bigr)\lor \bigl(E(y,u)\land E(y,z)\bigr)\\
  & \lor \bigl(E(u,x)\land E(u,y)\bigr)\lor \bigl(E(z,x)\land E(z,y)\bigr)\Bigr).
\end{align*}
\end{example}
The theory of limits of threshold graphs was studied in~\cite{DHJ08}; balanced finite models of this theory
are called half-graphs and they have recently found unexpected applications in model theory (see e.g.~\cite{MaSh}).

\begin{example}[mix-and-match] \label{ex:mix}
 For any two theories $T_1,T_2$ in disjoint languages $\mathcal L_1,\mathcal
L_2$ we can form their {\em disjoint union}\footnote{This construction will be
generalized in Section~\ref{sec:interpretations}.} $T_1\cup T_2$ in the language
$\mathcal L_1\mathbin{\stackrel\cdot\cup} \mathcal L_2$ by putting together the axioms of
$T_1$ and $T_2$. Let us see a few prominent examples.

For any theory $T$ and $c>0$, we denote by $T^c\df T\cup T_{c\operatorname{-Coloring}}$ the
theory of models of $T$ colored in $c$ distinguishable colors. As we noted, we
may not fully handle in our framework the case when the number of colors
is unspecified or infinite. However, the theory $T\cup T_{\operatorname{EqRel}}$
corresponds to models of $T$ that are vertex-colored in an unspecified number
of indistinguishable colors. Likewise, we let $T^<\df T\cup T_{\operatorname{LinOrder}}$
be the theory of {\em linearly ordered} models of $T$, which is
essentially the same as {\em labeled models}. The theory $T^{\operatorname{Cyc}}\df T\cup T_{\operatorname{CycOrder}}$
of \emph{cyclically ordered} models of $T$ is obtained similarly. These theories have recently gained
considerable attention~\cite{BKV,PaT,Tar}; we will return to them in Example~\ref{ex:ESS}.

The theories $T_{\operatorname{Graph}}^c$ and $T_{\operatorname{Graph}}^<$
have been (implicitly) instrumental in the study of quasi-random graphs since
the pioneering papers~\cite{Tho2,CGW}. Likewise, the theory
$T_{\operatorname{Tournament}}^<$ was very useful for the case of
quasi-random tournaments~\cite{ChGr2}; we will comment more on these
connections in the next section. A very interesting case is the theory
$T_{\operatorname{Perm}}\df T_{\operatorname{LinOrder}}\cup
T_{\operatorname{LinOrder}}$ of two linear orders $<_1$ and $<_2$ on the
same ground set. As one can see on Figure~\ref{fig:permutations}, its finite
models are in one-to-one natural correspondence with permutations of the set
$[n]$. It is this connection that (again, implicitly) underlines the theory
of {\em permutons}~\cite{HKM*} and makes an example of a combinatorial
structure that a priori does not fit our framework (remember that function
symbols are disallowed!) but can be made amenable to it after a small twist.
We will return to this in Section~\ref{sec:permutons}.
\end{example}

\begin{figure}[ht]
  \begin{center}
    \begin{tikzpicture}

\coordinate (left) at (0cm,0cm);
\coordinate (right) at (1cm,0cm);

\draw[arrows={latex-latex}] (left) -- (right);

\node[left] at (left) {%
  $\displaystyle\left(
  \begin{tabular}{*{4}{>{$\displaystyle}c<{$}}}
    1 & 2 & \cdots & n\\
    \sigma(1) & \sigma(2) & \cdots & \sigma(n)
  \end{tabular}
  \right)$%
};

\node[right] at (right) {%
  $\displaystyle
  \begin{dcases*}
    1 <_1 2 <_1 \cdots <_1 n &\\
    \sigma^{-1}(1) <_2 \sigma^{-1}(2) <_2 \cdots <_2 \sigma^{-1}(n)&
  \end{dcases*}
  $%
};

\end{tikzpicture}

    \caption{Correspondence between permutations and models of $T_{\operatorname{Perm}}$.}
    \label{fig:permutations}
  \end{center}
\end{figure}

All theories we have seen so far share the property that their predicates $P$ are always false on any tuple $v_1,\ldots,v_k$ containing repeated vertices. In Section~\ref{sec:interpretations} we will see why this property can (and will) be assumed without loss of generality, and after that we will see in Section~\ref{sec:densities} why it is very useful. Right now we just make a definition.

\begin{definition} \label{def:canonical}
A theory $T$ in a language $\mathcal L$ is {\em canonical} if for every $P\in\mathcal L$ of arity $k$ and for every
$1\leq i<j\leq k$, the theory $T$ entails\footnote{Recall from Section~\ref{sec:model_theory} that our assumptions allow
us not to distinguish between provability in the first-order logic and validity on finite models.} the formula
\begin{equation} \label{eq:canonical}
x_i=x_j\to \neg P(x_1,\ldots,x_k).
\end{equation}
\end{definition}

We finish this section with two examples of non-canonical theories.

\begin{example}\label{ex:edgeorderedgraphs}
  The language of the theory of \emph{edge ordered graphs} $T_{\operatorname{EdgeOrderedGraph}}$ consists of a binary predicate symbol $E$ encoding adjacency and a quaternary predicate symbol $P$ encoding the edge order, and it has the following axioms (since we are not aiming at canonicity here, we axiomatize non-strict order of the edges).
  \begin{align}
 \nonumber   \neg E(x,x); & \\
 \nonumber   E(x,y) & \to E(y,x);\\
  \nonumber  P(x_1,y_1,x_2,y_2)\land P(x_2,y_2,x_1,y_1) & \to (x_1 = x_2 \land y_1 = y_2)\lor (x_1 = y_2\land y_1 = x_2);\\
   \label{eq:P_transitivity} P(x_1,y_1,x_2,y_2)\land P(x_2,y_2,x_3,y_3) & \to P(x_1,y_1,x_3,y_3);\\
  \nonumber  E(x_1,y_1)\land E(x_2,y_2) & \equiv P(x_1,y_1,x_2,y_2)\lor P(x_2,y_2,x_1,y_1);\\
  \label{eq:P_symmetry}  P(x_1,y_1,x_2,y_2) & \to P(y_1,x_1,x_2,y_2)\land P(x_1,y_1,y_2,x_2).
  \end{align}
  Extremal problems for this theory have also received attention in the recent years (see~\cite{Tar}).
\end{example}

\begin{example} \label{ex:greybow}
The theories $T_{c\operatorname{-ColoredGraph}}$ and $T_{c\operatorname{-ColoredComplete}}$ (see Example~\ref{ex:colorings}) are
sufficiently popular in extremal combinatorics, but they are often redundant and, as a consequence, bulky. For
example, the theory $T_{c\operatorname{-ColoredComplete}}$ has 25506 models
on 6 vertices that makes it prohibitive for flag-algebraic calculations. To study rainbow-type problems for such structures,~\cite{BHL*} circumvented this drawback by considering {\em color-blind} isomorphisms,
i.e., those that are also allowed to permute colors.

In our language, the corresponding theory $T_{c\operatorname{-Greybow}}$
is given by an equivalence relation $P(x_1,y_1,x_2,y_2)$ with at most $c$ classes on the edges of a complete (for simplicity) graph. It has the axioms
\begin{align*}
P(x,y,x,y); &\\
P(x_1,y_1,x_2,y_2) & \equiv P(x_2,y_2,x_1,y_1); \\
\text{\eqref{eq:P_transitivity}, \eqref{eq:P_symmetry}}; &\\
\bigwedge_{0\leq i\leq c} x_i\neq y_i &\to \bigvee_{0\leq i<j\leq c} P(x_i,y_i,x_j,y_j).
\end{align*}

The theory $T_{\operatorname{Tournament}}\cup T_{c\operatorname{-ColoredComplete}}$ was the main tool in the recent solution~\cite{ErdosHajnal} of one of the Erd\H{o}s--Hajnal problems.
\end{example}

\subsection{Interpretations} \label{sec:interpretations}

Loosely speaking, interpretations allow us to define structures of one type from structures of another type. In mathematical
logic, this general paradigm is usually specialized by the concept of {\em first-order interpretations}, but given our
restrictions on syntax, we must go one step further down and, like in~\cite[\S 2.3]{flag}, consider only {\em open}
interpretations. The definition in~\cite{flag} aimed to embrace several different situations at once and, as a result, it was
a bit heavy and technical. In this paper we only give its lighter version called in~\cite[\S 2.3.3]{flag} ``global interpretations''.

\begin{definition}
Let $\mathcal L_1$ and $\mathcal L_2$ be finite languages containing only predicate symbols. A {\em translation} of
$\mathcal L_1$ into $\mathcal L_2$ is a mapping $I$ that takes every predicate symbol $P(x_1,\ldots,x_k)\in\mathcal L_1$ to an open formula $I(P)(x_1,\ldots,x_k)$ in the language $\mathcal L_2$ with the same variables. The translation $I$ is extended to open formulas of the language $\mathcal L_1$ in an obvious way, by declaring that it commutes with logical connectives. Let $T_1$ and $T_2$ be (as usual, universal) theories in the languages $\mathcal L_1$ and $\mathcal L_2$, respectively. The translation $I$ is an {\em open interpretation} of $T_1$ in $T_2$, denoted $I\interpret{T_1}{T_2}$, if for every axiom $\forall \vec x A(x_1,\ldots,x_n)$ of the theory $T_1$, we have $T_2\vdash \forall \vec x I(A)(x_1,\ldots,x_n)$.
\end{definition}

Before giving concrete examples, let us do a bit of abstract nonsense.

Theories and open interpretations make a category that becomes particularly natural if we identify ``indistinguishable'' interpretations. Namely, let us call two interpretations $I_1\interpret{T_1}{T_2}$ and $I_2\interpret{T_1}{T_2}$ {\em equivalent} if for any $P\in\mathcal L_1$ of arity $k$, we have $T_2\vdash \forall\vec x(I_1(P)(x_1,\ldots,x_k)\equiv I_2(P)(x_1,\ldots,x_k))$. This is clearly an equivalence relation, so we let {\sc Int} denote the corresponding factor-category. Two theories $T_1$ and $T_2$ are {\em isomorphic} if they are isomorphic in the category {\sc Int} or, in other words, if there exist open interpretations $I_1\interpret{T_1}{T_2}$ and $I_2\interpret{T_2}{T_1}$ such that both $I_2I_1$ and $I_1I_2$ are equivalent to the identity interpretations of $T_1$ and $T_2$ respectively.

Given an open interpretation $I\interpret{T_1}{T_2}$ and a model $M$ of $T_2$, we can naturally define a model $I(M)$ of $T_1$ with the same set of vertices. This gives a contravariant functor from {\sc Int} to the category of sets ($I(T)$ being the set of all finite models of $T$ up to isomorphism), and we will see several more natural and quite useful functors from {\sc Int} in the forthcoming sections. Collectively, these observations strongly suggest that open interpretations allow us to transfer a great deal of useful structure from one situation to another. In particular, {\em isomorphic theories are indistinguishable for all practical purposes}.

Since we mostly regard open interpretations as a handy tool, we did not attempt a serious study of the structure of the category {\sc Int} itself. There is, however, one property that we would like to highlight, namely that it allows {\em amalgamated sums} (otherwise known as pushouts, fibred coproducts, etc.) In other words, for every two open interpretations $I_1\interpret{T}{T_1}$ and $I_2\interpret{T}{T_2}$ there exist another theory $\widehat T$ and open interpretations $\widehat I_1$ and $\widehat I_2$ such that the diagram
\begin{equation} \label{eq:amalgamated}
  \begin{tikzcd}
    T\arrow[r, "I_1"]\arrow[d, "I_2", swap] & T_1\arrow[d, "\widehat I_2"]\\
    T_2\arrow[r, "\widehat I_1"] & \widehat T
  \end{tikzcd}
\end{equation}
is commutative and has the standard universality property. As usual, the latter implies the uniqueness of amalgamated sums up to isomorphism provided they exist, and we prove their existence as follows. By renaming predicate symbols if necessary we can assume that the languages $\mathcal L_1$ and $\mathcal L_2$ of the theories $T_1$ and $T_2$ are disjoint. The required theory $\widehat T$ is the theory $T_1\cup T_2$ with added axioms $\forall \vec x(I_1(P)(x_1,\ldots,x_n)\equiv I_2(P)(x_1,\ldots,x_n))$, one for every predicate symbol $P$ of the language of $T$. Checking that the diagram~\eqref{eq:amalgamated} is commutative is straightforward (recall that in {\sc Int} we identify equivalent interpretations!), and equally straightforward is the universality property.
\begin{example}[extra axioms, cntd.]
If a theory $T'$ is obtained from a theory $T$ in the same language by adding extra axioms, then the identity translation is an interpretation of $T$ in $T'$ called the \emph{axiom-adding interpretation}. If $I\interpret{T}{T_1}$ and $I\interpret{T}{T_2}$ are two interpretations of this sort, then their amalgamated sum is simply obtained by simultaneously adding to $T$ both sets of axioms. For example, the square on Figure~\ref{fig:containments} is an amalgamated sum.
\end{example}

\begin{example}[mix-and-match, cntd.] If $T_0$ is the empty theory in the empty language (that is, the initial object of {\sc Int}) then the amalgamated sum of trivial interpretations $T_0\leadsto T_1$ and $T_0\leadsto T_2$ is simply the disjoint union $T_1\cup T_2$. Interestingly enough, sometimes the theories $T_1\cup T_2$ and $T_1'\cup T_2$ may turn out to be isomorphic even if $T_1$ and $T_1'$ are not. For example, the theories $T_{\operatorname{Graph}}^<$ and $T_{\operatorname{Tournament}}^<$ are isomorphic: the ``feedback arc'' interpretation $I\interpret{T_{\operatorname{Graph}}^<}{T_{\operatorname{Tournament}}^<}$ translates the order $<$ by itself, and translates the edge predicate $E$ as $I(E)(x,y) \df (x < y \land E(y,x)) \lor (y < x\land E(x,y))$; it is easy to see that it is invertible. It is this isomorphism that was (implicitly) used in~\cite{ChGr2} for reducing questions about quasi-random tournaments to questions about quasi-random graphs. On the other hand, the theory $T$ obtained from $T_{\operatorname{Graph}}^2$ by additionally requiring that vertices of the same color are non-adjacent, is {\em not} isomorphic to $T_{\operatorname{Bipartite}}$. The interpretation $T_{\operatorname{Bipartite}}\leadsto T$ is trivial, but, for several good reasons, it does not have an inverse. As a consequence, there are certain results obtained via flag algebras for which one has to use 2-{\em colored} (as opposed to 2-{\em colorable}) graphs. For isomorphic theories this could have never happened.
\end{example}

\begin{example}[structure-erasing interpretations]\label{ex:strucerase} All interpretations of the form $T_1\leadsto T_1\cup T_2$ that act as identity on $T_1$ can be viewed as {\em structure-erasing} in the sense that they take a model of the theory $T_1\cup T_2$ and erase from it all information about predicate symbols from $\mathcal L_2$. Two important examples are the {\em color-erasing} interpretation $T\leadsto T^c$ and the {\em order-erasing} interpretation $T\leadsto T^<$. But interpretations of this nature may also arise in other situations. For example, the {\em orientation-erasing interpretation} $I\interpret{T_{\operatorname{Graph}}}{T_{\operatorname{Orgraph}}}$ is given by $I(E)(x,y)\df (E(x,y)\lor E(y,x))$. Another (edge) color-erasing interpretation $I\interpret{T_{\operatorname{Graph}}}{T_{c\text{-ColoredGraph}}}$ is given by $I(E)(x,y)\df E_0(x,y)\lor\cdots\lor E_{c-1}(x,y)$. Finally, the edge order-erasing interpretation $I\interpret{T_{\operatorname{Graph}}}{T_{\operatorname{EdgeOrderedGraph}}}$ is given by $I(E) = E$ (erasing all information on the order $P$ of edges).
\end{example}

\begin{remark}\label{rmk:strucerase}
  Up to isomorphism, every open interpretation can be seen as the composition of a structure-erasing interpretation and an axiom-adding interpretation. More specifically, given an open interpretation $I\interpret{T_1}{T_2}$, let $T$ be the theory obtained from $T_1\cup T_2$ by adding the axioms $P(\vec x) \equiv I(P)(\vec x)$ for every predicate symbol $P$ of $T_1$. It is easy to see then that the interpretation $J\interpret{T_2}{T}$ that acts identically on the predicate symbols of $T_2$ is an isomorphism (its inverse $J^{-1}$ acts identically on $T_2$ and acts as $I$ on $T_1$) and for the structure-erasing interpretation $S\interpret{T_1}{T_1\cup T_2}$ and the axiom-adding interpretation $A\interpret{T_1\cup T_2}{T}$, the diagram
  \begin{equation*}
    \begin{tikzcd}
      T_1\arrow[r, "I"]\arrow[d, "S", swap] & T_2\arrow[d,"J"]\\
      T_1\cup T_2\arrow[r, "A"] & T
    \end{tikzcd}
  \end{equation*}
  commutes.
\end{remark}

\begin{example}[unusual $k$-graphs]\label{ex:triangleinterpret} We may consider interpretations like
$I\interpret{T_{3\operatorname{-Hypergraph}}}{T_{\operatorname{Graph}}}$ given by
$I(E)(x,y,z)\df (E(x,y)\land E(x,z)\land E(y,z))$, i.e., we declare a triple of vertices to be a hyperedge iff it is a triangle
in the original (ordinary) graph. Interpretations of this sort, i.e., when we define higher-dimensional objects in terms of
low-dimensional ones, are the principal source of examples illustrating why fundamental results about graphs (and graphons)
cannot be always directly generalized to higher-order structures, see Examples~\ref{ex:HypergraphBadLimitConstant}
and~\ref{ex:RandomHypergraphons} and an excellent exposition in~\cite{Gow2}.

More generally, for any theory $T$ and any symmetric open formula $F(x_1,\ldots,x_k)$, there is a natural interpretation
$T_{k\operatorname{-Hypergraph}}\leadsto T$. This interpretation, with $T = T_{\operatorname{Tournament}}\cup T_{{k\choose 2}-\text{ColoredComplete}}$
was another major tool in solving the Erd\H{o}s--Hajnal problem mentioned in Example~\ref{ex:greybow}.
\end{example}

\begin{example}[Tur\'an's $(3,4)$-problem] Recall (see Example~\ref{ex:extra_axioms}) that $T_{\operatorname{\text{Tur\'an}}}$ is the extension of $T_{3\operatorname{-Hypergraph}}$ forbidding independent sets on four vertices.
Determining (even asymptotically) the minimum edge density of its models, often called {\em Tur\'an's $(3,4)$-problem}, is an outstanding open problem (see e.g.\ the survey~\cite{Kee}), and it is believed that perhaps a major
source of its difficulty is that the set of conjectured extremal examples in this case is extremely complex. Using the language of interpretations, we can at least conveniently highlight the internal structure of this set; the material below is borrowed from~\cite{Fon,fdf,Kee}.

Recall that $T_{\operatorname{FDF}}$ is the theory $T_{\operatorname{Orgraph}}$ augmented with the axiom forbidding induced copies of $\vec C_4$. The {\em Fon-der-Flaass interpretation} $\text{FDF}\interpret{T_{\operatorname{\text{Tur\'an}}}}{T_{\operatorname{FDF}}}$ is given by
\begin{align*}
  & \operatorname{FDF}(E)(x_0,x_1,x_2) \df \bigwedge_{a\neq b\in {\mathbb Z}_3}(x_a\neq x_b) \land \left(\bigvee_{a\in {\mathbb Z}_3}(E(x_a,x_{a+1})\land E(x_a,x_{a-1}))\right.
  \\ & \qquad
  \lor \left.\bigvee _{a\in {\mathbb Z}_3}(\neg E(x_a,x_{a+1})\land \neg E(x_a,x_{a-1})\land \neg E(x_{a-1},x_a)\land \neg E(x_{a+1},x_a))\right).
\end{align*}
In plain English (originally Russian), we declare a triple of vertices to form a 3-edge if and only if in the oriented graph spanned by these vertices we either have an isolated vertex or a vertex of out-degree 2.

We can further interpret $T_{\operatorname{FDF}}$ in $T_{\operatorname{ThreshGraph}}^3$ as follows:
$$
I(E)(x,y) \df \bigvee_{a\in \mathbb Z_3}(\chi_a(x)\land \chi_{a-1}(y)\land \neg E(x,y)) \lor \bigvee_{a\in \mathbb Z_3}(\chi_a(x)\land \chi_{a+1}(y)\land E(x,y)).
$$
It is routine to check that these two translations are indeed interpretations of respective theories, and it turns out that the set of conjectured extremal examples for Tur\'an's $(3,4)$-problem ``asymptotically coincides''\footnote{A precise meaning of this term will become clear soon.}, via the consecutive application of these two interpretations, with those models of $T_{\operatorname{ThreshGraph}}^3$ in which the 3-coloring is balanced and ``independent'' from the threshold graph.
\end{example}

The toolkit of useful interpretations can be substantially expanded if we additionally allow fixed vertices or restrictions of the domain, but, as we said before, we prefer to keep our exposition lighter. Instead, let us show that the restriction of canonicity (Definition~\ref{def:canonical}) is not very restrictive by proving that every theory can be ``subdivided'' into a canonical theory; cf.~a similar argument in~\cite[\S 7.1]{Kall}.

\begin{theorem} \label{thm:canonical}
  For any universal theory $T$ in a first-order language $\mathcal L$ containing only predicate symbols there exists a canonical theory isomorphic to it.
\end{theorem}

\begin{proof}(sketch) For any $P(x_1,\ldots,x_k)\in\mathcal L$ and any equivalence relation $\approx$ on $[k]$ we introduce a new predicate symbol $P_\approx$ of arity that is equal to the number of equivalence classes in $\approx$. Let $\mathcal L'$ be the language consisting of all these symbols, and let $D_\approx(x_1,\ldots,x_k)$ be the formula $\bigwedge_{i\approx j}(x_i=x_j) \land \bigwedge_{i\not\approx j} (x_i\neq x_j)$. We define the translation $I$ of the language $\mathcal L$ in $\mathcal L'$ as follows:
$$
I(P)(x_1,\ldots,x_k) \equiv \bigvee_\approx (D_\approx(x_1,\ldots,x_k) \land P_\approx(x_{i_1},\ldots,x_{i_\ell}))
$$
(which is also equivalent to $\bigwedge_\approx (D_\approx(x_1,\ldots,x_k) \to P_\approx(x_{i_1},\ldots,x_{i_\ell}))$), where $i_\nu$ is an arbitrary representative in the $\nu$th class of the relation $\approx$; we assume that those are enumerated in an arbitrary but fixed order. We let $T'$ consist of all canonicity axioms~\eqref{eq:canonical}, along with all formulas of the form $\forall\vec x I(A)(x_1,\ldots,x_n)$, where $A(x_1,\ldots,x_n)$ is an axiom of $T$. Then $I$ is automatically an interpretation of $T$ in $T'$.

In the opposite direction, we translate the predicate symbols $P_\approx$ as follows:
$$
J(P_\approx)(x_1,\ldots,x_\ell) \equiv \bigwedge_{1\leq\nu<\mu\leq \ell}(x_\nu\neq x_\mu) \land P(x_{\nu_1},\ldots,x_{\nu_k}),
$$
where $\nu_i$ is the equivalence class of $i$. It is straightforward to check that $J$ is an interpretation of $T'$ in $T$, and that $I$ and $J$ are inverse to each other.
\end{proof}

In categorical terms, the theorem above says that {\sc Int} is equivalent to its subcategory made by canonical universal theories.

\begin{example} Consider the non-canonical theory obtained from $T_{\operatorname{Graph}}$ by dropping the axiom $\neg E(x,x)$, i.e., let us allow loops. It is isomorphic to the canonical theory $T_{\operatorname{Graph}}^2$ (cf.~Example~\ref{ex:mix}) in which we use the additional unary predicates to distinguish between those vertices that have a loop on them and those that do not.
\end{example}

\begin{example}\label{ex:edgeorderedgraphscanonical}
  Applying Theorem~\ref{thm:canonical} above to $T_{\operatorname{EdgeOrderedGraph}}$ of Example~\ref{ex:edgeorderedgraphs} gives a theory isomorphic to it with a total of $17$ predicate symbols. However, it is straightforward to get a canonical theory isomorphic to $T_{\operatorname{EdgeOrderedGraph}}$ with only three predicate symbols $E,P',P''$ whose translations are the following.
  \begin{align*}
    E(x,y) & \leadsto E(x,y);\\
    P'(x,y,z) & \leadsto
    P(x,y,x,z)\land y\neq z;\\
    P''(x_1,y_1,x_2,y_2) & \leadsto
    P(x_1,y_1,x_2,y_2)\land x_1\neq x_2\land x_1\neq y_2\land y_1\neq x_2\land y_1\neq y_2.
  \end{align*}
  A similar ``compactification'' can be also done to the theory $T_{c\operatorname{-Greybow}}$ from Example~\ref{ex:greybow}.
\end{example}

{\em From now on all theories will be assumed to be canonical unless mentioned otherwise}.

\subsection{Densities} \label{sec:densities}

What we have done so far amounts to some very basic facts about a rather restricted fragment of first-order logic and model theory. Before we completely switch gears, let us remark that we strongly believe there should be more
connections between the classical model theory and its, as it were, measure-oriented version this work is contributing to. One very good indication of this are the works~\cite{ElS,ArC}  that use ultrafilters in much the same way they are used in model theory and non-standard analysis. Another relevant topic is that of {\em finitely forcible graphons}~\cite{LoSz2} that is a clear analogy of finite axiomatizability in the first order logic. But, by far and large, at the moment this potential seems to be largely unexplored.

\medskip
In any case, in the absence of quantifiers, our basic primitive is {\em counting}, and we begin with introducing the necessary notation in the finite setting.

Let $M$ and $N$ be two models of the same (universal) theory $T$ with $m=\lvert V(M)\rvert\leq\lvert V(N)\rvert = n$. How do we count the ``density'' or ``frequencies'' with which $M$ occurs in $N$? The approach that turns out to be the most robust and context-independent is to simply count the number of different submodels, normally referred to in combinatorics as {\em induced substructures}, normalized by $\binom{n}{m}$. In other words, let $p(M,N)$ be the probability of the event that $N\vert_{\rn V}$ is {\em isomorphic} to $M$ (denoted $N\vert_{\rn{V}}\cong M$), where $\rn V$ is an $m$-element subset of $V(N)$ chosen uniformly at random. This definition fully accounts for symmetries existing in the model $M$, and for these reasons it is the one used in flag algebras where frugality is paramount. When the latter is less of an issue, it is often more convenient to count instead (induced) embeddings as follows. Pick uniformly at random an {\em injective} mapping $\rn{\alpha}\injection{V(M)}{V(N)}$ (there are $(n)_m$ of them), and define $\tind(M,N)$ to be the probability that $\rn\alpha$ is an {\em induced embedding} of $M$ into $N$. The latter condition means that for any $k$-ary symbol $P$ and every tuple of distinct\footnote{Remember that $T$ is canonical.} vertices $v_1,\ldots,v_k\in V(M)$, $R_{P,M}(v_1,\ldots,v_k)$ if and only if $R_{P,N}(f(v_1),\ldots,f(v_k))$.

Another way to interpret $\tind(M,N)$ is by assuming that the vertices of $M$ are identified (in an arbitrary way) with integers from $[m]$, and then this is exactly the density of {\em labeled} submodels of $N$ that are
identical to $M$. Let us note in the brackets that although labeled models are the same objects as types in flag algebras (and partially labeled models correspond to flags), they are used here for rather different purposes.
For this reason we will avoid the word ``type'', and will denote labeled models by letters like $L$ or $K$, in order to distinguish them from unlabeled ones. We will use these two kinds of models interchangeably, based upon the following obvious identity:
\begin{align}\label{eq:ind-p}
  \tind(M,N) & = \frac{\lvert\Aut(M)\rvert}{m!}p(M,N) = \frac{p(M,N)}{(S_m: \Aut(M))},
\end{align}
where $m\df\lvert V(M)\rvert$ and $\Aut(M)$ is the group of automorphisms of $M$.

Yet another way of interpreting the quantity $\tind(M,N)$ is as a normalized counting of how many assignments of variables to distinct vertices of $N$ satisfy the open diagram $\Dopen(M)$ of $M$. Another useful parameter is obtained by instead counting the assignments that satisfy the positive open diagram $\PDopen(M)$ as follows. Pick uniformly at random an {\em injective} mapping $\rn{\alpha}\injection{V(M)}{V(N)}$ and define $\tinj(M,N)$ to be the probability that $\rn\alpha$ is a {\em positive embedding} of $M$ into $N$. The latter condition means that for any $k$-ary symbol $P$ and every tuple of distinct vertices $v_1,\ldots,v_k\in V(M)$, if $R_{P,M}(v_1,\ldots,v_k)$ then $R_{P,N}(f(v_1),\ldots,f(v_k))$.

It is easy to recover $\tinj$ from $\tind$ via the following identity:
\begin{align}\label{eq:inj-ind}
  \tinj(M,N) & = \sum_{M'\supseteq M}\tind(M',N),
\end{align}
where the sum is over all models $M'$ of $T$ with $V(M')=V(M)$, and $M'\supseteq M$ means that $R_{M',P}\supseteq R_{M,P}$ for any
$P\in\mathcal L$ or, equivalently, that $M'$ satisfies the positive open diagram $\PDopen(M)$ of $M$.

Note that we can apply M\"{o}bius Inversion to~\eqref{eq:inj-ind}, and get a formula expressing $\tind(M,N)$ as a (finite) linear combination of $(\tinj(M',N))_{M'\supseteq M}$. But since positive embeddings play very little role in our exposition, we defer details to Appendix~\ref{sec:Mobius}.

Densities also behave well with respect to open interpretations: if $I\interpret{T_1}{T_2}$ is such an interpretation and $N$ is a model of $T_2$, then the densities $p(\place, I(N))$ can be expressed as linear combinations of densities $p(\place, N)$. We will give more details in Section~\ref{sec:flag}, in the context where these combinations allow quite a natural interpretation.

It will also be convenient for us to let
\begin{align*}
  \tinj(M,N) \df \tind(M,N) \df p(M,N) & \df 0.
\end{align*}
whenever $\lvert V(M)\rvert > \lvert V(N)\rvert$.

As a final remark before we go into examples, note that all these densities are invariant under isomorphisms, that is, if $M\cong M'$ and $N\cong N'$, then
\begin{align*}
  p(M,N) & \!=\! p(M',N'); &
  \tind(M,N) & \!=\! \tind(M',N'); &
  \tinj(M,N) & \!=\! \tinj(M',N').
\end{align*}

\begin{example}[graphs]
   We denote by $K_\ell$ and $P_\ell$ the complete graph and the (undirected) path on $\ell$ vertices, respectively. We let $\overline{G}$ be the complement of the graph $G$, that is
$V(G)=V(\overline G)$ and the edges of $\overline{G}$ are non-edges of $G$ and vice versa. Finally, we denote by $I_\ell\df \overline K_\ell$ the empty graph on $\ell$ vertices.

  With this notation, the edge density of a graph $G$ is given by $p(K_2,G)$, and its triangle density is given by $p(K_3,G)$. In fact, for complete graphs and empty graphs, we have
  \begin{align*}
    p(K_\ell,G) & = \tind(K_\ell,G) = \tinj(K_\ell,G);\\
    p(I_\ell,G) & = \tind(I_\ell,G);\\
    \tinj(I_\ell,G) & = 1;
  \end{align*}
  for all graphs $G$ with at least $\ell$ vertices.

  For less trivial examples, for every $\ell\geq 3$ we have
  \begin{align*}
    p(K_2,P_\ell) = \tind(K_2,P_\ell) = \tinj(K_2,P_\ell) & = \frac{2}{\ell};
    \\
    p(P_3,P_\ell) & = \frac{6}{\ell(\ell-1)};
    \\
    \tind(P_3,P_\ell) = \tinj(P_3,P_\ell) & = \frac{2}{\ell(\ell-1)};
    \\
    p(P_3,K_\ell) = \tind(P_3,K_\ell) & = 0;
    \\
    \tinj(P_3,K_\ell) & = 1.
  \end{align*}

  The complementation operation behaves very well with respect to the densities $p$ and $\tind$, as it satisfies $p(H,G) = p(\overline{H},\overline{G})$ and $\tind(H,G) = \tind(\overline{H},\overline{G})$
  for all graphs $H$ and $G$.
The same cannot be said about $\tinj$ as we e.g.\ have $\tinj(I_\ell,G) = 1$ for every graph $G$ with at least $\ell$ vertices
while $\tinj(K_\ell,G) = 1$ if and only if $G$ is a complete graph. This inherent asymmetry (that comes up quite naturally in many applications of ordinary graphs) is one of the primary reasons why in the general case we prefer to work with induced densities.
\end{example}

\begin{example}[tournaments]
  In the theory of tournaments $T_{\operatorname{Tournament}}$ induced and non-induced embeddings are clearly the same, and we have $\tind(M,N) = \tinj(M,N)$. Let us do a few concrete calculations. Let $\Tr_\ell$ denote the transitive tournament on $\ell$ vertices (i.e., the only model of $T_{\operatorname{LinOrder}}$), let $\vec C_3$ denote the $3$-cycle (i.e., the only non-transitive tournament on $3$ vertices) and let $W_4$ and $L_4$ denote the uniquely defined tournaments on $4$-vertices whose outdegree sequences are $(3,1,1,1)$ and $(2,2,2,0)$ respectively. Then we have
  \begin{align*}
    p(\Tr_3,W_4) = p(\Tr_3,L_4) & = \frac{3}{4};
    \\
    \tind(\Tr_3,W_4) = \tind(\Tr_3,L_4) & = \frac{1}{8};
    \\
    p(\vec C_3,W_4) = p(\vec C_3,L_4) & = \frac{1}{4};
    \\
    \tind(\vec C_3,W_4) = \tind(\vec C_3,L_4) & = \frac{1}{8}.
  \end{align*}
\end{example}

\begin{example}[permutations]
  Recall that the theory of permutations is defined in our language as $T_{\operatorname{Perm}} = T_{\operatorname{LinOrder}}\cup T_{\operatorname{LinOrder}}$. Identifying a permutation $\sigma\function{[n]}{[n]}$ with the list of its values $(\sigma(1)\sigma(2)\cdots\sigma(n))$, we have
  \begin{align*}
    p(123,14235) & = \frac{5}{10} = \frac{1}{2};\\
    p(132,14235) & = \frac{1}{5};\\
    p(213,14235) & = \frac{1}{5};\\
    p(231,14235) & = 0;\\
    p(312,14235) & = \frac{1}{10};\\
    p(321,14235) & = 0.
  \end{align*}
\end{example}

\subsection{Convergent sequences} \label{sec:convergence}

As we mentioned in the introduction, there are two kinds of approaches to studying large, and eventually infinite, models of a theory:
semantic and syntactic. We begin with the
``neutral'' setting from which one can easily explore in either direction.

\medskip

The reader may have noticed that we used the term ``limit object'' in the introduction without specifying convergence; it is our first order of business now.

\begin{definition}
  Let $T$ be a (canonical) theory in the language $\mathcal{L}$. Let us denote by $\mathcal{M}_n[T]$ the set of all (unlabeled) finite models of $T$ up to isomorphism on $n$ vertices and let $\mathcal{M}[T] \df \bigcup_{n\in\mathbb{N}}\mathcal{M}_n[T]$ be the set of all finite models of $T$ up to isomorphism. Whenever $T$ is clear from context, we will drop $[T]$ from the notation.

  A sequence of models $(N_n)_{n\in\mathbb{N}}$ of $T$ is called \emph{increasing} if $\lvert V(N_n)\rvert < \lvert V(N_{n+1})\rvert$ for every $n\in\mathbb{N}$.

  The theory $T$ is called \emph{non-degenerate} if it has an increasing sequence of models, or, equivalently, if it has an infinite model.

  Let $d$ be one of $p$, $\tind$ or $\tinj$. An increasing sequence of models $(N_n)_{n\in\mathbb{N}}$ is called \emph{convergent} if $\lim_{n\to\infty}d(M,N_n)$ exists for every fixed model $M$ of $T$.
\end{definition}

A priori, we have three notions of convergence, but the proposition below says that they are equivalent.

\begin{proposition}
  If $(N_n)_{n\in\mathbb{N}}$ is an increasing sequence of models of a (canonical) theory, then the following are equivalent.
  \begin{itemize}
  \item The limit $\lim_{n\to\infty}p(M,N_n)$ exists for every fixed model $M$ of $T$;
  \item The limit $\lim_{n\to\infty}\tind(M,N_n)$ exists for every fixed model $M$ of $T$;
  \item The limit $\lim_{n\to\infty}\tinj(M,N_n)$ exists for every fixed model $M$ of $T$.
  \end{itemize}
\end{proposition}

\begin{proof}
  Follows from the fact that $p(M,N_n)$ and $\tind(M,N_n)$ differ only by a multiplicative constant and that $\tind(M,N_n)$ can be written as a (finite) linear combination in terms of $(\tinj(M',N_n))_{M'\in\mathcal{M}}$ and vice-versa.
\end{proof}

By the same token, convergence behaves well with respect to open interpretations: if $I\interpret{T_1}{T_2}$ is such an interpretation and $(N_n)_{n\in\mathbb{N}}$ is a convergent sequence of $T_2$-models then $(I(N_n))_{n\in\mathbb{N}}$ is a convergent sequence of models of the theory $T_1$.

If $d$ is one of $p$, $\tind$ or $\tinj$, any model $N\in\mathcal{M}$ gives rise to a functional
\begin{align*}
  \begin{functiondef}
    d(\place,N)\colon & \mathcal{M} & \longrightarrow & [0,1]\\
    & M & \longmapsto & d(M,N),
  \end{functiondef}
\end{align*}
which in turn can be seen as an element of $[0,1]^{\mathcal{M}}$. Now the definition of a convergent sequence of models is
simply a sequence of models that is convergent as elements of $[0,1]^{\mathcal{M}}$ in the usual product topology, in which we
require $\lim_{n\to\infty} d(M,N_n)$ to exist for every fixed $M$. No uniformity conditions or assumptions on the rate of
convergence are imposed.

Note that since $\mathcal{M}$ is countable, the space $[0,1]^{\mathcal{M}}$ is metrizable. One possible metric is
\begin{equation} \label{eq:distance}
  \operatorname{dist}((x_M)_{M\in\mathcal{M}}, (y_M)_{M\in\mathcal{M}})  = \sum_{m\in\mathbb{N}}\frac{\lvert x_{M_m} - y_{M_m}\rvert}{2^m},
\end{equation}
for a fixed ordering $(M_m)_{m\in\mathbb{N}}$ of $\mathcal{M}$. However, since this metric is rather arbitrary, it is rarely
used directly (the important property is that the space is metrizable {\em somehow}, cf.~\cite[\S 3.2]{flag}).

\begin{proposition}\label{prop:compact}
  Every increasing sequence of models of a theory has a convergent subsequence.
\end{proposition}

\begin{proof}
  Follows from the fact that $[0,1]^{\mathcal{M}}$ is compact which, in turn, follows
from Tychonoff's Theorem.
\end{proof}

\begin{example}[sequences of sparse hypergraphs] \label{ex:empty}
  In the theory of $k$-uniform hypergraphs $T_{k\operatorname{-Hypergraph}}$, the sequence of empty hypergraphs (i.e., hypergraphs without any edges) of increasing sizes $(I^{(k)}_n)_{n\in\mathbb{N}}$ is convergent, since
  \begin{align*}
    \lim_{n\to\infty}p(H,I^{(k)}_n)
    & =
    \begin{dcases*}
      1, & if $H$ is an empty hypergraph;\\
      0, & otherwise;
    \end{dcases*}
  \end{align*}
  for every $H\in\mathcal M[T_{k\operatorname{-Hypergraph}}]$.

  More generally, if $(H_n)_{n\in\mathbb{N}}$ is an increasing sequence of {\em sparse} hypergraphs, that is, such that the hyperedge density $p(K^{(k)}_k,H_n)$ is $o(1)$, then $(H_n)_{n\in\mathbb{N}}$ converges to the same limit:
  \begin{align*}
    \lim_{n\to\infty}p(H,H_n)
    & =
    \begin{dcases*}
      1, & if $H$ is an empty hypergraph;\\
      0, & otherwise.
    \end{dcases*}
  \end{align*}
\end{example}

\begin{example}[transitive tournaments] \label{ex:transitive}
  In the theory $T_{\operatorname{Tournament}}$, the sequence of transitive tournaments of increasing sizes $(\Tr_n)_{n\in\mathbb{N}}$ is a convergent sequence, since
  \begin{align*}
    \lim_{n\to\infty}p(M,\Tr_n)
    & =
    \begin{dcases*}
      1, & if $M$ is a transitive tournament;\\
      0, & otherwise
    \end{dcases*}
  \end{align*}
  for every tournament $M$.

  More generally, if $(N_n)_{n\in\mathbb{N}}$ is an increasing sequence of tournaments such that $p(\vec C_3,N_n) = o(1)$, then $(N_n)_{n\in\mathbb{N}}$ still converges to the same limit
  \begin{align} \label{eq:transitive}
    \lim_{n\to\infty}p(M,N_n)
    & =
    \begin{dcases*}
      1, & if $M$ is a transitive tournament;\\
      0, & otherwise.
    \end{dcases*}
  \end{align}
\end{example}

\begin{example} \label{ex:substitutional}
Let $\ell\geq 1$ be an integer, and let us define the theory of ordinary graphs forbidding even cycles
$C_{2\ell}$, not necessarily induced. If we do it naively, by appending to $T_{\operatorname{Graph}}$ the axiom
$\neg\of{\bigwedge_{i\in\mathbb Z_{2\ell}} E(x_i, x_{i+1})}$, we immediately realize that the
instance of this formula obtained by the substitution $x_i\mapsto x_{i\bmod 2}$ is simply $\neg E(x_0,x_1)$ and
what we get is the theory of empty graphs.

Thus, we have to be careful and explicitly forbid variable collisions (which we already did on appropriate
occasions before) as, say,
\begin{equation} \label{eq:c_2l}
\of{\bigwedge_{i\neq j} (x_i\neq x_j)} \longrightarrow \of{\neg\of{\bigwedge_{i\in\mathbb Z_{2\ell}} E(x_i, x_{i+1})}}.
\end{equation}
Then this theory certainly has (quite) non-trivial models of arbitrary size. Nonetheless, the celebrated Erd\H{o}s--Rado Theorem
in extremal graph theory implies that the edge density in every increasing sequence of models is $o(1)$. Hence, from
the perspective of our framework, limits of the theory $T_{\operatorname{Graph}}$ + \text{\eqref{eq:c_2l}} is just as trivial as the ones from
theories considered in the two previous examples.
\end{example}

\begin{remark} \label{rem:induced}
Examples~\ref{ex:empty}-\ref{ex:substitutional} pertain to a prominent topic in combinatorics called {\em {\rm (}Induced{\rm )} Removal Lemmas} or \emph{Property Testability}. Questions of this kind can be asked in two different forms as follows. Let $T_1$ and $T_2$ be two theories in the same language $\mathcal L$ such that $T_2$ extends $T_1$ by appending extra axioms to it, as in Example~\ref{ex:extra_axioms}. Let $(N_n)_{n\in\mathbb{N}}$ be a convergent sequence of $T_1$-models such that $\lim_{n\to\infty}p(M,N_n)=0$ for every $M\in \mathcal M[T_1]\setminus \mathcal M[T_2]$. Can $(N_n)_{n\in\mathbb{N}}$ be replaced by a sequence of models $(N'_n)_{n\in\mathbb{N}}$ {\em of the theory $T_2$} such that:
\begin{description}
\item[Version~1] We can obtain $N_n'$ from $N_n$ by altering an $o(1)$-fraction of values in the relations $R_{P,N_n}\ (P\in \mathcal L)$;

\item[Version~2] The sequence $(N'_n)_{n\in\mathbb{N}}$ converges to the same limit, that is, we have $\lim_{n\to\infty}p(M,N_n') = \lim_{n\to\infty\to}p(M,N_n)$ for all $M\in\mathcal M[T_1]$.
\end{description}
(Version~2 is clearly weaker than Version~1.)

Version~1 is the standard induced removal lemmas re-cast in the logical language. In the context of Example~\ref{ex:empty} it is obvious, but already for almost transitive tournaments (Example~\ref{ex:transitive}) it requires a non-trivial argument to prove that tournaments in a sequence with the property~\eqref{eq:transitive} can be made transitive by reverting a fraction $o(1)$ of arcs (Y.~Makarychev and I.~Mezhirov, personal communications). For the pair $T_1 = T_{\operatorname{Graph}}, T_2 = T_{\operatorname{TF-Graph}}$ it constitutes the famous {\em Triangle Removal Lemma}. For the theory $T_1=T_{k\operatorname{-Hypergraph}}$ (and arbitrary $T_2$), the first proof came from a non-trivial generalization of the Graph Regularity Lemma to hypergraphs, see~\cite{RoSc}. For general pairs $(T_1, T_2)$, the proofs due to Austin--Tao~\cite{AuT} and Aroskar--Cummings~\cite{ArC} use completely different methods. We will prove a continuous analogue of this statement
(Theorem~\ref{thm:ierl}).

Remarkably, even the much weaker Version~2 is not entirely obvious. It easily follows from either a syntactic description of limit objects (Theorem~\ref{thm:cryptomorphism}) or a semantic one (Theorem~\ref{thm:theoncryptomorphism}). But we are not aware of any entirely ``local'', ``finite'' proof of that statement.
\end{remark}

\begin{example}[Tur\'{a}n graphs]\label{ex:TuranGraphs}
  In the theory of graphs $T_{\operatorname{Graph}}$, the \emph{Tur\'{a}n graph} $T_{n,\ell}$ is the complete $\ell$-partite graph on $n$ vertices with parts as equal as possible. It is easy to see that for any fixed $\ell\in\mathbb{N}_+$, the sequence $(T_{n,\ell})_{n\in\mathbb{N}}$ is a convergent sequence and
  \begin{align*}
    \lim_{n\to\infty} \tinj(G,T_{n,\ell}) & = \frac{P_G(\ell)}{\ell^{\lvert V(G)\rvert}},
  \end{align*}
  where $P_G$ is the chromatic polynomial of $G$, that is, $P_G(\ell)$ is the number of proper vertex colorings of $G$ with (at most) $\ell$ colors.
\end{example}

\begin{example}[Erd\H{o}s--R\'{e}nyi random graphs]\label{ex:ErdosRenyi}
  The \emph{Erd\H{o}s--R\'{e}nyi random graph model} is the random graph $\rn{G_{n,p}}$ on $n$ vertices in which each edge is independently present with probability $p$.

  It is a straightforward exercise in distribution concentration (see e.g.~\cite[Theorem~4.4.5]{AlS}) to prove that the sequence $(\rn{G_{n,p}})_{n\in\mathbb{N}}$ is convergent with probability $1$ for every fixed $p\in[0,1]$ and
  \begin{equation}\label{eq:QRGraph}
    \begin{aligned}
    \lim_{n\to\infty}\tind(H,\rn{G_{n,p}}) & = p^\ell(1-p)^{\binom{m}{2}-\ell};\\
    \lim_{n\to\infty}\tinj(H,\rn{G_{n,p}}) & = p^\ell;
    \end{aligned}
  \end{equation}
  with probability $1$ for every graph $H$ with $m$ vertices and $\ell$ edges.

  A (deterministic) increasing sequence of graphs satisfying~\eqref{eq:QRGraph} (for every graph $H$) is called {\em quasi-random}. Beginning with seminal papers~\cite{Tho2,CGW} that identified several a priori different properties equivalent to quasi-randomness, it has become a very prominent area of combinatorial research. Perhaps one of the most impressive of these properties is that the non-induced version of~\eqref{eq:QRGraph}
  for just $K_2$ and $C_4$ implies that it holds for every $H$. The original proof of this fact in~\cite{CGW} is completely syntactic, but with the theory of graphons at our disposal (which was nonexistent at the time), a simpler semantic
  proof can be extracted from the much earlier paper~\cite{DiFr81}. More examples on how the semantic theory of graphons can be used to simplify syntactic proofs of graph quasi-randomness can be found in~\cite{Jan11}.

  It is worth noting that independent samples from $\rn{G_{n,p}}$ are very far apart in the {\em edit distance}
  (see~\cite[\S 8.1]{Lov4} for details of the definition), even if they are very close with respect to densities.
  This, among other things, demonstrates that the phenomenon of removal lemmas (Version~1 in Remark~\ref{rem:induced}) is
  quite unique and depends on the fact that the density of the models $N$ we are interested in is actually $o(1)$. No
  useful analogue of induced removal lemmas seems to be possible without this restriction.
\end{example}

\begin{example}[$3$-uniform random hypergraphs]\label{ex:RandomHypergraphs}
  Consider the random $3$-uniform hypergraph $\rn{H_{n,p}}$ obtained in a fashion similar to the Erd\H{o}s--R\'{e}nyi random model, that is, it is the random hypergraph on $n$ vertices in which each hyperedge is independently present with probability $p$.

  Again, it is a straightforward exercise to prove that the sequence $(\rn{H_{n,p}})_{n\in\mathbb{N}}$ is convergent with probability $1$ for every fixed $p\in[0,1]$ and
  \begin{align*}
    \lim_{n\to\infty}\tind(H,\rn{H_{n,p}}) & = p^\ell(1-p)^{\binom{m}{3}-\ell};\\
    \lim_{n\to\infty}\tinj(H,\rn{H_{n,p}}) & = p^\ell;
  \end{align*}
  with probability $1$ for every $3$-uniform hypergraph $H$ with exactly $m$ vertices and $\ell$ hyperedges.

  On the other hand, we can consider the $3$-uniform hypergraph $\rn{H'_{n,p}}$ obtained directly from $\rn{G_{n,p}}$ by
  declaring the hyperedges of $\rn{H'_{n,p}}$ to correspond to triangles of $\rn{G_{n,p}}$ (that is, $\rn{H'_{n,p}} = I(\rn{G_{n,p}})$ for the open interpretation from Example~\ref{ex:triangleinterpret}). Again it is straightforward to check that the sequence $(\rn{H'_{n,p}})_{n\in\mathbb{N}}$ is convergent with probability $1$ for every fixed $p\in[0,1]$.

  Let us now put $p = q^3$. Then
  \begin{align*}
    \lim_{n\to\infty}p(K_3^{(3)},\rn{H_{n,p}}) & = \lim_{n\to\infty}p(K_3^{(3)},\rn{H'_{n,q}}) = p,
  \end{align*}
  with probability $1$ for the $3$-uniform hypergraph $K_3^{(3)}$ corresponding to one hyperedge, i.e., the hypergraphs $\rn{H_{n,p}}$ and $\rn{H'_{n,p}}$ asymptotically have the
  same (hyper)edge density, just as expected. However, for $p\in(0,1)$, these sequences are quite different in terms of other densities. For example,
  let $K_4^-$ be the $3$-uniform hypergraph on $4$ vertices with exactly $3$ hyperedges. Then we have
  \begin{align*}
    \lim_{n\to\infty}\tind(K_4^-,\rn{H_{n,p}}) & = p^3(1-p); &
    \lim_{n\to\infty}\tind(K_4^-,\rn{H'_{n,q}}) & = 0;
  \end{align*}
  with probability $1$.
\end{example}

\begin{example}[Erd\H{o}s--Stone--Simonovits theorem]\label{ex:ESS}
  Given a family of non-empty graphs $\mathcal F$, let
  \begin{align*}
    \pi(\mathcal F)
    & \df
    \lim_{n\to\infty} \max\{p(K_2,G) \mid \lvert V(G)\rvert = n\land\forall F\in \mathcal F,\tinj(F,G) = 0\}.
  \end{align*}
  From Proposition~\ref{prop:compact}, it follows that $\pi(\mathcal F)$ is the same as the maximum of $\lim_{n\to\infty} p(K_2,G_n)$ over all convergent sequences $(G_n)_{n\in\mathbb{N}}$ in $\Forbp{T_{\operatorname{Graph}}}{\mathcal F}$ (cf.~Example~\ref{ex:extra_axioms}). The celebrated Erd\H{o}s--Stone--Simonovits Theorem says that
  \begin{align*}
    \pi(\mathcal F) & = 1 - \frac{1}{\inf_{F\in \mathcal F} \chi(F) - 1},
  \end{align*}
  where $\chi(F)$ is the chromatic number of the graph $F$.

  Remarkably, this theorem extends to the setting of ordered graphs~\cite[Theorem 1]{PaT} as follows: consider the order-erasing interpretation $I\interpret{T_{\operatorname{Graph}}}{T_{\operatorname{Graph}}^<}$ (cf.~Example~\ref{ex:strucerase}) and for a family of non-empty ordered graphs $\mathcal F$ let
  \begin{align*}
    \pi_<(\mathcal F)
    & \df
    \lim_{n\to\infty} \max\{p(K_2,I(G)) \mid \lvert V(G)\rvert = n\land\forall F\in \mathcal F,\tinj(F,G) = 0\}
    \\
    & =
    \max\{\lim_{n\to\infty} p(K_2,I(G_n)) \mid (G_n)_{n\in\mathbb{N}} \text{ is a convergent sequence in }
    \Forbp{T_{\operatorname{Graph}}^<}{\mathcal F}\}.
  \end{align*}
  Then we still have
  \begin{align*}
    \pi_<(\mathcal F) & = 1 - \frac{1}{\inf_{F\in \mathcal F} \chi_<(F) - 1},
  \end{align*}
  where $\chi_<(F)$ is the \emph{interval chromatic number} of $F$, that is, the smallest $k$ such that there exists a proper vertex coloring of $F$ with $k$ colors, each color class being an interval of the order of the vertices.
The analogous result~\cite[Theorem~1]{BKV} for cyclically ordered graphs ($\pi_{\operatorname{Cyc}}(\mathcal{F})$) holds using the \emph{cyclic chromatic number} $\chi_{\operatorname{Cyc}}(F)$, which is the smallest $k$ such
that there exists a proper vertex coloring of $F$ with $k$ colors, each color class being an interval of the cyclic order of the vertices.
In contrast with the usual chromatic number, which is NP-hard, both the interval and the cyclic chromatic numbers are easily computable in polynomial time with a greedy algorithm.
\end{example}

\subsection{Flag algebras -- the syntax} \label{sec:flag}

In one sentence, the theory of flag algebras can be summarized as the study of relations that the coordinates of
$\phi\in[0,1]^{\mathcal{M}}$ must satisfy if $\phi$ is obtained as the limit of functionals $p(\place, N_n)$ for a
converging sequence of models $(N_n)_{n\in\mathbb{N}}$ for its own sake, without any explicit references to the actual
limit object. In this section we present a lightweight\footnote{The main difference is that for our purposes here, we
only need to work with models without labels. In particular, we completely skip
all material pertaining to non-trivial ``types''.} introduction to the basic concepts of flag algebras: all theorems
of this section are simplified versions of~\cite{flag} and we refer the interested reader to the aforementioned work
for more thorough treatment.

The first kind of relations that the coordinates of a limit $\phi\in[0,1]^{\mathcal{M}}$ must respect is given by the so-called chain rule.

\begin{lemma}[chain rule]\label{lem:chain}
  If $M,N\in\mathcal{M}$ are models of a theory $T$ and $\lvert V(M)\rvert\leq\ell\leq\lvert V(N)\rvert$, then
  \begin{align*}
    p(M,N) & = \sum_{M'\in\mathcal{M}_\ell} p(M,M')p(M',N).
  \end{align*}
\end{lemma}

This means that if we extend $\phi\in[0,1]^{\mathcal{M}[T]}$ to a linear functional on the space $\mathbb{R}\mathcal{M}[T]$ of formal linear combinations of finite models by
\begin{align*}
  \phi\left(\sum_{M\in\mathcal{M}}c_M M \right) & \df \sum_{M\in\mathcal{M}} c_M\phi(M),
\end{align*}
then the linear subspace $\mathcal{K}[T]$ generated by elements of the form
\begin{align*}
  M - \sum_{M'\in\mathcal{M}_\ell} p(M,M')M',
\end{align*}
for $\ell\geq\lvert V(M)\rvert$ is contained in the kernel of $\phi$. In other words, defining $\mathcal{A}[T] \df \mathbb{R}\mathcal{M}[T]/\mathcal{K}[T]$, we can think of $\phi$ as a linear functional on $\mathcal{A}[T]$.

Note for the record that similar identities hold for $\tind, \tinj$:
$$
 \tind(M,N) = \sum_{M'\in\mathcal{M}_\ell} \tind(M,M')p(M',N)
$$
and
\begin{equation} \label{eq:chain_inj}
 \tinj(M,N) = \sum_{M'\in\mathcal{M}_\ell} \tinj(M,M')p(M',N),
\end{equation}
whenever $\ell\geq\lvert V(M)\rvert$. We see that induced densities appear quite naturally even if
we are interested in graph homomorphisms/positive embeddings.

The next step is to study what sort of relations must be satisfied by products $\phi(M)\phi(N)$ of coordinates of a limit $\phi\in[0,1]^{\mathcal{M}}$. For that, we need to extend the definition of density to more than one model.

\begin{definition} \label{def:mult}
  Let $m_1,m_2,\ldots,m_t,n\in\mathbb{N}$ be non-negative integers such that $\sum_{i=1}^tm_i\leq n$ and
  let $M_1,M_2,\ldots,M_t,N\in\mathcal{M}$ be models of a theory $T$ such that $\lvert V(M_i)\rvert = m_i$,
  for every $i\in[t]$ and $\lvert V(N)\rvert = n$. We define the quantity $p(M_1,M_2,\ldots,M_t;N)$ via the following probabilistic experiment. We pick {\em pairwise disjoint} subsets
  $(\rn{V_1},\rn{V_2},\ldots,\rn{V_t})$ of $V(N)$ uniformly at random and set
  \begin{align*}
    p(M_1,M_2,\ldots,M_t;N)
    & \df
    \prob{\forall i\in[t], N\vert_{\rn V_i} \cong M_i}.
  \end{align*}
\end{definition}

\begin{lemma}[chain rule]\label{lem:chain2}
  If $M_1,M_2,\ldots,M_t,N\in\mathcal{M}$ are models of a theory $T$ and $\sum_{i=1}^t\lvert V(M_i)\rvert\leq\ell\leq\lvert V(N)\rvert$, then
  \begin{align*}
    p(M_1,M_2,\ldots,M_t;N) & = \sum_{M\in\mathcal{M}_\ell} p(M_1,M_2,\ldots,M_t;M)p(M,N).
  \end{align*}
\end{lemma}

\begin{example}[graphs]
  In the theory of graphs $T_{\operatorname{Graph}}$, for every $\ell\geq 4$, we have
  \begin{align*}
    p(K_2,K_2;K_\ell) & = 1;\\
    p(K_2,K_2;P_\ell) & = \frac{4}{\ell(\ell-1)};\\
    p(K_2,\overline{K}_2;P_\ell) & = \frac{2(\ell-3)}{\ell(\ell-1)}.
  \end{align*}
\end{example}

\begin{example}[permutations]
  In the theory of permutations, we have
  \begin{align*}
    p(12,12;14235) & = \frac{18}{30};\\
    p(12,21;14235) & = \frac{6}{30};\\
    p(21,21;14235) & = \frac{0}{30} = 0.
  \end{align*}
\end{example}

Definition~\ref{def:mult} may seem not very natural at first, since we compute densities avoiding collisions. But this is precisely what turns out to be necessary (and sufficient) to formally capture the ``infiniteness'' of our object: collisions have zero probability of occurring. This leads to what is called the \emph{flag algebra} of the theory $T$.

\begin{lemma}
  The bilinear mapping $\mathbb{R}\mathcal{M}\times\mathbb{R}\mathcal{M}\to\mathcal{A}$ defined by
  \begin{align*}
    M_1\cdot M_2 & \df \sum_{N\in\mathcal{M}_n}p(M_1,M_2;N)N,
  \end{align*}
  for every $M_1,M_2\in\mathcal{M}$ and every $n\geq\lvert V(M_1)\rvert+\lvert V(M_2)\rvert$ does not depend on the choice
  of $n$ and induces a symmetric bilinear mapping $\mathcal{A}\times\mathcal{A}\to\mathcal{A}$.

  Furthermore, if $T$ is non-degenerate, then this induced mapping endows the vector space $\mathcal{A}$ with the structure
  of a commutative associative algebra whose identity element $1$ is the (equivalence class of the) unique model on $0$ vertices.
\end{lemma}

The next lemma quantitatively refines the remark about collisions made above.

\begin{lemma}\label{lem:flaglowcollision}
  If $M_1\in\mathcal{M}_{m_1},M_2\in\mathcal{M}_{m_2},\ldots,M_t\in\mathcal{M}_{m_t}$ and $N\in\mathcal{M}_n$, then
  \begin{align*}
    \left\lvert p(M_1,M_2,\ldots,M_t;N) - \prod_{i=1}^tp(M_i,N)\right\rvert
    & \leq
    \frac{(m_1+m_2+\cdots+m_t)^{O(1)}}{n}.
  \end{align*}
\end{lemma}

In particular, if $(N_n)_{n\in\mathbb{N}}$ is a sequence converging to $\phi\in [0,1]^{\mathcal{M}}$, then the
functionals $p(\place,N_n)$ look more and more like algebra homomorphisms from $\mathcal{A}$ to $\mathbb{R}$, hence in the limit $\phi$ must be an algebra homomorphism.

One more property that a limit $\phi$ must satisfy is that $\phi(M)\geq 0$ for every model $M\in\mathcal{M}$.

\begin{definition}
  In a non-degenerate theory $T$, the set of \emph{positive homomorphisms} $\HomT{T}$ is the set of all algebra homomorphisms
  $\phi\function{\mathcal{A}[T]}{\mathbb R}$ such that $\phi(M)\geq 0$ for every $M\in\mathcal{M}[T]$.
\end{definition}

Now, the next (relatively simple) result says that the set of constraints we have imposed on $\phi$ is both sound and complete.

\begin{theorem}[Lov\'{a}sz--Szegedy~\cite{LoSz}, Razborov~\cite{flag}]\label{thm:cryptomorphism}
  If $(N_n)_{n\in\mathbb{N}}$ is a convergent sequence of models, then $\lim_{n\to\infty} p^{N_n}\!\in\!\Hom$. Conversely, if $\phi\in\Hom$, then there exists a convergent sequence of models $(N_n)_{n\in\mathbb{N}}$ such that $\lim_{n\to\infty}p^{N_n} = \phi$.
\end{theorem}

In other words, the above theorem says that convergent sequences of models are cryptomorphic to flag algebra positive homomorphisms.

\medskip

In Section~\ref{sec:interpretations} we saw that an open interpretation $I\interpret{T_1}{T_2}$
gives us a natural way of creating a model $I(M)\in\mathcal{M}[T_1]$ from a model $M\in\mathcal{M}[T_2]$.
Given the ``intended'' meaning (vaguely suggested by Theorem~\ref{thm:cryptomorphism}) of $\HomT T$ as the set of
``infinite'' models of the theory $T$, it is natural to expect that $I$ should also give rise to a mapping
$\HomT{T_2}\to\HomT{T_1}$, and that this latter mapping can be described by simple syntactical means. It indeed
turns out to be the case.

\begin{theorem}\label{thm:flagpi}
  Let $T_1$ and $T_2$ be non-degenerate theories and $I\interpret{T_1}{T_2}$ be an open interpretation. Then the linear mapping $\mathbb{R}\mathcal{M}[T_1]\to\mathcal{A}[T_2]$ defined by
  \begin{align*}
    \pi^I(M_1) & \df \sum\{M_2\in\mathcal{M}[T_2] \mid I(M_2)\cong M_1\},
  \end{align*}
  for every $M_1\in\mathcal{M}[T_1]$, satisfies $\pi^I(\mathcal{K}[T_1]) = 0$ and hence induces a mapping $\pi^I\function{\mathcal{A}[T_1]}{\mathcal{A}[T_2]}$. This mapping  is a positive algebra homomorphism, which in particular implies that if $\phi\in\HomT{T_2}$, then $\phi\comp\pi^I\in\HomT{T_1}$.
\end{theorem}

Before concluding with a few examples, let us interpret the theorem above in categorical terms.

Let \textsc{POAlg} be the category of {\em partially ordered associative commutative $\mathbb R$-algebras}. Its objects are pairs $(A,\leq)$, where $A$ is an
associative commutative algebra,
and $\leq$ is a partial order on $A$ compatible with algebra operations. By this we mean that $x\leq y\Longrightarrow x+z\leq y+z$, $(x\geq 0 \land y\geq 0)\Longrightarrow xy\geq 0$, and the restriction of $\leq$ onto $\mathbb R$ is the standard linear order. Morphisms $f\function{(A_1,\leq_1)}{(A_2,\leq_2)}$ of \textsc{POAlg} are algebra homomorphisms $f\function{A_1}{A_2}$ such that $x\leq y\Longrightarrow f(x)\leq f(y)$.

Let now\footnote{We use the symbol $\ll$ since $\leq$ was already reserved in~\cite{flag} for a much stronger semantic version.} $\ll_T$ be the partial order on $\mathcal A[T]$ defined as follows: we have $f\ll_T g$ if and only
if $(g-f)$ can be expressed (in $\mathcal A[T]$) as $\sum_i c_iM_i$ with $M_i\in\mathcal M[T]$ and $c_i\geq 0$. It is straightforward to check that this order is compatible with the algebra operation, that is, the pair $(\mathcal A[T],\ll_T)$ is an object of \textsc{POAlg}. Then Theorem~\ref{thm:flagpi} provides a functor $\pi$ from \textsc{Int} to \textsc{POAlg} given by
$$
\pi(T_1\xrightarrow{I}T_2) \df (\mathcal{A}[T_1],\ll_{T_1}) \xrightarrow{\pi^I} (\mathcal{A}[T_2],\ll_{T_2}).
$$
Composing it with the contravariant functor
$\operatorname{Hom}(\place,(\mathbb{R},\leq))$ from \textsc{POAlg} to
\textsc{Set}, we get the contravariant functor $\pi^\ast$ from \textsc{Int}
to \textsc{Set} given by $\pi^\ast(T)\df \HomT T$ and
$\pi^\ast(T_1\xrightarrow{I}T_2)(\phi)\df \phi\comp\pi^I$ for every $\phi\in\HomT{T_2}$. It is compatible with the action of $I$ on convergent sequences
(Section~\ref{sec:convergence}), etc.

\begin{example}[extra axioms]
  If a theory $T'$ is obtained from a theory $T$ in the same language by adding extra axioms and $I\interpret{T}{T'}$ is the identity translation, then for every model $M\in\mathcal{M}[T]$ of $T$, we have
  \begin{align*}
    \pi^I(M)
    & =
    \begin{dcases*}
      M, & if $M$ is a model of $T'$;\\
      0, & otherwise.
    \end{dcases*}
  \end{align*}
  Furthermore $\mathcal A[T']$ is a factor-algebra of $\mathcal A[T]$, and hence $\pi^\ast(I)$ is
  injective.
\end{example}

\begin{example}[color-erasing and orientation-erasing] \label{ex:erasing}
  If $T$ is a non-degenerate theory and $I\interpret{T}{T^c}$ is the color-erasing interpretation, then
  \begin{align*}
    \pi^I(M)
    & =
    \sum\{M'\in\mathcal{M}[T^c] \mid M'\text{ is a coloring of } M \text{ with }c \text{ colors}\},
  \end{align*}
  for every model $M\in\mathcal{M}[T]$ of $T$. Note that this sum is {\em unweighted}, that is
  it ignores the number of ways in which $M'$ can be obtained from $M$ even if $M$ possesses
  non-trivial automorphisms.

  Similarly, if $I\interpret{T_{\operatorname{Graph}}}{T_{\operatorname{Orgraph}}}$ is the orientation-erasing interpretation, then
  \begin{align*}
    \pi^I(G)
    & =
    \sum\{G'\in\mathcal{M}[T_{\operatorname{Orgraph}}] \mid G'\text{ is an orientation of } G\},
  \end{align*}
  for every graph $G$.

  In both cases $\pi^I$ is an injective algebra homomorphism, but
  $\pi^\ast(I)$ is very far from being injective: if, for example, we apply
  the orientation-erasing interpretation to an arbitrary ``tournamon'' (i.e.,
  an element of $\HomT{T_{\operatorname{Tournament}}}$), we
   get the same complete graphon. It is not hard to see, though, that
  in both cases $\pi^\ast(I)$ is surjective: it basically says that every graph can be colored or oriented in at least one way.
\end{example}

\begin{example}
Let us now review under this angle the ``triangular'' interpretation
$I\interpret{T_{\operatorname{3-Hypergraph}}}{T_{\operatorname{Graph}}}$ given by $I(E)(x,y,z)\equiv
(E(x,y)\land E(y,z)\land E(x,z))$ from Example~\ref{ex:triangleinterpret}.
First, the algebra homomorphism $\pi^I$ is not injective since $\pi^I(K_4^-)=0$ (cf.~Example~\ref{ex:RandomHypergraphs}). By the same token, the induced map $\pi^\ast(I)$ is not
surjective: any $\phi\in\HomT{T_{\operatorname{3-Hypergraph}}}$ in its image must
necessarily satisfy $\phi(K_4^-)=0$. The map $\pi^\ast(I)$ is also not injective since all $\phi\in\HomT{T_{\operatorname{Graph}}}$ with $\phi(K_3)=0$ lead to the same (empty) 3-graph. As an immediate consequence, the algebra homomorphism $\pi^I$ cannot be surjective:
say, $K_2$ is not in its range.
\end{example}

We finish this section with an example of an application of open interpretations. To the best of our knowledge, no statement precisely in the form~\eqref{eq:exI} below is present in the literature.

\begin{example}[Erd\H{o}s--Stone--Simonovits theorem, cntd.]
  Further generalizing Example~\ref{ex:ESS}, suppose we are given an interpretation $I\interpret{T_{\operatorname{Graph}}}{T}$ of the theory of graphs in a non-degenerate theory $T$.
  Let\footnote{Unfortunately, there is an unavoidable collision of notation here: both the extremal value $\pi_I$ and the flag algebra homomorphism
  $\pi^I\function{\mathcal{A}[T_{\operatorname{Graph}}]}{\mathcal{A}[T]}$ use the letter $\pi$.}
  \begin{align*}
    \pi_I
    & \df
    \lim_{n\to\infty} \max\{p(K_2,I(N)) \mid N\in\mathcal{M}_n[T]\}
    \\
    & =
    \max\{\lim_{n\to\infty} p(K_2,I(N_n)) \mid (N_n)_{n\in\mathbb{N}} \text{ is a convergent sequence in }
    \mathcal{M}[T]\}
    \\
    & =
    \max\{\phi(\pi^I(K_2)) \mid \phi\in\HomT{T}\}
  \end{align*}
  (the maximum in the third line exists since $\HomT{T}$ is compact, and the second line is equal to the
  third due to Theorem~\ref{thm:cryptomorphism}).
  Recall that $T_{n,\ell}$ denotes the Tur\'{a}n graph (see Example~\ref{ex:TuranGraphs}) and let
  \begin{align} \label{eq:chi_def}
    \chi(I)
    & \df
    \sup\{\ell  \mid
    \forall n\in\mathbb{N}, \exists N\in\mathcal{M}_n[T], I(N)\supseteq T_{n,\ell}\}+1.
  \end{align}
  Note that since $T$ is non-degenerate, it follows that $\chi(I)\geq 2$.

  Let us offer a simple proof that
  \begin{align} \label{eq:ESS_general}
    \pi_I & = 1 - \frac{1}{\chi(I) - 1}.
  \end{align}

  If $\chi(I)=\infty$, then for every $n\in\mathbb{N}$, there exists $N\in\mathcal{M}_n[T]$ such that $I(N)\supseteq T_{n,n}\cong K_n$, so~\eqref{eq:ESS_general} follows trivially (with the right-hand side evaluating to $1$).

  Suppose then that $\chi(I)<\infty$. Then the definition of $\chi(I)$ implies that there exists an increasing sequence $(N_n)_{n\in\mathbb{N}}$ in $\mathcal{M}_n[T]$ satisfying $I(N_n)\supseteq T_{n,\chi(I)-1}$, so the right-hand side of~\eqref{eq:ESS_general} is a lower bound for $\pi_I$.

  For the other direction, suppose for a contradiction that there exists $\phi\in\HomT{T}$ such that $\phi(\pi^I(K_2)) > 1 - 1/(\chi(I)-1)$. Then for every $n\in\mathbb N$ we apply the ordinary Erd\H{o}s--Stone--Simonovits theorem,
  in the form given by Example~\ref{ex:ESS}, to $\mathcal F\df \{T_{n,\chi(I)}\}$ and any sequence of graphs converging to $\phi\comp \pi^I$. Taking into account the conversion formulas~\eqref{eq:ind-p}, \eqref{eq:inj-ind}, we conclude
  that there exists some $G_n\supseteq T_{n,\chi(I)}$ such that $\phi(\pi^I(G_n))>0$. This in particular implies that there exists $N_n\in\mathcal{M}[T]$ such that $\phi(N_n) > 0$ and $I(N_n) = G_n\supseteq T_{n,\chi(I)}$,
  contradicting the definition of $\chi(I)$. Therefore, the right-hand side of~\eqref{eq:ESS_general} is also an upper bound for $\pi_I$.

  \medskip

Let us note a few interesting special cases. First,
  let $\mathcal{F}$ be a family of graphs such that $T\df \Forb{T_{\operatorname{Graph}}}{\mathcal{F}}$ is non-degenerate, i.e., there are arbitrarily large graphs missing all $F\in\mathcal F$ as {\em induced} subgraphs. Let $I$ be the axiom-adding interpretation (i.e., $I$ acts identically on the predicate symbol $E$). Then~\eqref{eq:ESS_general} becomes the {\em induced} version of Erd\H{o}s--Stone--Simonovits:
\begin{equation} \label{eq:exI}
\pi_{\text{ind}}(\mathcal F) = 1 - \frac{1}{\chi_{\text{ind}}(\mathcal F) - 1}.
\end{equation}
Here $\pi_{\text{ind}}(\mathcal F)$ is the maximal possible (asymptotically) density of a graph that does not contain {\em induced} copies of graphs in
$\mathcal F$, and $\chi_{\text{ind}}(\mathcal F)$ is given by~\eqref{eq:chi_def}, where $N$ runs over all $\mathcal F$-free graphs. As we mentioned before, we have not seen this statement in the literature in this generality.

But we should also remark that the quantity $\chi_{\text{ind}}(\mathcal F)$ is not as well-behaving as the ordinary chromatic number; in fact, a priori it is not even clear that it is computable. As yet another indication
let us note that principality does not hold in the induced setting. For example, $\chi_{\text{ind}}(\{P_3\})=\infty$ (as $K_n$ does not contain induced copies of $P_3$) and $\chi_{\text{ind}}(\{K_3\})= \chi(K_3)=3$ but $\chi_{\text{ind}}(\{P_3,K_3\})=\chi(P_3)=2$.

Next, let $\mathcal F$ be a family of non-empty graphs [ordered graphs, cyclically ordered graphs] and let $I_{\mathcal{F}}\interpret{T_{\operatorname{Graph}}}{\Forbp{T_{\operatorname{Graph}}}{\mathcal{F}}}$ [$I^<_{\mathcal{F}}\interpret{T_{\operatorname{Graph}}}{\Forbp{T_{\operatorname{Graph}}^<}{\mathcal{F}}}$, $I^{\operatorname{Cyc}}_{\mathcal{F}}\interpret{T_{\operatorname{Graph}}}{\Forbp{T_{\operatorname{Graph}}^{\operatorname{Cyc}}}{\mathcal{F}}}$, respectively] again act identically on the predicate symbol $E$. Then the extremal values of Example~\ref{ex:ESS} can be obtained as
  \begin{align*}
    \pi(\mathcal{F}) & = \pi_{I_{\mathcal{F}}}; &
    \pi_<(\mathcal{F}) & = \pi_{I^<_{\mathcal{F}}}; &
    \pi_{\operatorname{Cyc}}(\mathcal{F}) & = \pi_{I^{\operatorname{Cyc}}_{\mathcal{F}}}.
  \end{align*}
For these particular cases we also have principality (see e.g.~\cite[Theorem~1]{PaT} and~\cite[Theorem~1]{BKV}):
  \begin{align*}
    \chi(I_{\mathcal{F}}) & = \inf_{F\in\mathcal{F}} \chi(F); &
    \chi(I^<_{\mathcal{F}}) & = \inf_{F\in\mathcal{F}} \chi_<(F); &
    \chi(I^{\operatorname{Cyc}}_{\mathcal{F}}) & = \inf_{F\in\mathcal{F}} \chi_{\operatorname{Cyc}}(F).
  \end{align*}

Let us remark, however, that this is not true in general even in the
non-induced setting. For example, consider the theory
$T_{\operatorname{Graph}}^2$ of graphs with vertex coloring into two colors
(the coloring need not be proper) and for a family $\mathcal{F}$ of
non-empty colored graphs, let
$I^2_{\mathcal{F}}\interpret{T_{\operatorname{Graph}}}{\Forbp{T_{\operatorname{Graph}}^2}{\mathcal{F}}}$
act as identity on $E$. If
$F_1,F_2,F_3\in\mathcal{M}_2[T_{\operatorname{Graph}}^2]$ are the three
models such that $I_2(F_i)\cong K_2$, then we have $\chi(I^2_{\{F_i\}}) =
\infty$ (as we can color all vertices of $T_{n,\ell}$ with the same color so
as to avoid $F_i$) but $\chi(I^2_{\{F_1,F_2,F_3\}}) = 2$ (as $\pi^I(K_2) =
F_1 + F_2 + F_3$).

  Just as in the case of the classic Erd\H{o}s--Stone--Simonovits Theorem, when $\chi(I)=2$, \eqref{eq:ESS_general} does not say anything useful about the asymptotic behavior of the maximum number of edges of models of $T$ as the graphs in the image of the interpretation are necessarily sparse. We refer the interested reader to~\cite{BKV,PaT,Tar} for results in this sparse setting for $T_{\operatorname{Graph}}^<$ and $T_{\operatorname{Graph}}^{\operatorname{Cyc}}$.
\end{example}

\subsection{Graphons} \label{sec:graphons}

In this section we present the most successful case of the semantical approach
so far: the limit objects of $T_{\operatorname{Graph}}$. Again, we give only a
few most basic facts about graphons; for more details we
refer the interested reader to~\cite{Lov4}.

\begin{definition}[Graphons]
  A \emph{graphon} is a symmetric (Lebesgue) measurable function $W\function{[0,1]^2}{[0,1]}$, that is, a measurable function
  such that $W(x,y)=W(y,x)$ for every $x,y\in[0,1]$.

  If $W$ is a graphon and $H$ is a graph, then we define
  \begin{equation}\label{eq:graphondensities}
    \begin{aligned}
      \tinj(H,W)
      & \df
      \int_{[0,1]^{V(H)}}\prod_{\{v,w\}\in E(H)}W(x_v,x_w)dx;
      \\
      \tind(H,W)
      & \df
      \int_{[0,1]^{V(H)}}
      \;\quad\prod_{\mathclap{\{v,w\}\in E(H)}}\;\;
      W(x_v,x_w)
      \;\;\prod_{\mathclap{\{v,w\}\in E(\overline{H})}}\;\;
      (1-W(x_v,x_w))dx;
      \\
      p(H,W)
      & \df
      \frac{\lvert V(H)\rvert!}{\lvert\Aut(H)\rvert}\tind(H,W),
    \end{aligned}
  \end{equation}
  where $E(H) = \{\{v,w\} \mid (v,w)\in R_{E,H}\}$ denotes the set of edges of $H$ and $\overline{H}$ denotes the complement of $H$.
\end{definition}

The intuition behind the notion of a graphon is that it is a graph whose vertices are points of
$[0,1]$ and $(v,w)\in[0,1]^2$ is a weighted edge of weight $W(v,w)$.
Respectively, the formulas~\eqref{eq:graphondensities} are identical to those
introduced in Section~\ref{sec:densities}, except that we replace averaging
over a finite domain by integration (that can be also viewed as averaging over $[0,1]$).

The next theorem says that graphons capture the limits of
convergent sequences of graphs.
\begin{theorem}[Lov\'{a}sz--Szegedy~\cite{LoSz}]\label{thm:graphoncryptomorphism}
  If $(G_n)_{n\in\mathbb{N}}$ is a convergent sequence of graphs, then there exists a graphon $W$ such that
  \begin{equation}\label{eq:graphoncryptomorphism}
    \begin{aligned}
      \lim_{n\to\infty}\tinj(H,G_n) & = \tinj(H,W);\\
      \lim_{n\to\infty}\tind(H,G_n) & = \tind(H,W);\\
      \lim_{n\to\infty}p(H,G_n) & = p(H,W);
    \end{aligned}
  \end{equation}
  for every fixed graph $H$. Conversely, if $W$ is a graphon, then there exists a convergent sequence of graphs $(G_n)_{n\in\mathbb{N}}$ such that~\eqref{eq:graphoncryptomorphism} holds for every fixed graph
  $H$.
\end{theorem}
Combining this theorem with Theorem~\ref{thm:cryptomorphism}, we get
\begin{corollary} \label{cor:graph_cryptomorphisms}
If $\phi\in\HomT{T_{\operatorname{\rm Graph}}}$, then there exists a graphon $W$ such
that $\phi(H)=p(H,W)$ for every fixed graph $H$. Conversely, for every graphon $W$,
the functional $p(\place,W)\in[0,1]^{\mathcal{M}[T_{\operatorname{\rm Graph}}]}$ defines an
element of $\HomT{T_{\operatorname{\rm Graph}}}$.
\end{corollary}

The examples below illustrate that when a sequence of graphs has a good
structure, it is fairly easy to ``guess'' the limit graphon.
\begin{example}[$\ell$-disjoint cliques and Tur\'{a}n graphons]
  Let $\ell\geq 1$ be fixed and let $G_n$ be the graph on $\ell n$ vertices consisting of $\ell$ disjoint cliques of $n$ vertices each. Then $(G_n)_{n\in\mathbb{N}}$ is convergent and the natural limit graphon of this sequence is the step-function $W_\ell\function{[0,1]^2}{[0,1]}$ given by
  \begin{align*}
    W_\ell(x,y) & =
    \begin{dcases}
      1, & \text{ if there exists } i\in[\ell]
      \text{ such that } x,y\in\left[\frac{i-1}{\ell},\frac{i}{\ell}\right);\\
      0, & \text{ otherwise.}
    \end{dcases}
  \end{align*}

  Following up on Example~\ref{ex:TuranGraphs}, for every $\ell\in\mathbb{N}_+$ the sequence $(T_{n,\ell})_{n\in\mathbb{N}}$ of Tur\'{a}n graphs converges to $1 - W_\ell$.
\end{example}

Note that in the example above, we have $\overline{G_n} = T_{\ell n,\ell}$. This is a special case of a more general fact: if $(G_n)_{n\in\mathbb{N}}$ is a sequence of graphs converging to some graphon $W$, then $(\overline{G_n})_{n\in\mathbb{N}}$ is also convergent and converges to the graphon $1-W$.

\begin{example}[Erd\H{o}s--R\'{e}nyi random model]
  For every $p\in[0,1]$, let $W_p\equiv p$ be the constant graphon with value $p$. Then $(\rn{G_{n,p}})_{n\in\mathbb{N}}$ converges to $W_p$ with probability $1$.
\end{example}

The next natural question to ask is when two graphons are equivalent in the sense that they represent the limit of the same convergent sequences of graphs. It is expected that if we permute the elements of $[0,1]$, then the graphon should still represent the same limit. However, since we must preserve measurability of the graphon, the correct way of ``permuting'' the elements of a graphon is to use measure preserving functions. The next theorem characterizes this notion of graphon equivalence.

\begin{theorem}[{{\cite[Corollary~10.35a]{Lov4}}}]
\label{thm:graphon_uniqueness}
  Let $W,W'$ be two graphons. The following are equivalent.
  \begin{itemize}
  \item For every graph $H$, we have $\tinj(H,W) = \tinj(H,W')$, that is,
      $W$ and $W'$ correspond to the same element of
      $\HomT{T_{\operatorname{Graph}}}$ in the sense of Corollary~\ref{cor:graph_cryptomorphisms}.
  \item There exist measure preserving functions $f,g\function{[0,1]}{[0,1]}$ such that $W(f(x),f(y)) = W'(g(x),g(y))$ for almost every $(x,y)\in[0,1]^2$.
  \end{itemize}
\end{theorem}

In this context, the main purpose of our work can be summarized as follows:
find a natural and convenient (and certainly well-behaving with respect to
open interpretations) generalization of graphons to arbitrary theories so
that Corollary~\ref{cor:graph_cryptomorphisms} and Theorem~\ref{thm:graphon_uniqueness} still hold. Before embarking on the project, let us briefly
review one prominent generalization of graphons that has been known before
and that is quite important to our work. It introduced much of the language we
will be using.

\subsection{Hypergraphons} \label{sec:hypergraphons}

In this section we present the first case of a limit object of a theory with predicates of arity larger than $2$: hypergraphons~\cite{ElS}.

In analogy with graphons, one might conjecture that the correct way to define a $k$-uniform hypergraphon would be as a symmetric measurable function $W\function{[0,1]^k}{[0,1]}$ and define $\tinj$, $\tind$ and $p$ in analogy with~\eqref{eq:graphondensities}. However, the example below shows that this does not work.

\begin{example}\label{ex:HypergraphBadLimitConstant}
  Following up on Example~\ref{ex:RandomHypergraphs}, we know that the sequence of random $3$-uniform
  hypergraphs $(\rn{H'_{n,p}})_{n\in\mathbb{N}}$ is convergent with probability $1$. Also, since all vertices are
  equal in this random model, we would ``expect'' the limit hypergraphon to be a constant function. However,
  any constant $W\function{[0,1]^3}{[0,1]}$ does not work as shown by the same calculation
  with the 3-graph $K_4^-$ as in Example~\ref{ex:RandomHypergraphs}.
\end{example}

As a matter of fact, the limit of this sequence cannot be written as {\em any} symmetric measurable function $W\function{[0,1]^3}{[0,1]}$ whatsoever: this easily follows from the uniqueness theorem for
hypergraphons~\cite[Theorem~9]{ElS} that we will also recover below (Theorem~\ref{thm:TheonUniqueness}). The reason why symmetric measurable functions $W\function{[0,1]^3}{[0,1]}$ do not work is that we are missing degrees of freedom associated to pairs of vertices (say, all ``non-trivial'' elements in the image of $\pi^\ast(I)$, where $I\interpret{T_{\operatorname{3-Hypergraph}}}{T_{\operatorname{Graph}}}$ is a ``non-trivial'' interpretation are bound to not be covered).

Before we go into the definition of hypergraphons, let us first fix some notation that we will also use in the following sections.
To the reader familiar with hypergraphons let us remark that our notation does {\em not}
a priori assume any symmetry.

\begin{definition}\label{def:EV}
  For a set $V$, let $r(V)$ denote
  the collection of all non-empty finite subsets of $V$ and let
  \begin{align*}
    \mathcal{E}_V & \df [0,1]^{r(V)}.
  \end{align*}
  In particular, if $V$ is finite, then the set $\mathcal{E}_V$ is a hypercube of dimension
  $2^{\lvert V\rvert}-1$ where each coordinate is indexed by a non-empty subset of $V$. We endow $\mathcal E_V$ with the standard Lebesgue measure $\lambda$, which turns it into a probability space.

  We define the (right) action of the symmetric group $S_V$ over $V$ on $\mathcal{E}_V$ by letting
  \begin{align*}
    (x\cdot\sigma)_A \df x_{\sigma(A)}
  \end{align*}
  for a permutation $\sigma\function{V}{V}$ and a point $x=(x_A)_{A\in r(V)}\in\mathcal{E}_V$.

   As a shorthand, when $V=[k]$, we will write $r(k)$, $\mathcal{E}_k$ and $S_k$ instead of $r([k])$, $\mathcal{E}_{[k]}$ and $S_{[k]}$ respectively. Elements in $r(k)$ are ordered as follows: a set $A$ precedes a set $B$ if and only if $\lvert A\rvert <\lvert B\rvert$ or $\lvert A\rvert = \lvert B\rvert$ and $A>B$ in the lexicographic order. This determines a natural identification between $\mathcal E_k$ and $[0,1]^{2^k-1}$: for example, the point $(a,b,c,d,e,f,g)\in [0,1]^7$ corresponds to the point $x\in\mathcal E_3$ given by $x_{\{1\}}=a,\ x_{\{2\}}=b,\ x_{\{3\}}=c,\ x_{\{1,2\}}=d,\ x_{\{1,3\}}=e,\ x_{\{2,3\}}=f,\ x_{\{1,2,3\}}=g$.

  For an injective function $\alpha\injection{[k]}{V}$ we denote, with slight abuse of notation, the induced function $\alpha\injection{r(k)}{r(V)}$ using the same letter, that is, for every $A\in r(k)$, we have
  \begin{align*}
    \alpha(A) & \df \{\alpha(i) \mid i\in A\},
  \end{align*}
  and we let $\alpha^*\function{\mathcal{E}_V}{\mathcal{E}_k}$ be the natural projection given by
  \begin{align*}
    \alpha^\ast(x)_A \df x_{\alpha(A)}
  \end{align*}
  for every $x=(x_B)_{B\in r(V)}\in\mathcal{E}_V$ and $A\in r(k)$. This notation is consistent with the
  previously introduced action of $S_k$.
\end{definition}

\begin{definition}[Hypergraphons] \label{def:hypergraphons}
  Let $k>0$ be a fixed constant. A \emph{$k$-hypergraphon} is an $S_k$-invariant measurable subset $\mathcal H$ of $\mathcal{E}_k$.

  For a $k$-uniform hypergraph $G$, we let
  \begin{align*}
    R(G) & \df \{\alpha\injection{[k]}{V(G)} \mid \im(\alpha)\in E(G)\};
    \\
    \overline{R}(G)
    & \df
    \{\alpha\injection{[k]}{V(G)} \mid \im(\alpha)\notin E(G)\}
   \end{align*}
be ``symmetrizations'' of the sets of its edges and non-edges, respectively. Assume now that we also have a $k$-hypergraphon $\mathcal H\subseteq\mathcal E_k$. We let
    \begin{align*}
    \Tinj(G,\mathcal{H}) & \df \bigcap_{\alpha\in R(G)}(\alpha^*)^{-1}(\mathcal{H}) \subseteq \mathcal{E}_{V(G)};
    \\
    \Tind(G,\mathcal{H})
    & \df
    \Tinj(G,\mathcal{H})
    \cap
    \bigcap_{\alpha\in\overline{R}(G)}(\alpha^\ast)^{-1}(\mathcal{E}_k\setminus\mathcal{H}) \subseteq \mathcal{E}_{V(G)}.
    \end{align*}
The intuition behind these definitions is as follows. An ``induced copy'' of $G$
in $\mathcal{H}$ is a point $x\in\mathcal{E}_{V(G)}$ such that for every
hyperedge tuple $\alpha\in R(G)$, the induced function $\alpha^\ast$ maps $x$ to a point inside
$\mathcal{H}$ (i.e., a ``hyperedge'' of $\mathcal{H}$) and for every
non-hyperedge tuple $\alpha\in\overline{R}(K)$, the induced function $\alpha^\ast$ maps $x$ to a
point outside $\mathcal{H}$ (i.e., a ``non-hyperedge'' of $\mathcal{H}$). A
(non-induced) copy of $G$ is obtained by dropping the second requirement.
This makes $\Tinj(G,\mathcal{H})$ and $\Tind(G,\mathcal{H})$ intuitively
correspond to the set of non-induced and induced copies of $G$ in
$\mathcal{H}$ respectively.

Now we define induced and non-induced densities straightforwardly, just as in Sections~\ref{sec:densities} and~\ref{sec:graphons}:
    \begin{align*}
    \tinj(G,\mathcal{H})
    & \df
    \lambda(\Tinj(G,\mathcal{H}));
    \\
    \tind(G,\mathcal{H})
    & \df
    \lambda(\Tind(G,\mathcal{H}));
    \\
    p(G,\mathcal{H})
    & \df
    \frac{\lvert V(G)\rvert!}{\lvert\Aut(G)\rvert}\tind(G,\mathcal{H}).
  \end{align*}
\end{definition}

The correspondence between graphons and 2-hypergraphons is not entirely straightforward. If $\mathcal H$ is a 2-hypergraphon, then by Fubini's Theorem, the set $\{p\in [0,1]\mid (u,v,p)\in \mathcal H\}$ is measurable for almost all $(u,v)\in [0,1]^2$, and $W(u,v)\df \lambda(\set{p\in [0,1]}{(u,v,p)\in\mathcal H})$, extended arbitrarily at singular points, is also measurable. This gives us the graphon associated with $\mathcal H$ that gives rise to the same element of $\HomT{T_{\operatorname{Graph}}}$ as $\mathcal H$. Conversely, if $W\function{[0,1]^2}{[0,1]}$ is a graphon, then we can turn it into a 2-hypergraphon by letting
$$
\mathcal H\df \set{x\in\mathcal E_2}{x_{\{1,2\}}\leq W(x_{\{1\}}, x_{\{2\}})}.
$$

Analogously to Theorem~\ref{thm:graphoncryptomorphism}, the next theorem says that $k$-hypergraphons capture precisely the limits of convergent sequences of $k$-uniform hypergraphs.

\begin{theorem}[Elek--Szegedy~\cite{ElS}]\label{thm:hypergraphoncryptomorphism}
  For every convergent sequence of $k$-uniform hypergraphs $(H_n)_{n\in\mathbb{N}}$, there exists a $k$-hypergraphon $\mathcal{H}$ such that
  \begin{equation}\label{eq:hypergraphoncryptomorphism}
    \begin{aligned}
      \lim_{n\to\infty}\tinj(G,H_n) & = \tinj(G,\mathcal{H});\\
      \lim_{n\to\infty}\tind(G,H_n) & = \tind(G,\mathcal{H});\\
      \lim_{n\to\infty}p(G,H_n) & = p(G,\mathcal{H});
    \end{aligned}
  \end{equation}
  for every fixed $k$-uniform hypergraph $G$. Conversely, if $\mathcal{H}$ is a $k$-hypergraphon, then there exists a convergent sequence of $k$-uniform hypergraphs $(H_n)_{n\in\mathbb{N}}$ such that~\eqref{eq:hypergraphoncryptomorphism} holds for every fixed $k$-uniform hypergraph $G$.
\end{theorem}

Thus, we get that $k$-hypergraphons are also cryptomorphic to elements of $\HomT{T_{k\operatorname{-Hypergraph}}}$ (cf.~Corollary~\ref{cor:graph_cryptomorphisms}).

\begin{example}[$3$-uniform random hypergraphs, cntd.]\label{ex:RandomHypergraphons}
  Following up on Example~\ref{ex:RandomHypergraphs}, if we define the $3$-hypergraphons $\mathcal{H}_p$ and $\mathcal{H}'_p$ by
  \begin{align*}
    \mathcal{H}_p
    & =
    \{x\in\mathcal{E}_3 \mid x_{\{1,2,3\}}\leq p\};
    \\
    \mathcal{H}'_p
    & =
    \{x\in\mathcal{E}_3 \mid \max\{x_{\{1,2\}},x_{\{1,3\}},x_{\{2,3\}}\} \leq p\};
  \end{align*}
  then with probability $1$, the sequences $(\rn{H_{n,p}})_{n\in\mathbb{N}}$ and $(\rn{H'_{n,p}})_{n\in\mathbb{N}}$ converge to $\mathcal{H}_p$ and $\mathcal{H}'_p$ respectively.
\end{example}

As one can imagine, since hypergraphons are somewhat more complicated than graphons, the question of equivalence for
hypergraphons (i.e., when they represent limits of the same sequences) is also more intricate. Elek and Szegedy define
for this purpose so-called structure preserving maps~\cite[\S 4.1]{ElS}, but since in this paper we adapt a different
(and, arguably, simpler) language, we defer further discussion until the next section in which we will formulate much more general Theorem~\ref{thm:TheonUniqueness}.

\section{Peons and theons}
\label{sec:peonstheons}

In this section we present our main definitions of peons and theons and
formulate the main results. It is very important from this point on that all
theories we are considering are canonical (Definition~\ref{def:canonical});
if we want to apply these notions to a non-canonical theory it should be
subdivided first as explained in Theorem~\ref{thm:canonical}. Otherwise,
although all our definitions are set up in such a way that formally they
work for non-canonical theories, the information about the behavior on the
diagonal will be completely lost.

\begin{definition}\label{def:peons}
  For a predicate symbol $P$ of arity $k$, a \emph{$P$-on} is a Lebesgue measurable subset of $\mathcal{E}_{k}$. We use the name \emph{peon} when we do not specify the predicate symbol $P$.

  Let now $\mathcal L$ be a language, as always finite and with predicate symbols only. An {\em Euclidean structure} in the language $\mathcal L$ is a function $\mathcal{N}$ that maps each predicate symbol $P\in\mathcal{L}$ to a $P$-on $\mathcal{N}_P\subseteq\mathcal{E}_{k(P)}$.

  If $M$ is an (ordinary) structure in the language $\mathcal L$ and $P\in\mathcal L$ then, in analogy with
  Definition~\ref{def:hypergraphons}, we let\footnote{We prefer to introduce a slightly different notation since $R_{P,M}\subseteq V(M)^{k(P)}$ was defined in Section~\ref{sec:model_theory} for arbitrary $\alpha$, not necessarily injective.}
  \begin{align*}
  R_P(M) & \df \set{\alpha\injection{[k(P)]}{V(M)}}{\alpha\in R_{P,M}}
  \intertext{and}
  \overline R_P(M) & \df \set{\alpha\injection{[k(P)]}{V(M)}}{\alpha\notin R_{P,M}}.
  \end{align*}

  Now, for an Euclidean structure $\mathcal N$ we give essentially the same chain of definitions as in Section~\ref{sec:hypergraphons}:
  \begin{align*}
    \Tinj(M,\mathcal{N})
    & \df
    \bigcap_{P\in\mathcal{L}}\bigcap_{\alpha\in R_P(M)}(\alpha^*)^{-1}(\mathcal{N}_P)
    \subseteq\mathcal{E}_{V(M)};
    \\
    \Tind(M,\mathcal{N})
    & \df
    \Tinj(M,\mathcal{N})\cap
    \bigcap_{P\in\mathcal{L}}\bigcap_{\alpha\in \overline{R}_{P}(M)}(\alpha^*)^{-1}(\mathcal E_{k(P)}\setminus \mathcal{N}_P);
    \\
    \tinj(M,\mathcal{N}) & \df \lambda(\Tinj(M,\mathcal{N}));
    \\
    \tind(M,\mathcal{N}) & \df \lambda(\Tind(M,\mathcal{N}));
    \\
    \phi_{\mathcal{N}}(M) \df p(M,\mathcal{N})
    & \df \frac{\lvert V(M)\rvert!}{\lvert\Aut(M)\rvert}\tind(M,\mathcal{N}).
  \end{align*}
\end{definition}

Again, the intuition behind these definitions is the same as in the
hypergraphon case: an ``induced copy'' of $M$ in $\mathcal{N}$ is a point
$x\in\mathcal{E}_{V(M)}$ such that for every $P\in\mathcal L$ and every tuple $\alpha\in R_{P}(M)$,
the induced function $\alpha^\ast$ maps $x$ to a point inside $\mathcal{N}_P$ (i.e., a point of
$\mathcal{N}$ that ``satisfies'' the predicate $P$) and for every tuple
$\alpha\in\overline{R}_{P}(M)$, the induced function $\alpha^\ast$ maps $x$ to a point outside
$\mathcal{N}_P$ (i.e., a point of $\mathcal{N}$ that ``falsifies'' $P$). A
(non-induced) copy of $M$ is again obtained by dropping the
$\overline{R}_{P}(M)$ requirements. This makes $\Tind(M,\mathcal{N})$ and
$\Tinj(M,\mathcal{N})$ correspond to the set of induced and
non-induced copies of $M$ in $\mathcal{N}$ respectively, except that we
totally ignore the values $P(v_1,\ldots,v_k)$ for which $v_1,\ldots,v_k$ are
not pairwise distinct.

\begin{definition} \label{def:theons}
  Let $T$ be a (canonical) theory in a language $\mathcal{L}$. A structure $M$ is {\em canonical} if it satisfies all axioms~\eqref{eq:canonical}, that is, the predicate $P(v_1,\ldots,v_k)$ is always false in $M$ whenever the tuple $(v_1,\ldots,v_k)$ contains repeated entries.
  A \emph{weak $T$-on} is an Euclidean structure $\mathcal{N}$ in $\mathcal L$ such that $\tind(M,\mathcal{N})= 0$ for every canonical structure $M$ that is {\em not} a model of $T$.

  The \emph{diagonal} of $\mathcal{E}_V$ is the closed set
  \begin{align*}
    \mathcal{D}_V & \df \{x\in\mathcal{E}_V \mid \exists i,j\in V(i\neq j\land x_{\{i\}}=x_{\{j\}})\}.
  \end{align*}
  Again we use $\mathcal{D}_k$ as a shorthand for $\mathcal{D}_{[k]}$.

  A \emph{strong $T$-on} is an Euclidean structure $\mathcal{N}$ such that
  \begin{align*}
    \Tind(M,\mathcal{N}) & \subseteq \mathcal{D}_{V(M)}
  \end{align*}
  for every canonical structure $M$ that is not a model of $T$.
We will use the name \emph{theon} when the theory $T$ is clear from the context.

  Finally, a theon $\mathcal{N}$ is {\em Borel} if $\mathcal{N}_P$ is a Borel set for every predicate symbol $P\in\mathcal{L}$
  in our language.
\end{definition}

Thus, Definition~\ref{def:theons} generalizes $k$-hypergraphons (which are precisely strong $T_{k\operatorname{-Hypergraph}}$-ons) in three different ways:
\begin{itemize}
\item The symmetry condition is removed (which leads to peons);

\item Different combinatorial structures on the same ground set can be combined together at no extra cost (this gives us Euclidean structures and weak theons);

\item The resulting object can even be assumed to fully retain the combinatorial structure possessed by ordinary models of $T$, except for the diagonal (strong theons).

      One good reason why we are not attempting to control the behavior on the diagonal are highly asymmetric theories like
      $T = T_{\operatorname{Tournament}}$. For example, both peons $\set{x\in\mathcal E_2}{x_1<x_2}$ and
      $\set{x\in\mathcal E_2}{x_1\leq x_2}$ are weak $T$-ons representing the limit of transitive tournaments from Example~\ref{ex:transitive}. Which of these two is the ``right'' strong $T$-on is completely arbitrary, and, as we said before,
      if for whichever reasons one needs to consider tournaments with loops, the ``right'' way of doing this is by appending to
      the language a separate unary predicate.
\end{itemize}

While the first item in the above list is more of cosmetic nature, the last two seem to be
somewhat novel, and their importance is clearly determined by
whether weak and strong theons can be shown to exist. So without further ado we formulate our central results addressing
that question.

\begin{theorem}[Induced Euclidean Removal Lemma]\label{thm:ierl}
  If $T$ is a theory in a language $\mathcal{L}$ and $\mathcal{N}$ is a weak $T$-on, then there exists a strong $T$-on $\mathcal{N}'$ such that
  \begin{equation} \label{eq:difference}
    \lambda(\mathcal{N}_P\symmdiff\mathcal{N}'_P) = 0,
  \end{equation}
  for every predicate symbol $P\in\mathcal{L}$.
\end{theorem}

Note that~\eqref{eq:difference} in particular implies that $p(M,\mathcal N)=p(M,\mathcal N')$ for every $M$, that is, the $T$-ons $\mathcal N$ and $\mathcal N'$ are indistinguishable in
the statistical framework.

The following theorem is a far-reaching generalization of Theorem~\ref{thm:hypergraphoncryptomorphism}.

\begin{theorem}[Existence]\label{thm:theoncryptomorphism}
  If $(N_n)_{n\in\mathbb{N}}$ is a convergent sequence of models of a theory $T$, then there exists a weak $T$-on $\mathcal{N}$ such that
  \begin{equation}\label{eq:theoncryptomorphism}
    \begin{aligned}
      \lim_{n\to\infty}p(M,N_n) & = p(M,\mathcal{N});\\
      \lim_{n\to\infty}\tind(M,N_n) & = \tind(M,\mathcal{N});\\
      \lim_{n\to\infty}\tinj(M,N_n) & = \tinj(M,\mathcal{N});
    \end{aligned}
  \end{equation}
  for every fixed model $M$ of $T$. Conversely, if $\mathcal{N}$ is a weak $T$-on, then there exists a convergent sequence $(N_n)_{n\in\mathbb{N}}$ of models of $T$ such that~\eqref{eq:theoncryptomorphism} holds for every fixed model $M$ of $T$.
\end{theorem}

In other words, for every canonical theory $T$, weak $T$-ons are
cryptomorphic to elements of $\HomT T$, and we will denote by $\phi_{\mathcal N}$ the
element of $\HomT T$ corresponding\footnote{A cumulative summary of all the
cryptomorphisms mentioned in the text will be given in Theorem~\ref{thm:allcrypto}.} to a $T$-on $\mathcal N$. Note that
Theorems~\ref{thm:ierl} and~\ref{thm:theoncryptomorphism} together imply a
similar conclusion for strong theons. One reason why we prefer to keep weak
theons as an intermediate step is that the ideas behind the proofs of Theorems~\ref{thm:ierl} and~\ref{thm:theoncryptomorphism}
are rather disjoint, and, moreover, the first one is not even
constructive -- we do not know if strong {\em Borel} theons always exist (we
will provide more comments on this in the next section).

In the classical model theory, if we want to check whether a given structure is a model of a theory $T$, it suffices to perform this check for axioms only. We now show that the same is true for theons, both weak and strong.

\begin{definition} \label{def:truth}
Let $\mathcal L$ be a language and $\mathcal N$ be an Euclidean structure in $\mathcal L$. For an open formula
$F(x_1,\ldots,x_n)$ in the language $\mathcal L$ we define its {\em interpretation}\footnote{We use the same letter $T$ for ``truth'' as in Definition~\ref{def:peons} in the hope that this will not create confusion.} $T(F,\mathcal N)\subseteq \mathcal E_n$ as follows:
\begin{enumerate}
\item if $F$ is $P(x_{i_1},\ldots,x_{i_k})$ and $i_1,\ldots,i_k$ are not pairwise distinct, or $F$ is $(x_i=x_j)$ with $i\neq j$ then $T(F,\mathcal N)\df\emptyset$;

\item $T(x_i=x_i,\mathcal N)\df \mathcal E_n$;

\item if $F$ is $P(x_{i_1},\ldots,x_{i_k})$ and $i_1,\ldots,i_k$ are pairwise distinct, then $T(F,\mathcal N)\df (i^\ast)^{-1}(\mathcal N_P)$, where $i$ is viewed as a function $i\injection{[k]}{[n]}$;
  \label{it:truthmain}

\item $T(F,\mathcal N)$ commutes with propositional connectives (e.g., we have $T(F_1\lor F_2,\mathcal{N}) \df T(F_1,\mathcal{N})\cup T(F_2,\mathcal{N})$).
\end{enumerate}
\end{definition}

\begin{remark} \label{rmk:straightforward}
  A straightforward but very useful observation is that if $M$ is a canonical structure with $V(M)=[m]$ and $\mathcal{N}$
  is an Euclidean structure, then $T(\Dopen(M),\mathcal{N}) = \Tind(M,\mathcal{N})$ and
    $T(\PDopen(M),\mathcal{N})= \Tinj(M,\mathcal{N})$. Hence Definition~\ref{def:truth}
    can be viewed as a generalization of these notions to arbitrary open formulas.
\end{remark}

\begin{definition}\label{def:subsclosed}
For an open formula $F(x_1,\ldots,x_n)$ and an equivalence relation
$\approx$ on $[n]$ with $m$ classes we let $F_\approx(y_1,\ldots,y_m)\df F(y_{\nu_1},\ldots,y_{\nu_n})$,
where $\nu_i$ is the equivalence class of $i$ (cf.~the proof of Theorem~\ref{thm:canonical}). A theory $T$ is
{\em substitutionally closed} if for every axiom $\forall \vec
x F(x_1,\ldots,x_n)$ and any equivalence relation $\approx$ on $[n]$, $T$ proves $\forall \vec y F_\approx(\vec y)$
using only propositional rules and, possibly, {\em renaming} variables in its axioms (thus substitutions of the
same variable for two different variables are disallowed).
\end{definition}

\begin{remark}\label{rmk:substclosed}
  Note that $\forall\vec y F_\approx(\vec y)$ is always entailed by $\forall\vec x F(\vec x)$. Hence,
  to be on the safe side one can always make a universal theory substitutionally closed by adding axioms to
  it; in other words, this is a property of a particular {\em axiomatization} rather than of the theory itself
  viewed as a set of theorems. On the other hand, if the intention is on the contrary to {\em rule out} non-trivial
  substitutional instances, then the simplest way to do it is by explicitly planting in the additional assumption
$\bigwedge_{i\neq j}(x_i\neq x_j)$ as we did in Example~\ref{ex:substitutional} (and, prior to that, in several
appropriate places in Section~\ref{sec:model_theory}).
\end{remark}

\begin{example} \label{ex:substitional2}
  All concrete canonical theories considered so far have been substitutionally closed. Slightly developing on Example~\ref{ex:substitutional},
let us also consider the axiom
\begin{equation} \label{eq:induced}
 \neg\of{\bigwedge_{i\in\mathbb Z_{2\ell}} E(x_i,x_{i+1})\land \bigwedge_{\substack{i\neq j\in\mathbb Z_{2\ell}\\ \lvert i-j\rvert\geq 2}}
\neg E(x_i,x_j)}
\end{equation}
forbidding {\em induced} copies of $C_{2\ell}$. Then $T\df T_{\operatorname{Graph}}+\text{\eqref{eq:induced}}_\ell$ is substitutionally closed
when $\ell\geq 3$. The reason is simple: $C_{2\ell}$ does not contain twin vertices and hence any
attempt at identifying a pair of variables immediately leads to a propositional tautology. This substitutionally closed theory
is not trivial: e.g.\ any blow-up of $K_3$ is a $T$-on.
On the contrary, the theory $T_{\operatorname{Graph}}+\text{\eqref{eq:induced}}_2$ is not substitutionally closed and remains trivial
(the theory of empty graphs from Example~\ref{ex:empty}).
\end{example}

\begin{theorem} \label{thm:characterization}
Let $T$ be a canonical substitutionally closed theory and $\mathcal N$ be an
Euclidean structure in the same language $\mathcal L$. Then $\mathcal N$ is a
weak {\rm [}strong{\rm ]} $T$-on if and only if for every axiom $\forall\vec
x F(x_1,\ldots,x_n)$ of the theory $T$ we have $\lambda(T(F,\mathcal
N))=1$ {\rm [}$T(F,\mathcal N) \supseteq \mathcal E_n\setminus\mathcal D_n$,
respectively{\rm ]}.
\end{theorem}
\begin{proof}
Note that for two different canonical structures $M$ and $M'$ on the same
vertex set, the sets $R_P(M)$ and $R_P(M')$ are different for at least one
$P\in\mathcal L$ and hence $\Tind(M,\mathcal N)$ and $\Tind(M',\mathcal N)$
are disjoint. Fix $n>0$ and let $\mathcal K_n$ be the set of all (labeled)
canonical structures on the vertex set $\{v_1,\ldots,v_n\}$. The above remark
readily implies that the sets $\set{\Tind(K,\mathcal N)}{K\in\mathcal K_n}$
form a (measurable) partition of $\mathcal E_n$. Now it is easy to prove, by a
straightforward induction on the construction of the formula $F$, that (cf.~Remark~\ref{rmk:straightforward})
\begin{equation} \label{eq:phi_expression}
  T(F,\mathcal N) = \mathop{\stackrel\cdot{\bigcup}}_{\substack{K\in\mathcal K_n \\ K\models F(v_1,\ldots,v_n)}}
  \Tind(K,\mathcal N).
\end{equation}
The ``only if'' part follows.

In the opposite direction, let $M$ be a canonical structure that is not a
model of $T$, i.e., we have $M\models \neg F(w_1,\ldots,w_m)$ for at least one
tuple $w_1,\ldots,w_m\in V(M)$ and an axiom $\forall \vec x F(\vec x)$ of $T$. Let $\approx$ be the equivalence
relation on
$[m]$ defined by $i\approx j$ iff $w_i=w_j$. Then $M\models
\neg F_\approx(v_1,\ldots,v_n)$ for {\em pairwise distinct}
$v_1,\ldots,v_n\in V(M)$ with $\{v_1,\ldots,v_n\} = \{w_1,\ldots,w_m\}$. Since $T$ is
substitutionally closed, we know that $F_\approx(y_1,\ldots,y_n)$ is a {\em propositional} consequence of some axioms
$A_1(y_1,\ldots,y_n),\ldots, A_r(y_1,\ldots,y_n)$ of the theory $T$, possibly up to renaming variables.
Let $N$ be the submodel of $M$ induced by $\alpha\df (v_1,\ldots,v_n)$, then by our assumption
we have $\lambda(T(A_i, \mathcal N))=1$ ($T(A_i, \mathcal N) \supseteq \mathcal E_n\setminus\mathcal D_n$ in the strong case).
Since $T$ commutes with propositional connectives (see Definition~\ref{def:truth}), we conclude that $\lambda(T(F_\approx, \mathcal N)) =1$
($T(F_\approx, \mathcal N)\supseteq\mathcal E_n\setminus\mathcal D_n$ in the strong case). Applying~\eqref{eq:phi_expression} to $F_\approx$ and noting that $N$
does not appear in the union, we see that $\lambda(T(N,\mathcal N))=0$ ($\Tind(N,\mathcal N)\subseteq\mathcal D_n$ in the strong case).
It only remains to note that according to our
definitions, we have $\Tind(M,\mathcal N)\subseteq
(\alpha^\ast)^{-1}(\Tind(N,\mathcal N))$ and $(\alpha^\ast)^{-1}(\mathcal D_n)\subseteq\mathcal D_{V(M)}$.
\end{proof}

\begin{example}
The restriction of being substitutionally closed is essential. Indeed, the exceptional theory $T_{\operatorname{Graph}}+\text{\eqref{eq:induced}}_2$ in
Example~\ref{ex:substitional2} is, as we observed, a peculiar axiomatization of the theory of empty graphs.
On the other hand, the second assumption in Theorem~\ref{thm:characterization} is satisfied by the complete graphon $\mathcal N=\mathcal E_2$.
\end{example}

\begin{remark}\label{rmk:theoninterpret}
Another application of this construction is that it easily allows us to
define the action of open interpretations on theons. Namely, let
$I\interpret{T_1}{T_2}$ be such an interpretation, where $T_\nu$ is in the
language $\mathcal L_\nu$, and let $\mathcal N$ be a $T_2$-on (weak or strong). For
every $P\in\mathcal L_1$, $I(P)$ is an open formula in the language $\mathcal L_2$ and thus
we may form a $P$-on $T(I(P),\mathcal N)\subseteq
\mathcal E_{k(P)}$ according to Definition~\ref{def:truth}. Then the Euclidean structure made by these $P$-ons
for $P\in\mathcal L_1$ is a $T_1$-on (weak or strong) that will be denoted by
$I(\mathcal N)$ and satisfies $\phi_{I(\mathcal{N})} = \phi_{\mathcal{N}}\comp\pi^I$ (cf.~Theorem~\ref{thm:flagpi}). The proof goes along the same lines as the proof of Theorem~\ref{thm:characterization}.
\end{remark}

Before we proceed to the rather technical statement of the uniqueness theorem, let us provide some intuition for operations that preserve densities of submodels in a theon.

For simplicity, let us consider the case of a single predicate $P$ of arity $3$. In Theorem~\ref{thm:graphon_uniqueness}
for graphons, we have seen one example of such an operation, ``permuting vertices''. Namely, let  $f_1\function{[0,1]}{[0,1]}$
be an arbitrary measure preserving function. If we let
\begin{align*}
  \mathcal{N}'
  & =
  \{x\in\mathcal{E}_3 \mid
  (f_1(x_{\{1\}}),f_1(x_{\{2\}}),f_1(x_{\{3\}}),x_{\{1,2\}},x_{\{1,3\}},x_{\{2,3\}},x_{\{1,2,3\}})\in\mathcal{N}\},
\end{align*}
then $\mathcal{N}$ and $\mathcal{N}'$ represent the same limit object (i.e., we have
$\phi_{\mathcal{N}}=\phi_{\mathcal{N}'}$). This is a complete triviality.

It is equally clear that in the same manner we can ``permute'' the variables indexed by sets of higher cardinalities.
Say, for a measure preserving function $f_2\function{[0,1]}{[0,1]}$, the $P$-on
\begin{align*}
  \mathcal{N}''
  & =
  \left\{x\in\mathcal{E}_3 \mid
  (x_{\{1\}},x_{\{2\}},x_{\{3\}},f_2(x_{\{1,2\}}),f_2(x_{\{1,3\}}),f_2(x_{\{2,3\}}),x_{\{1,2,3\}})
  \in\mathcal{N}\right\}
\end{align*}
also represents the same limit object as $\mathcal N$.

Let us now do something slightly more interesting and allow $f_2$ to depend on the vertices.
That is, we take a measurable function $f_2\function{\mathcal{E}_2}{[0,1]}$ such that for
every $(x_{\{1\}},x_{\{2\}})\in[0,1]^2$ the function $x_{\{1,2\}}\mapsto f_2(x_{\{1\}},x_{\{2\}},x_{\{1,2\}})$ is measure preserving and define
\begin{align*}
  \mathcal{N}^{(3)}
  =
  \Bigl\{x\in\mathcal{E}_3 \Big\vert
  & \bigl(
  x_{\{1\}},x_{\{2\}},x_{\{3\}},
  \\ & f_2(x_{\{1\}},x_{\{2\}},x_{\{1,2\}}),f_2(x_{\{1\}},x_{\{3\}},x_{\{1,3\}}),f_2(x_{\{2\}},x_{\{3\}},x_{\{2,3\}}),
  \\ & x_{\{1,2,3\}}\bigr)
  \in\mathcal{N}\Bigr\}.
\end{align*}

Then we will already need a consistency condition that, as it turns out, simply amounts to requiring that $f_2$
is symmetric. The reason is best illustrated by the following simple example; remarkably, the symmetry condition
is mostly needed when the underlying predicates are highly asymmetric.

\begin{example}
Let us for a moment switch from $\mathcal E_3$ to $\mathcal E_2$, i.e., to ordinary digraphons (cf.~\cite{DiJa}).
Then $\mathcal N \df \set{x\in\mathcal E_2}{x_{\{1,2\}}\leq 1/2}$ describes a random graph (viewed as a model of
$T_{\operatorname{Digraph}}$ in which a graph edge is replaced by anti-parallel edges of the digraph), while the digraphon $\mathcal N' \df \set{x\in\mathcal E_2}{x_{\{1,2\}}\leq 1/2 \equiv x_{\{1\}}\leq x_{\{2\}}}$ corresponds to
a random tournament. These are totally different combinatorial objects.

Nonetheless, the (non-symmetric) function
$$
f_2(x_{\{1\}},x_{\{2\}},x_{\{1,2\}}) =
\begin{dcases*}
  x_{\{1,2\}} & if $x_{\{1\}}\leq x_{\{2\}}$;\\
  1 - x_{\{1,2\}} & if $x_{\{1\}} > x_{\{2\}}$
\end{dcases*}
$$
maps $\mathcal N$ to $\mathcal N'$ a.e.\ and vice versa.
\end{example}

Naturally, we can go one step further and mix all these ``permutations'' as follows. If
$f_d\function{\mathcal{E}_d}{[0,1]}$ for $d=1,2,3$, then the $P$-on
\begin{equation}\label{eq:3peon}
  \begin{aligned}
    \mathcal{N}^{(4)}
    =
    \Bigl\{x\in\mathcal{E}_3 \Big\vert
    & \bigl(
    f_1(x_{\{1\}}),f_1(x_{\{2\}}),f_1(x_{\{3\}}),
    \\ & f_2(x_{\{1\}},x_{\{2\}},x_{\{1,2\}}),f_2(x_{\{1\}},x_{\{3\}},x_{\{1,3\}}),f_2(x_{\{2\}},x_{\{3\}},x_{\{2,3\}}),
    \\ & f_3(x_{\{1\}},x_{\{2\}},x_{\{3\}},x_{\{1,2\}},x_{\{1,3\}},x_{\{2,3\}},x_{\{1,2,3\}})\bigr)
    \in\mathcal{N}\Bigr\}
  \end{aligned}
\end{equation}
represents the same limit object as $\mathcal{N}$ as long as each $f_d$ is $S_d$-invariant and is measure preserving on
the highest order argument.

Finally, let us note that in general there may not exist any measure preserving transformation $f$ taking $\mathcal N$ to
$\mathcal N'$ directly. The following example is paradigmatic in this respect.

\begin{example} \label{ex:twoorderedtheons}
Recall from Example~\ref{ex:orders} that $T\df T_{\operatorname{LinOrder}}$ has only one model of each size. This
immediately implies that $\HomT{T}$ has only one element, hence all $T$-ons represent this unique limit object.
However, it is straightforward to check that for the $T$-ons
  \begin{align*}
    \mathcal{N} & \df \{x\in\mathcal{E}_2 \mid x_{\{1\}}\bmod (1/2) < x_{\{2\}} \bmod (1/2)\};\\
    \mathcal{N}' & \df \{x\in\mathcal{E}_2 \mid x_{\{1\}}\bmod (1/3) < x_{\{2\}} \bmod (1/3)\}
  \end{align*}
  there is no family of symmetric measure preserving functions $f$ taking one into another (see Figures~\ref{subfig:1/2} and~\ref{subfig:1/3}).

  The remedy is to employ a ``middle theon'' (cf.~\cite[Theorem~3.10]{Lov4}).
  Let $\mathcal N_0 \df \{x\in\mathcal{E}_2 \mid x_{\{1\}}\bmod (1/6) < x_{\{2\}} \bmod (1/6)\}$ (Figure~\ref{subfig:1/6})
  and consider measure preserving functions $f_1(x) = (3x)\bmod 1$,\ $g_1(x) = (2x)\bmod 1$. Then, suppressing the dummy argument
  $x_{\{1,2\}}$, we have
  $$
  x\in\mathcal N_0 \equiv (f_1(x_{\{1\}}), f_1(x_{\{2\}})) \in \mathcal N
  \equiv (g_1(x_{\{1\}}), g_1(x_{\{2\}})) \in \mathcal N'.
  $$

  Dually, and, perhaps, more naturally, we could instead ``flatten out'' the structure and consider the ``standard model'' $\mathcal N_1 \df \{x\in\mathcal{E}_2 \mid x_{\{1\}} < x_{\{2\}}\}$ (Figure~\ref{subfig:1}). Then, employing the same functions $f_1,g_1$ as above, we would have
  \begin{eqnarray*}
  x\in\mathcal N &\equiv& (g_1(x_{\{1\}}),g_1(x_{\{2\}})) \in \mathcal N_1;\\
  x\in\mathcal N' &\equiv& (f_1(x_{\{1\}}),f_1(x_{\{2\}})) \in \mathcal N_1.
  \end{eqnarray*}
  We would like to note, however, that we do not know how to extend this second approach to the general situation.
\end{example}

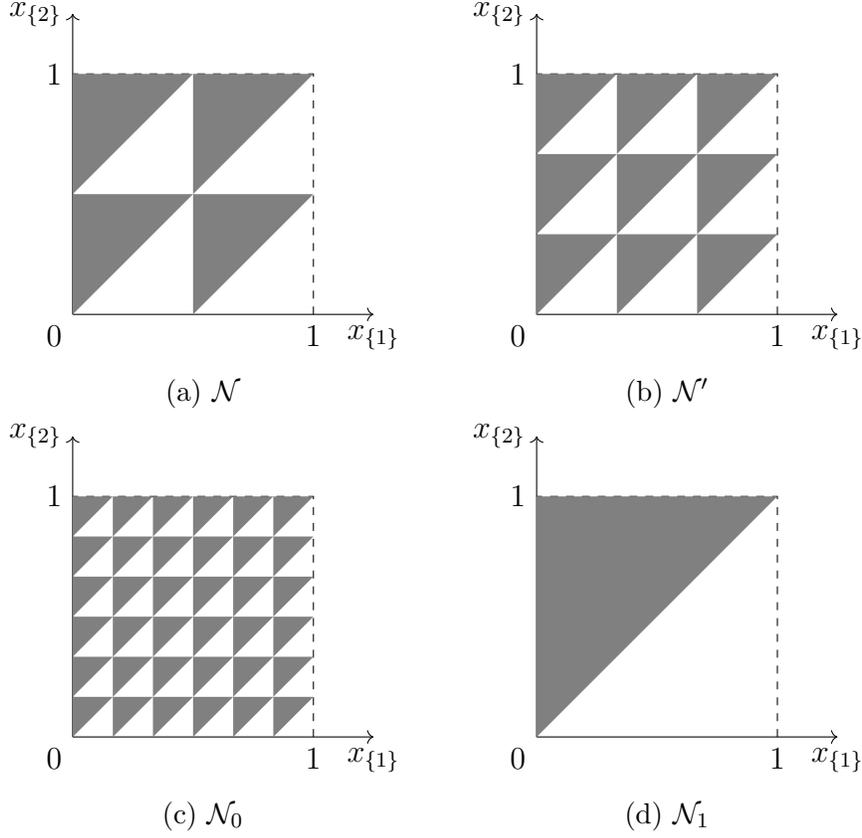
\begin{figure}[ht]
  \begin{center}
    \begingroup

\def\basesize{3.2}
\def\axissize{4}

\def\produceImage#1#2#3{%
  \begin{subfigure}[b]{0.4\textwidth}
    \begin{center}
      \begin{tikzpicture}
        \foreach \i in {0,...,#1}{
          \foreach \j in {0,...,#1}{
            \pgfmathsetmacro{\mult}{\basesize / #1};
            \coordinate (\i\j) at ($\mult*(\i,\j)$);
          }
        }

        \draw[dashed] (\basesize,0) -- (\basesize,\basesize) -- (0,\basesize);
        \draw[->] (0,0) -- (\axissize,0);
        \draw[->] (0,0) -- (0,\axissize);

        \node[below] at (\axissize,0) {$x_{\{1\}}$};
        \node[left] at (0,\axissize) {$x_{\{2\}}$};

        \node[below left] at (0,0) {$0$};
        \node[below] at (\basesize,0) {$1$};
        \node[left] at (0,\basesize) {$1$};

        \foreach \i in {1,...,#1}{
          \pgfmathtruncatemacro{\pi}{\i - 1};
          \foreach \j in {1,...,#1}{
            \pgfmathtruncatemacro{\pj}{\j - 1};

            \fill[gray] (\pi\pj) -- (\pi\j) -- (\i\j) -- cycle;
          }
        }

      \end{tikzpicture}
      \caption{$#2$}
      \label{#3}
    \end{center}
  \end{subfigure}
}

\produceImage{2}{\mathcal{N}}{subfig:1/2}
\quad
\produceImage{3}{\mathcal{N}'}{subfig:1/3}

\produceImage{6}{\mathcal{N}_0}{subfig:1/6}
\quad
\produceImage{1}{\mathcal{N}_1}{subfig:1}

\endgroup
    \caption{Projections of $T_{\operatorname{LinOrder}}$-ons of Example~\ref{ex:twoorderedtheons} on the coordinates $x_{\{1\}}$ and $x_{\{2\}}$.}
    \label{fig:twoorderedtheons}
  \end{center}
\end{figure}

Let us now proceed to formal definitions and statements.

\begin{definition}\label{def:hoa}
  For a finite set $V$, let $r(V)^* \df r(V)\setminus\{V\}$ and let $\mathcal{E}_V^* \df [0,1]^{r(V)^*}$. Again, as a shorthand, when $V=[k]$, we will write $r(k)^*$ and $\mathcal{E}_k^*$ instead of $r([k])^*$ and $\mathcal{E}_{[k]}^*$.

  Let $f\function{\mathcal{E}_V}{[0,1]}$. The function $f$ is said to be \emph{symmetric} if it is invariant under the action of $S_V$. Furthermore, the function $f$ is said to be \emph{measure preserving on the
  highest order argument} (h.o.a.) if it is measurable and for every $x^*\in\mathcal{E}_V^*$, the function
  \begin{align*}
    \begin{tabular}{>{$\displaystyle}r<{$}>{$\displaystyle}c<{$}>{$\displaystyle}l<{$}}
      [0,1]\cong [0,1]^{\{V\}} & \longrightarrow & [0,1]\\
      y & \longmapsto & f(x^*, y)
    \end{tabular}
  \end{align*}
  is measure preserving.

  Suppose now that $f = (f_1,\ldots,f_k)$ is a family of symmetric functions with $f_d\function{\mathcal{E}_d}{[0,1]}$. Then we define a new sequence $\widehat f =(\widehat f_1,\ldots,\widehat f_k)$ with
  $\widehat f_d\function{\mathcal E_d}{\mathcal E_d}$ by
  \begin{align*}
    \begin{functiondef}
      \widehat f_d(x)_A \df f_{\lvert A\rvert}(\alpha_A^\ast(x))\quad (A\in r(d)),
    \end{functiondef}
  \end{align*}
  where $\alpha_A\injection{[\lvert A\rvert]}{[d]}$ is a fixed injection with $\im(\alpha_A)=A$. This definition is independent of the choice of $\alpha_A$ since the function $f_{\lvert A\rvert}$ is symmetric, although in practice it is always convenient to take as $\alpha_A$ the enumeration of $A$ in the increasing order. As a consequence, $\widehat f_d$ is $S_d$-equivariant, and it is straightforward to check that all $\widehat f_d$'s are measure preserving in the ordinary sense. Note also that the diagram
  $$
  \begin{tikzcd}
    \mathcal E_D\arrow[r, "\widehat f_D"]\arrow[d, "\beta^\ast", swap] & \mathcal E_D\arrow[d, "\beta^\ast"]\\
    \mathcal E_d\arrow[r, "\widehat f_d"] & \mathcal E_d
  \end{tikzcd}
  $$
is commutative, where $1\leq d\leq D\leq k$ and $\beta\injection{[d]}{[D]}$ is an arbitrary injection. Hence (since $\beta^\ast$ is surjective) $\widehat f_1,\ldots,\widehat f_{k-1}$ are in principle completely determined by $\widehat f_k$. It is, however, more handy to keep all of them in the notation.
\end{definition}

\begin{theorem}[Uniqueness, first form]\label{thm:TheonUniqueness}
  Let $T$ be a canonical theory in a language $\mathcal{L}$, let $k\df\max\{k(P) \mid P\in\mathcal{L}\}$, and
  let $\mathcal{N}$ and $\mathcal{N}'$ be two weak $T$-ons. The following are equivalent.
  \begin{enumerate}[label={\arabic*.}, ref={\arabic*)}]
  \item We have $\phi_{\mathcal{N}}=\phi_{\mathcal{N}'}$, that is $\mathcal N$ and $\mathcal N'$ give
  rise to the same element of $\HomT T$;
    \label{it:densityequiv}
  \item There exist families $f = (f_1,\ldots, f_k)$ and $g = (g_1,\ldots,g_k)$ of symmetric functions
  measure preserving on h.o.a., $f_d\function{\mathcal{E}_d}{[0,1]}$ and $g_d\function{\mathcal{E}_d}{[0,1]}$ and a weak $T$-on $\mathcal N''$
with the property
    $$
     x\in\mathcal N_P'' \equiv \widehat{f}_{k(P)}(x)\in\mathcal{N}_P \equiv \widehat{g}_{k(P)}(x)\in\mathcal{N}'_P,
    $$
     for every $P\in \mathcal L$ and almost every $x\in\mathcal{E}_{k(P)}$.
     \label{it:fgequiv}
  \end{enumerate}
\end{theorem}

\begin{remark}\label{rmk:graphon_uniqueness}
  Upon closer inspection of graphon uniqueness (Theorem~\ref{thm:graphon_uniqueness}), the reader may have noticed that there is no analogue of the functions $f_2,g_2\function{\mathcal{E}_2}{[0,1]}$. The reason comes from the way that we represent $2$-hypergraphons as graphons (cf.~Section~\ref{sec:hypergraphons}): a $2$-hypergraphon $\mathcal{H}$ corresponds to the graphon $W(u,v) = \lambda(\{p\in[0,1]\mid (u,v,p)\in\mathcal{H}\})$ and since $f_2$ and $g_2$ are measure preserving on h.o.a., these functions do not affect $W$.
\end{remark}

Finally, let us present a slightly stronger but somewhat more technical version that will turn out to be
useful in Section~\ref{sec:otherobjects} (cf.~\cite[Theorem~7.1(vi)]{DiJa}). For an
intuition, note that the choice of $[0,1]$ as the probability space on which the intermediate $T$-on $\mathcal N''$ in Theorem~\ref{thm:TheonUniqueness} lives is rather arbitrary; we will further elaborate on this point in Section~\ref{sec:otherobjects}.
In particular, we can take as its ground space the square $\Omega=[0,1]^2$. Then the stronger version essentially says that
one of the two functions $f,g$ can be taken as (or, rather, induced from) the projection $\Omega\to [0,1]$.

\begin{definition}[Definition~\ref{def:hoa}, cntd.]\label{def:hoa2}
   Consider the product action of $S_V$ on $\mathcal{E}_V\times\mathcal{E}_V$ and
  let $h\function{\mathcal{E}_V\times\mathcal{E}_V}{[0,1]}$. Analogously to the previous case, the function $h$ is said to be \emph{symmetric} if it is invariant under the action of $S_V$. Furthermore, the function $h$
  is \emph{measure preserving on the highest order argument} (h.o.a.) if it is measurable and for every $(x^*, \widehat x^*) \in\mathcal{E}_V^*\times\mathcal{E}_V^*$, the function
  \begin{align*}
    \begin{tabular}{>{$\displaystyle}r<{$}>{$\displaystyle}c<{$}>{$\displaystyle}l<{$}}
      [0,1]^2\cong [0,1]^{\{V\}}\times[0,1]^{\{V\}} & \longrightarrow & [0,1]\\
      (y,\widehat y) & \longmapsto & h((x^*, y), (\widehat x^*, \widehat y))
    \end{tabular}
  \end{align*}
  is measure preserving.

  Likewise, if $h = (h_1,\ldots,h_k)$ is a family of symmetric functions with $h_d\function{\mathcal{E}_d\times\mathcal{E}_d}{[0,1]}$,
  then we define the tuple $\widehat h = (\widehat h_1,\ldots, \widehat h_k)$;  $\widehat h_d\function{\mathcal E_d\times\mathcal E_d}{\mathcal E_d}$ by
  \begin{align*}
    \begin{functiondef}
      \widehat h_d(x,\widehat x)_A \df h_{\lvert A\rvert}(\alpha_A^*(x), \alpha_A^*(\widehat x))\quad (A\in r(d)),
    \end{functiondef}
  \end{align*}
  where $\alpha_A$ is as before.
\end{definition}

\begin{theorem}[Uniqueness, second form]\label{thm:TheonUniquenessSecond}
  Let $T$ be a canonical theory in a language $\mathcal{L}$, let $k\df\max\{k(P) \mid P\in\mathcal{L}\}$, and
  let $\mathcal{N}$ and $\mathcal{N}'$ be two $T$-ons. The following are equivalent.
  \begin{enumerate}[label={\arabic*.}, ref={\arabic*)}]
  \item We have $\phi_{\mathcal{N}}=\phi_{\mathcal{N}'}$, that is $\mathcal N$ and $\mathcal N'$ give rise to the same element of $\HomT T$;
  \item There exists a family $h = (h_1,\ldots, h_k)$ of symmetric functions measure preserving on h.o.a.,
  $h_d\function{\mathcal{E}_d\times\mathcal{E}_d}{[0,1]}$ such that
    $$
     x\in\mathcal{N}_P\equiv \widehat{h}_{k(P)}(x,\widehat x)\in\mathcal{N}'_P,
    $$
    for every predicate symbol $P\in\mathcal{L}$ and for almost every $(x,\widehat x)\in\mathcal{E}_{k(P)}\times\mathcal{E}_{k(P)}$.
    \label{it:hequiv}
  \end{enumerate}
\end{theorem}

\begin{example}\label{ex:uniquenessnecessary}

  In the notation of Example~\ref{ex:twoorderedtheons}, we can set
  \begin{align*}
    h_1(x,\widehat x) & = \frac{(2x)\bmod 1}{3} + \frac{\lfloor 3\widehat x\rfloor}{3},
  \end{align*}
  which gives
  \begin{align*}
    x\in\mathcal{N} & \equiv\widehat{h}_2(x,\widehat x)\in\mathcal{N}',
  \end{align*}
  for almost every $(x,\widehat x)\in\mathcal{E}_2\times\mathcal{E}_2$.

  On the other hand, setting
  \begin{align*}
    h'_1(x,\widehat x) & = \frac{(3x)\bmod 1}{2} + \frac{\lfloor 2\widehat x\rfloor}{2}
  \end{align*}
  gives
  \begin{align*}
    \widehat{h}'_2(x,\widehat x)\in\mathcal{N} & \equiv x\in\mathcal{N}',
  \end{align*}
  for almost every $(x,\widehat x)\in\mathcal{E}_2\times\mathcal{E}_2$.
\end{example}

\section{Euclidean removal lemmas}
\label{sec:removal}

As a warm-up, we begin with proving a (much simpler and constructive)
version of Theorem~\ref{thm:ierl} for almost Horn theories, which we define below.

\begin{definition}\label{def:horn}
  A {\em literal} is either an atomic formula ({\em positive} literal) or its negation ({\em negative} literal). An \emph{almost Horn clause} is a disjunction of literals with at most one positive literal not involving equality\footnote{Thus, the difference with a Horn clause is that we allow any number of positive literals based on equality. For example, the formula $P(x,y)\lor\neg Q(x)\lor\neg R(y,z)\lor x=y\lor y=z\lor x\neq z$ is an almost Horn clause but not a Horn clause.}.

  Let us call a canonical theory an \emph{almost Horn theory} if all of its axioms are almost Horn clauses (note that the canonicity axioms~\eqref{eq:canonical} are equivalent to Horn clauses).
\end{definition}

By using variable substitution and renaming and arguments similar to Theorem~\ref{thm:canonical} we can re-axiomatize any almost Horn theory to have only three types of axioms.
\begin{enumerate}
\item \emph{Fact clauses}, which are of the form
  \begin{align}\label{eq:HornFact}
    \bigwedge_{1\leq i < j\leq n} x_i \neq x_j
    \to Q(x_1,\ldots,x_{k(Q)}),
  \end{align}
  where $k(Q)\leq n$.
\item \emph{Definite clauses}, which are of the form
  \begin{align}\label{eq:HornDefinite}
    \bigwedge_{1\leq i < j\leq n} x_i \neq x_j
    \land \bigwedge_{t=1}^T P_t(x_{i_{t,1}},\ldots,x_{i_{t,k(P_t)}})
    \to Q(x_1,\ldots,x_{k(Q)}),
  \end{align}
  where $T > 0$; $k(Q), i_{t,j}\leq n$ and for any $t$, $i_{t,1},\ldots, i_{t,k(P_t)}$ are pairwise distinct.
\item \emph{Goal clauses}, which are of the form
  \begin{align}\label{eq:HornGoal}
    \neg\left(\bigwedge_{1\leq i < j\leq n} x_i \neq x_j
    \land \bigwedge_{t=1}^T P_t(x_{i_{t,1}},\ldots,x_{i_{t,k(P_t)}})\right),
  \end{align}
  where $T > 0$; $i_{t,j}\leq n$ and for any $t$, $i_{t,1},\ldots, i_{t,k(P_t)}$ are pairwise distinct.
\end{enumerate}
Note that this axiomatization makes the theory substitutionally closed (cf.~Definition~\ref{def:subsclosed}): every non-trivial
substitution will trivialize all the axioms due to the presence of the term $\bigwedge_{1\leq i < j\leq n} x_i \neq x_j$.

\begin{example}\label{ex:Horn}
  Up to re-axiomatization, the theories $T_{\operatorname{Graph}}$, $T_{k\operatorname{-Hypergraph}}$, $T_{\operatorname{Order}}$ and $T_{\operatorname{EqRel}}$ are almost Horn theories. Furthermore, any theory obtained from an almost Horn theory $T$ by forbidding {\em non-induced} models (i.e., by adding goal clauses of the form $\neg\PDopen(M)$ for some $M\in\mathcal{M}[T]$) is also an almost Horn theory.
\end{example}

\begin{definition} \label{def:density}
  Let $A\subseteq [0,1]^d$ be a Lebesgue measurable set. A point $x\in [0,1]^d$ is a \emph{Lebesgue density point of $A$} if
  \begin{align}\label{eq:density}
    \lim_{r\to 0^+}\frac{\lambda(B(x,r)\cap A)}{\lambda(B(x,r)\cap [0,1]^d)} & = 1,
  \end{align}
  where $B(x,r)$ denotes the $\ell_\infty$-ball\footnote{In fact, one can use other norms to define Lebesgue density points and
  get an a.e.\ equivalent definition, but for us it will be slightly more convenient to use the $\ell_\infty$-norm.} of radius $r$ centered in $x$.

  We will denote the set of all Lebesgue density points of $A$ by $D(A)$.
\end{definition}

The property of Lebesgue density points below says that almost every point of a Lebesgue
measurable set is a density point of it and almost every point of its complement  is not a density point (see e.g.~\cite[I-5.8(ii)]{Bog} or~\cite[Theorem~3.21]{Oxt}).

\begin{proposition}\label{prop:Ldensityae}
  If $A$ is a Lebesgue measurable subset of $[0,1]^d$ then $D(A)$ is a Borel set such that $\lambda(A\symmdiff D(A))=0$ and $D(D(A))=D(A)$.
\end{proposition}

\begin{theorem}[Horn Euclidean Removal Lemma]\label{thm:herl}
  Let $T$ be an almost Horn theory in a language $\mathcal{L}$. If $\mathcal{N}$ is a weak $T$-on, then setting
  \begin{equation} \label{eq:npprime}
    \mathcal{N}'_P  \df D(\mathcal{N}_P)
  \end{equation}
  for every predicate symbol $P\in\mathcal{L}$ yields a strong Borel $T$-on $\mathcal{N}'$ such that
  \begin{align}\label{eq:difference-Horn}
    \lambda(\mathcal{N}_P\symmdiff\mathcal{N}'_P) & = 0,
  \end{align}
  for every predicate symbol $P\in\mathcal{L}$. In particular, we have $\phi_{\mathcal{N}} = \phi_{\mathcal{N}'}$.
\end{theorem}

\begin{proof}
  By our previous observations, we may re-axiomatize $T$ to be substitutionally closed and only have axioms of the forms~\eqref{eq:HornFact},~\eqref{eq:HornDefinite} and~\eqref{eq:HornGoal}.

  From Proposition~\ref{prop:Ldensityae}, $\mathcal{N}'$ satisfies~\eqref{eq:difference-Horn}, which in particular implies that $\mathcal{N}'$ is a Borel $T$-on satisfying $\phi_{\mathcal{N}}=\phi_{\mathcal{N}'}$; it only remains to prove that it is strong. By Theorem~\ref{thm:characterization}, it is enough to show that $T(F,\mathcal{N}')\supseteq\mathcal{E}_n\setminus\mathcal{D}_n$ for every axiom $\forall\vec{x}F(x_1,\ldots,x_n)$.

  Consider first a fact clause $F$ of the form~\eqref{eq:HornFact}. Since $T(F,\mathcal{N}) = \mathcal{N}_Q$, by the weak version of Theorem~\ref{thm:characterization}, we have $\lambda(\mathcal{N}_Q) = 1$, which implies
  $T(F,\mathcal{N}') = \mathcal{N}'_Q = D(\mathcal{N}_Q) = \mathcal{E}_{k(Q)}$.

  \smallskip

  Consider now a definite clause $F$ of the form~\eqref{eq:HornDefinite}. For every $t\in[T]$, let $\alpha_t\injection{[k(P_t)]}{[n]}$ be given by $\alpha_t(j) = i_{t,j}$. Let also $\iota\injection{[k(Q)]}{[n]}$ be the natural inclusion. Then we have
  \begin{align*}
    T(F,\mathcal{N}')
    & =
    (\iota^*)^{-1}(\mathcal{N}'_Q)\cup
    \left(\mathcal{E}_n\middle\backslash\bigcap_{t=1}^T (\alpha_t^*)^{-1}(\mathcal{N}'_{P_t})\right).
  \end{align*}
  Hence it is enough to show that
  $\bigcap_{t=1}^T(\alpha_t^*)^{-1}(\mathcal{N}'_{P_t})\setminus \mathcal D_n \subseteq(\iota^*)^{-1}(\mathcal{N}'_Q)$.

  Fix then $z$ in the first set and let $G \df \{A\in r(n) \mid 0 < z_A < 1\}$. Fix also $\epsilon > 0$ and let $r_0 > 0$ be small enough such that for every $A\in G$ we have $(z_A - r_0, z_A + r_0)\subseteq[0,1]$ and for every $r\in(0,r_0)$ and every $t\in[T]$ we have
  \begin{align*}
    \frac{\lambda(B(\alpha_t^*(z),r)\cap\mathcal{N}_{P_t})}{\lambda(B(\alpha_t^*(z),r)\cap\mathcal{E}_{k(P_t)})}
    & \geq
    1 - \frac{\epsilon}{T}.
  \end{align*}
  This inequality scales\footnote{This is precisely why we prefer to use the $\ell_\infty$-norm: it behaves exceptionally
  well with respect to projections.} to $\mathcal E_n$ as
  \begin{align*}
    \frac{\lambda(B(z,r)\cap(\alpha_t^*)^{-1}(\mathcal{N}_{P_t}))}{\lambda(B(z,r)\cap\mathcal{E}_n)}
    & \geq
    1 - \frac{\epsilon}{T}.
  \end{align*}
  (The denominator in the above is equal to $r^{2^n-1}\cdot 2^{\lvert G\rvert}$.)

  By the union bound, it follows that
  \begin{align}\label{eq:union_bound}
    \lambda\left(B(z,r)\cap\bigcap_{t=1}^T(\alpha_t^*)^{-1}(\mathcal{N}_{P_t})\right)
    & \geq
    (1 - \epsilon)\lambda(B(z,r)\cap\mathcal{E}_n),
  \end{align}
  and from the weak version of~\eqref{thm:characterization} for $T(F,\mathcal{N})$, we get
  \begin{align*}
    \lambda(B(z,r)\cap(\iota^*)^{-1}(\mathcal{N}_Q))
    & \geq
    (1 - \epsilon)\lambda(B(z,r)\cap\mathcal{E}_n),
  \end{align*}
  hence
  \begin{align*}
    \frac{\lambda(B(\iota^*(z),r)\cap\mathcal{N}_Q)}{\lambda(B(\iota^*(z),r)\cap\mathcal{E}_{k(Q)})}
    & \geq
    1 - \epsilon.
  \end{align*}
  As $\epsilon>0$ was arbitrary, this implies $z\in (\iota^*)^{-1}(D(\mathcal{N}_Q)) = (\iota^*)^{-1}(\mathcal{N}'_Q)$ as desired.

  \smallskip

  Finally, consider a goal clause $F$ of the form~\eqref{eq:HornGoal}, define $\alpha_t$ as in the previous case and let again $z\in \bigcap_{t=1}^T(\alpha_t^*)^{-1}(\mathcal{N}'_{P_t})\setminus \mathcal D_n$.
Repeating the first part of the previous argument, we get~\eqref{eq:union_bound}.
However, this time since $T(F,\mathcal{N}) = \mathcal{E}_n\setminus\bigcap_{t=1}^T (\alpha_t^*)^{-1}(\mathcal{N}_{P_t})$, the weak version of Theorem~\ref{thm:characterization} implies $\lambda(\bigcap_{t=1}^T(\alpha_t^*)^{-1}(\mathcal{N}_{P_t}))=0$, a contradiction.
\end{proof}

Following up on Example~\ref{ex:Horn}, we have the following corollary.

\begin{corollary}[Non-induced Euclidean Removal Lemma]\label{cor:erl}
  Let $T$ be a theory of the form $\Forbp{T_{\operatorname{Pure}}}{\mathcal F}$,
   where $T_{\operatorname{Pure}}$ is the pure canonical theory in the underlying language $\mathcal L$ with the set of axioms~\eqref{eq:canonical}. If $\mathcal{N}$ is a weak $T$-on, then there exists a strong Borel $T$-on $\mathcal{N}'$ such that
  \begin{align*}
    \lambda(\mathcal{N}'_P\symmdiff\mathcal{N}_P) = 0
  \end{align*}
  for every predicate symbol $P\in\mathcal{L}$.
\end{corollary}

\begin{proof}
  Since $T$ is almost Horn, this is a partial case of Theorem~\ref{thm:herl}.
\end{proof}

For completeness, let us also explicitly state the dual of Corollary~\ref{cor:erl}.

Let us call a canonical theory $T$ {\em positive} if all its axioms $\forall\vec x F(\vec x)$ different from~\eqref{eq:canonical} are positive, that is any occurrence of an atomic formula is in the scope of an even number
of negations.

\begin{corollary}[Positive Euclidean Removal Lemma] \label{cor:perl} Let $T$ be a positive theory in a language $\mathcal L$. If $\mathcal N$ is a weak $T$-on then there exists a strong Borel $T$-on $\mathcal N'$ such that
$$
\lambda(\mathcal N_P'\symmdiff \mathcal N_P)=0
$$
for every predicate symbol $P\in\mathcal L$. In particular, this implies that $\phi_{\mathcal N'}=\phi_{\mathcal N}$.
\end{corollary}
\begin{proof}
The theory $T'$ obtained from $T$ by negating all atomic formulas is almost Horn. Apply to it Theorem~\ref{thm:herl} and negate the resulting $T'$-on (note that~\eqref{eq:npprime} now becomes
$
\mathcal N_P'\df \mathcal E_{k(P)} \setminus D(\mathcal E_{k(P)}\setminus \mathcal N_P)
$).
\end{proof}

The dual of full theorem~\ref{thm:herl} also follows by the same argument.

\bigskip
Note that the underlying reason why the proof of Theorem~\ref{thm:herl} works is that every point $y\in\mathcal{N}'_P$ is ``guaranteed'' to be correct because in its neighborhood ``almost all'' points are also in $\mathcal{N}'_P$. The same idea will be used in the proof of Theorem~\ref{thm:ierl}, but this time we need to ensure that points both in $\mathcal{N}'_P$ and its complement are correct. However, there are points that are neither density points of $\mathcal{N}_P$ nor of its complement, and this is precisely where we will have to resort to the axiom of choice.

The idea of the proof is that we want to ``repair'' the peons in a way that all axioms of the theory are respected and apply Theorem~\ref{thm:characterization}. To do that, we first invoke the Compactness Theorem for propositional
logic and reduce the problem to ``repairing'' only finitely many points $y\in\mathcal{E}_{k(P)}$.

Then we define random variables $\rn{y^{(r)}}$ uniformly distributed over $B(y,r)$ and we decide whether to put $y$ in the $P$-on $\mathcal{N}'_P$ based on whether $\rn{y^{(r)}}$ is in $\mathcal{N}_P$ or not. If $r$ is small enough, then with high probability density points of $\mathcal{N}_P$ will be put in $\mathcal{N}'_P$ and density points of $\mathcal{E}_{k(P)}\setminus\mathcal{N}_P$ will be put in $\mathcal{E}_{k(P)}\setminus\mathcal{N}'_P$. The remaining points will be assigned randomly, but will have a positive measure witness to the fact that they satisfy the axioms of the theory.

Let us now do the formal proof.

\begin{proofof}{Theorem~\ref{thm:ierl}}
  By Remark~\ref{rmk:substclosed}, we can assume without loss of generality
  that $T$ is substitutionally closed. Let us call a point
  $y\in\mathcal{E}_{k(P)}\setminus\mathcal{D}_{k(P)}$ \emph{bad for
  $P\in\mathcal{L}$} if $y\notin D(\mathcal{N}_P)\cup
  D(\mathcal{E}_{k(P)}\setminus\mathcal{N}_P)$ (i.e., if $y$ is not a density
  point of either $\mathcal{N}_P$ or its complement) and let $\mathcal{B}_P$
  be the set of all points that are bad for $P$. Note that
  $\lambda(\mathcal{B}_P) = 0$ by Proposition~\ref{prop:Ldensityae}.

Our $P$-ons $\mathcal N_P'$ will contain the set $D(\mathcal N_P)$ and will be disjoint from the set $D(\mathcal{E}_{k(P)}\setminus\mathcal{N}_P)$, which will immediately give~\eqref{eq:difference}. The behavior
of $\mathcal N_P'$ on the remaining set $\mathcal B_P$ can be described by an (uncountable) set
of propositional variables $p_{P,y}\ (P\in\mathcal L,\ y\in\mathcal B_P)$
with the intended meaning ``$p_{P,y}=1\equiv y\in\mathcal N_P'$''. By Theorem~\ref{thm:characterization}, the $T$-on $\mathcal N' = (\mathcal N_P')_{P\in\mathcal L}$ is strong if and only if for every axiom $\forall\vec
x F(x_1,\ldots,x_n)$ and every $z\in\mathcal E_n\setminus \mathcal D_n$ we
have $z\in T(F,\mathcal N')$. For any fixed $z$ the latter fact is expressible by a
finite propositional formula $A_{F,z}$ in the variables $p_{P,y}$. We have to prove that this
system of propositional constraints is consistent.

For this purpose we invoke the Compactness Theorem for propositional logic
(see e.g.~\cite{ChKr}): as we noted in the introduction, while this step may look
innocent, it is actually equivalent to a weak form of the axiom of choice.
According to this theorem, it is sufficient to prove that any {\em finite}
system $\{A_{F_1,z_1},\ldots, A_{F_\ell, z_\ell}\}$ of constraints is
consistent. Fix for the rest of the argument any such system, and let us
denote by $n_\nu$ the number of variables in $F_\nu$. Let also $Y_P$ be
the set of all $y\in \mathcal E_{k(P)}$ for which at least one of these
constraints contains a propositional variable $p_{P,y}$. Note that all $y\in Y_P$ are of the form
$i^\ast(z_\nu)$ for some $\nu\in [\ell]$ and
$i\injection{[k(P)]}{[n_\nu]}$. In particular, since $z_\nu\not\in \mathcal
D_{n_\nu}$, we have $Y_P\cap\mathcal D_{k(P)}=\emptyset$.

Now, let $\Omega\subseteq [0,1]$ be the finite set of all the coordinates of
all the points $z_1,\ldots,z_\ell$ (hence any $y\in Y_P$ also has these
coordinates).
For $x\in\Omega$ and $X\subseteq\Omega$ let us introduce
a random variable $\rn{\xi^{(r)}}(x,X)$ uniformly distributed in $[x-r,x+r]\cap[0,1]$;
all these variables are assumed to be mutually independent, {\em including those that correspond to the
same $x$}.

These variables naturally define random perturbations $\rn{z_1^{(r)}},\ldots,
\rn{z_{\ell}^{(r)}}$ of the points $z_1,\ldots,z_\ell$, as well as of all
points $y\in Y_P$. Namely, we let
$$
(\rn{z_\nu^{(r)}})_A \df \rn{\xi^{(r)}}((z_\nu)_A, \{(z_\nu)_{\{i\}} \mid i\in A\}),
$$
and similarly for $y\in Y_P$:
$$
(\rn{y^{(r)}})_A \df \rn{\xi^{(r)}}(y_A, \{y_{\{i\}} \mid i\in A\}).
$$
Two straightforward but very useful facts about these distributions are:
\begin{description}
\item[Consistency] Let $\nu\in [\ell]$, and assume that $y=i^\ast(z_\nu)$
    for some $i\injection{[k(P)]}{[n_\nu]}$. Then $\rn{y^{(r)}}$ is the
    pushforward distribution $i^\ast(\rn{z_\nu^{(r)}})$.

\item[Local Independence] For any {\em fixed} $\nu\in[\ell]$, the variable
    $\rn{z_\nu^{(r)}}$ has uniform distribution over $B(z_\nu,r)\cap\mathcal{E}_{n_\nu}$, and
    the same is true for $\rn{y^{(r)}}\ (y\in Y_P)$. Indeed, since
    $z_\nu\not\in\mathcal D_{n_\nu}$, all sets $\set{(z_\nu)_{\{i\}}}{i\in
    A}$ are pairwise different. Hence all random variables involved in
    the definition of $\rn{z_\nu^{(r)}}$ (or $\rn{y^{(r)}}$) are mutually
    independent\footnote{This is precisely why we need the extra parameter $X$: our definition of
    the diagonal $\mathcal D_n$ does not forbid collisions in higher-order coordinates.}.
\end{description}

We now can also define a random Boolean assignment $\rn{u^{(r)}}$ to the
variables $p_{P,y}\ (y\in Y_P)$ by letting $\rn{u^{(r)}}_{P,y}=1 \equiv
\rn{y^{(r)}}\in\mathcal N_P$. As there are only finitely many of them, we can
fix an assignment $u$ in such a way that
\begin{equation} \label{eq:u_consistency}
\limsup_{r\to 0}\prob{\rn{u^{(r)}}=u}>0.
\end{equation}
 We claim that this $u$ is good,
i.e., it satisfies all the axioms $A_{F_\nu, z_\nu}$.

Recalling the definition of $A_{F_\nu,z_\nu}$, we want to show that
upon updating all $P$-ons $\mathcal N_P$ to $\mathcal N_P'$ on the points
$y\in \mathcal B_P\cap Y_P$ according to the rule $y\in\mathcal N_P'\equiv
u_{P,y}=1$ we will have $z_\nu\in T(F_\nu,\mathcal N')$ for all $\nu$. For that we compare
to the event $\rn{z_\nu^{(r)}}\in T(F_\nu,\mathcal N)$.

Firstly, we have $\prob{\rn{z_\nu^{(r)}}\in T(F_\nu,\mathcal N)}=1$ simply because
$\mathcal N$ is a weak $T$-on. Thus, it suffices to show that
\begin{equation} \label{eq:final_strike}
\limsup_{r\to 0} \prob{\rn{z_\nu^{(r)}}\in T(F_\nu,\mathcal N) \equiv
z_\nu\in T(F_\nu,\mathcal N')}>0.
\end{equation}

For $P\in\mathcal L$ and $i\injection{[k(P)]}{[n_\nu]}$, let $E^{(r)}(P,i)$
be the event $i^\ast(\rn{z_\nu^{(r)}})\in\mathcal N_P \equiv
i^\ast(z_\nu)\in\mathcal N_P'$; then the event in~\eqref{eq:final_strike} is
implied by the conjunction of all $E^{(r)}(P,i)$ (see item~\ref{it:truthmain} in
Definition~\ref{def:truth}). However, since
$i^\ast(\rn{z_\nu^{(r)}})=\rn{y^{(r)}}$, by the Consistency property, the
conjunction $\bigwedge\set{E^{(r)}(P,i)}{i^\ast(z_\nu)\in \mathcal B_P}$ is
precisely the event $\rn{u^{(r)}}=u$ in~\eqref{eq:u_consistency}. On the
other hand, if $i^\ast(z_\nu)\not\in\mathcal B_P$ then
$\lim_{r\to 0}\prob{E^{(r)}(P,i)}=1$ is a consequence of Definition~\ref{def:density} and Local Independence. Hence~\eqref{eq:u_consistency}
implies~\eqref{eq:final_strike}.

The proof of Theorem~\ref{thm:ierl} is complete.
\end{proofof}

Note that the proof of Theorem~\ref{thm:ierl} above actually gives us more information (which will be useful in Section~\ref{sec:posetons}) about the structure of the difference set:

\begin{proposition}\label{prop:ierl}
  If $T$ is a theory in a language $\mathcal{L}$ and $\mathcal{N}$ is a weak $T$-on, then there exists a strong $T$-on $\mathcal{N}'$ such that
  \begin{align*}
    D(\mathcal{N}_P)\setminus\mathcal{D}_{k(P)} & \subseteq \mathcal{N}'_P
    \subseteq \mathcal{E}_{k(P)}\setminus D(\mathcal{E}_{k(P)}\setminus\mathcal{N}_P)
  \end{align*}
  for every predicate symbol $P\in\mathcal{L}$.
\end{proposition}

\subsection{Constructive proof for linear orders} \label{sec:constructive}

As we mentioned before, Theorem~\ref{thm:ierl} gives a non-constructive proof of the existence of the desired strong $T$-on using the axiom of choice. As a consequence, this strong $T$-on is not necessarily Borel. On the
other hand, Theorem~\ref{thm:herl} gives a choice-free construction of a strong Borel $T$-on in the case when $T$ is an almost Horn theory. While we do not know whether a constructive proof of Theorem~\ref{thm:ierl} is
possible in general, in this subsection we present an ad hoc argument for the non-Horn theory of linear orders ($T_{\operatorname{LinOrder}}$). The proof is somewhat on a technical side and this result is not used in
the rest of the paper. But it highlights the difficulties on the way of trying to get a constructive version of Theorem~\ref{thm:ierl} for arbitrary theories.

\medskip

We start with defining a few notions that already were informally used in various contexts.

\begin{definition}\label{def:symm}
  Let $P$ be a predicate symbol of arity $2$ and let $\mathcal{N}\subseteq\mathcal{E}_2$ be a $P$-on.

  Let us say that the peon $\mathcal{N}$ is \emph{anti-symmetric} if
  $(x_{\{1\}},x_{\{2\}},x_{\{1,2\}})\in\mathcal{N} \equiv (x_{\{2\}},x_{\{1\}},x_{\{1,2\}})\notin\mathcal{N}$
  for every $x\in\mathcal{E}_2\setminus\mathcal{D}_2$. In other words, $\mathcal{N}$ is a strong
  $T_{\operatorname{Tournament}}$-on.

  Let us say that $\mathcal{N}$ is \emph{transitive} if for every $x\in\mathcal{E}_3\setminus\mathcal{D}_3$,
  we have
  \begin{align*}
    (x_{\{1\}},x_{\{2\}},x_{\{1,2\}})\in\mathcal{N}\land (x_{\{2\}},x_{\{3\}},x_{\{2,3\}})\in\mathcal{N}
    \to
    (x_{\{1\}},x_{\{3\}},x_{\{1,3\}})\in\mathcal{N}.
  \end{align*}
  In other words, $\mathcal{N}$ is a strong $T_{\operatorname{PreOrder}}$-on, where $T_{\operatorname{PreOrder}}$ is the theory of (partial) preorders.
In these terms, a strong $T_{\operatorname{LinOrder}}$-on is simply an anti-symmetric and transitive peon.

  For $(x,y)\in\mathcal{E}_2^*$, we define the section
  \begin{align*}
    A_{\mathcal{N}}(x,y) & \df \{z\in[0,1] \mid (x,y,z)\in\mathcal{N}\}.
  \end{align*}
The $P$-on $\mathcal{N}$ is called \emph{$\mathcal{E}_2^*$-measurable} if for all $(x,y)\in\mathcal{E}_2^*$, the section $A_{\mathcal{N}}(x,y)$ is either $\emptyset$ or $[0,1]$,
i.e., $\mathcal N$ does not depend on the third coordinate.
\end{definition}

$\mathcal{E}_2^*$-measurable peons correspond to $\{0,1\}$-valued graphons in Lov\'{a}sz's
terminology; the reason we prefer the name $\mathcal{E}_2^*$-measurable is that peons are sets rather than
measurable functions to $[0,1]$ (cf.~the correspondence between graphons and $2$-hypergraphons in
Section~\ref{sec:hypergraphons}).

It may seem that all strong $T_{\operatorname{LinOrder}}$-ons are $\mathcal{E}_2^*$-measurable, but the
following example shows that this is not the case.

\begin{example}
  The $T_{\operatorname{LinOrder}}$-on $\mathcal{N}$ defined by
  \begin{align*}
    \mathcal{N}_{\prec}
    =
    \{x\in\mathcal{E}_2 \mid {}
    &
    x_{\{1\}}\bmod(1/2) < x_{\{2\}}\bmod(1/2)
    \\
    & \lor
    (x_{\{1\}} = x_{\{2\}} - 1/2\land x_{\{1,2\}} < 1/2 \land x_{\{2\}}\neq 1)
    \\
    & \lor
    (x_{\{1\}} = x_{\{2\}} + 1/2\land x_{\{1,2\}} \geq 1/2)
    \\
    & \lor
    x_{\{1\}} = 1
    \}
  \end{align*}
  is a strong $T_{\operatorname{LinOrder}}$-on (cf.~Theorem~\ref{thm:characterization}); it differs from the (weak) $T_{\operatorname{LinOrder}}$-on $\mathcal{N}$ of Example~\ref{ex:twoorderedtheons} by a set of measure~0.

  Intuitively, this $T_{\operatorname{LinOrder}}$-on corresponds to the ``random total order'' $\rn{\preceq}$ on $[0,1]$ defined as follows. First we let $a\rn{\preceq} b$ for every $x,y\in[0,1]$ such that
  \begin{align*}
    x\bmod(1/2) < y\bmod(1/2).
  \end{align*}
  We also let $1\rn{\preceq} x$ for every $x\in[0,1]$. Then for each $x\in[0,1/2)$ we, roughly speaking, make a ``random choice'' $x\rn{\preceq} x+1/2$ or $x+1/2\rn{\preceq} x$ with probability $1/2$ each.
\end{example}

The constructive (i.e., avoiding the axiom of choice) proof of Theorem~\ref{thm:ierl} for $T_{\operatorname{LinOrder}}$ below
can be summarized into the following three steps:
\begin{enumerate}
\item Get a (weak Borel) anti-symmetric $T_{\operatorname{LinOrder}}$-on.
\item Get $\mathcal{E}_2^*$-measurability preserving anti-symmetry (and Borel measurability).
\item Get transitivity preserving anti-symmetry (as well as Borel and $\mathcal{E}_2^*$-measurability).
\end{enumerate}

The first item on this program is easy but, remarkably, it takes care of the only non-Horn axiom
of $T_{\operatorname{LinOrder}}$.

\begin{lemma}\label{lem:theonsymmetry}
  If $\mathcal{N}$ is a weak $T_{\operatorname{LinOrder}}$-on, then there exists a weak Borel anti-symmetric
  $T_{\operatorname{LinOrder}}$-on $\mathcal{N}'$ such that $\lambda(\mathcal{N}\symmdiff\mathcal{N}') = 0$.

  Furthermore, if $\mathcal{N}$ is $\mathcal{E}_2^*$-measurable, then $\mathcal{N}'$ can also be taken
  $\mathcal{E}_2^*$-measurable.
\end{lemma}

\begin{proof}
  By possibly changing $\mathcal{N}$ in a zero-measure set, we may suppose it is a Borel theon. Then we let
  \begin{align*}
    \mathcal{N}' & \df
    \{x\in\mathcal{E}_2 \mid x\in\mathcal{N}\setminus(\mathcal{N}\cdot\sigma)
    \lor (x\notin\mathcal{N}\symmdiff(\mathcal{N}\cdot\sigma)\land x_{\{1\}} < x_{\{2\}})\},
  \end{align*}
  where $\sigma$ is the unique non-identity permutation in $S_2$ (recall the natural action of $S_2$ on
  $\mathcal{E}_2$ from Definition~\ref{def:EV}). It is obvious that this construction preserves $\mathcal{E}_2^*$-measurability.
\end{proof}

\begin{lemma}\label{lem:E2*measurable}
  If $\mathcal{N}$ is a weak Borel anti-symmetric $T_{\operatorname{LinOrder}}$-on, then there exists a weak
  Borel $\mathcal{E}_2^*$-measurable anti-symmetric $T_{\operatorname{LinOrder}}$-on $\mathcal{N}'$ such that
  $\lambda(\mathcal{N}\symmdiff\mathcal{N}') = 0$.
\end{lemma}

\begin{proof}
  Since the Borel $\sigma$-algebra on $\mathcal{E}_2$ is the product of Borel $\sigma$-algebras of $[0,1]$
  in each coordinate, it follows that every section $A_{\mathcal{N}}(x,y)$ is a Borel set for every
  $(x,y)\in\mathcal{E}_2^*$.

  Let us call a pair $(x,y)\in\mathcal{E}_2^*$ \emph{vanishing} if
  $\lambda(A_{\mathcal{N}}(x,y))=0$. By Fubini's Theorem, we know that if $\mathcal{V}$ is the
  set of vanishing pairs then $(\mathcal{V}\times[0,1])\cap\mathcal{N}$ is Borel and has measure $0$.

  Let us call a pair $(x,y)\in\mathcal{E}_2^*$ \emph{full} if
  $\lambda(A_{\mathcal{N}}(x,y))=1$. By Fubini's Theorem, we know that if $\mathcal{F}$ is the
  set of full pairs then $(\mathcal{F}\times[0,1])\setminus\mathcal{N}$ is Borel and has measure $0$.

  Finally, let us call a pair $(x,y)\in\mathcal{E}_2^*$ \emph{bad} if
  $0<\lambda(A_{\mathcal{N}}(x,y))<1$. Again, Fubini's Theorem implies that the set of bad
  pairs $\mathcal{B}$ is a Borel set.

  Note that if $\mathcal{B}$ has zero measure then setting
  \begin{align*}
    \mathcal{N}'
    & \df
    \bigl(\mathcal{N}\cup(\mathcal{F}\times[0,1])\bigr)\setminus
    \bigl((\mathcal{B}\cup\mathcal{V})\times[0,1]\bigr)
  \end{align*}
  gives a weak Borel $\mathcal{E}_2^*$-measurable $T_{\operatorname{LinOrder}}$-on, to which we can apply
  Lemma~\ref{lem:theonsymmetry} once more and get back anti-symmetry while preserving
  $\mathcal{E}_2^*$-measurability. Thus, it remains to prove that $\mathcal{B}$ does have
  zero measure.

  Suppose not, then by countable additivity there must exist $n\in\mathbb{N}_+$ such that
  \begin{align*}
    \mathcal{B}_n
    & \df
    \set{(x,y) \in \mathcal{E}_2^*}{
      \frac{1}{n} \leq \lambda(A_{\mathcal{N}}(x,y)) \leq 1 - \frac{1}{n}}
  \end{align*}
  has positive measure. Note that the anti-symmetry of $\mathcal{N}$ implies that
  $\mathcal{B}_n$ is symmetric, that is, we have
  $(x,y)\in\mathcal{B}_n\equiv(y,x)\in\mathcal{B}_n$.

  For every $x\in[0,1]$, define the section
  \begin{align*}
    \mathcal{B}_n(x) & \df \{y\in[0,1] \mid (x,y)\in\mathcal{B}_n\}.
  \end{align*}
  By Fubini's Theorem, the set $X$ of $x$ such that $\lambda(\mathcal{B}_n(x)) > 0$ has
  positive measure.

  We now pick $\rn x$ uniformly at random from the set
  $$
  \set{x\in\mathcal E_3}{x_{\{2\}}, x_{\{3\}}\in \mathcal B_n(x_{\{1\}})}.
  $$

  Since
  \begin{align*}
    \indprob{\rn{x_{\{1\}}},\rn{x_{\{2\}}},\rn{x_{\{3\}}}}{
      \indprob{\rn{x_{\{1,2\}}},\rn{x_{\{1,3\}}}}{
        (\rn{x_{\{2\}}},\rn{x_{\{1\}}},\rn{x_{\{1,2\}}})\in\mathcal{N}\land
        (\rn{x_{\{1\}}},\rn{x_{\{3\}}},\rn{x_{\{1,3\}}})\in\mathcal{N}
      } \geq \frac{1}{n^2}
    }
    & = 1,
  \end{align*}
  and since $\mathcal{N}$ is a weak $T_{\operatorname{LinOrder}}$-on, it follows that
  \begin{align*}
        \prob{(\rn{x_{\{2\}}},\rn{x_{\{3\}}},\rn{x_{\{2,3\}}})\in\mathcal{N}
      } = 1.
  \end{align*}
  But this implies
  \begin{align*}
    \indprob{\rn{x_{\{1\}}},\rn{x_{\{2\}}},\rn{x_{\{3\}}}}{
      \indprob{\rn{x_{\{1,2\}}},\rn{x_{\{2,3\}}}}{
        (\rn{x_{\{1\}}},\rn{x_{\{2\}}},\rn{x_{\{1,2\}}})\in\mathcal{N}\land
        (\rn{x_{\{2\}}},\rn{x_{\{3\}}},\rn{x_{\{2,3\}}})\in\mathcal{N}
      } \geq \frac{1}{n}
    } & = 1,
  \end{align*}
  hence, repeating the previous argument,  $\prob{(\rn{x_{\{1\}}},\rn{x_{\{3\}}},\rn{x_{\{1,3\}}})\in\mathcal{N}} = 1$, contradicting the fact
  that $\rn{x_{\{3\}}}$ is picked in $\mathcal{B}_n(\rn{x_{\{1\}}})$.

  Therefore the set of bad pairs $\mathcal{B}$ has zero measure and the proof is complete.
\end{proof}

Before we proceed to the final step, let us prove an easy lemma about anti-symmetric peons.

\begin{lemma}\label{lem:antisymvertex}
  Let $\mathcal{N}$ be an anti-symmetric peon and let $U\subseteq[0,1]$ be a Lebesgue measurable set with
  $\lambda(U) > 0$. Then there exist $x_1,x_2\in U$ such that
  \begin{align}
    \label{eq:antisymsection}
    \lambda(\{(y,z)\in U\times[0,1] \mid (x_1,y,z)\in\mathcal{N}\}) & > 0;
    \\
    \notag
    \lambda(\{(y,z)\in U\times[0,1] \mid (y,x_2,z)\in\mathcal{N}\}) & > 0.
  \end{align}
\end{lemma}

\begin{proof}
  For every $x\in U$, let $V(x)$ be the set in~\eqref{eq:antisymsection} with $x_1 = x$.

  Since $\mathcal{N}$ is anti-symmetric, by Fubini's Theorem, we have
  \begin{align*}
    0 < \frac{\lambda(U)^2}{2} & = \lambda((U\times U\times [0,1])\cap\mathcal{N})
    = \int_U \lambda(V(x)) d\lambda(x),
  \end{align*}
  so there exists $x_1\in U$ such that $\lambda(V(x_1))>0$. The assertion for $x_2$ follows by
  anti-symmetry.
\end{proof}

\begin{theorem}\label{thm:strongLinOrder}
  If $\mathcal{N}$ is a weak $T_{\operatorname{LinOrder}}$-on, then there exists a strong
  Borel $\mathcal{E}_2^*$-measurable $T_{\operatorname{LinOrder}}$-on $\mathcal{N}'$ such
  that $\lambda(\mathcal{N}\symmdiff\mathcal{N}')=0$.
\end{theorem}

\begin{proof}
  By Lemmas~\ref{lem:theonsymmetry} and~\ref{lem:E2*measurable}, we may suppose that $\mathcal{N}$ is a weak Borel
  $\mathcal{E}_2^*$-measurable anti-symmetric $T_{\operatorname{LinOrder}}$-on.

  Since all $T_{\operatorname{LinOrder}}$-ons in this proof will be $\mathcal{E}_2^*$-measurable, we will
  suppress all dummy variables $x_V$ indexed by $V$ with $\lvert V\rvert\geq 2$. Furthermore, since all
  variables are now indexed by singletons, we will use the notation $x_i$ for $x_{\{i\}}$.

  Let $F$ be the open formula
  \begin{align*}
    x\prec y\land y\prec z\to x\prec z.
  \end{align*}
  By Theorem~\ref{thm:characterization}, we have $\lambda(T(F,\mathcal{N}))=1$. For every $x\in[0,1]$,
  define the sections
  \begin{align*}
    T(F,\mathcal{N})_1(x) & \df \{(x_2,x_3)\in[0,1]^2 \mid (x,x_2,x_3)\in T(F,\mathcal{N})\};\\
    T(F,\mathcal{N})_2(x) & \df \{(x_1,x_3)\in[0,1]^2 \mid (x_1,x,x_3)\in T(F,\mathcal{N})\};\\
    T(F,\mathcal{N})_3(x) & \df \{(x_1,x_2)\in[0,1]^2 \mid (x_1,x_2,x)\in T(F,\mathcal{N})\};
  \end{align*}
  and let $G$ be the set of ``good'' points $x\in[0,1]$ such that
  $\lambda(T(F,\mathcal{N})_i(x)) = 1$ for all $i\in[3]$. Note that $\lambda(G)=1$ by Fubini's Theorem.

  For every $(x_1,x_2)\in[0,1]^2$, define the ``witness'' set
  \begin{align*}
    W(x_1,x_2) & \df \{y\in G \mid (x_1,y)\in\mathcal{N}\land (y,x_2)\in\mathcal{N}\}.
  \end{align*}

  Let us call a pair $(x_1,x_2)\in (G\times G)\setminus\mathcal{D}_2$ \emph{excellent} if at least one of
  $W(x_1,x_2)$ or $W(x_2,x_1)$ has positive measure and let $E\subseteq (G\times
  G)\setminus\mathcal{D}_2$ be the set of excellent pairs (note that $(x_1,x_2)\in E\equiv (x_2,x_1)\in E$).

  We now define $\mathcal{N}'$ by
  \begin{equation}\label{eq:LinOrder}
    \begin{aligned}
      \mathcal{N}'
      & \df
      \{(x_1,x_2)\in E \mid \lambda(W(x_1,x_2)) > 0\}
      \\ & \qquad
      \cup\{(x_1,x_2)\in (G\times G)\setminus E \mid x_1 < x_2\}
      \\ & \qquad
      \cup([0,1]\setminus G)\times G
      \\ & \qquad
      \cup\{(x_1,x_2)\in([0,1]\setminus G)\times([0,1]\setminus G) \mid x_1 < x_2\}
    \end{aligned}
  \end{equation}
  Fubini's Theorem guarantees that this is a Borel set. The intuition behind this construction is that $x_1$
  is declared ``smaller than'' $x_2$ if there is a positive measure witness $W(x_1,x_2)$ to this
  fact. However, since we need the resulting relation to be total, we need to consistently decide the ordering
  between pairs that are not excellent.

  Let us prove that $\mathcal{N}'$ is anti-symmetric.
If $(x_1,x_2)\notin E$ with $x_1\neq x_2$, then clearly $(x_1,x_2)\in\mathcal{N}'\equiv
  (x_2,x_1)\notin\mathcal{N}'$.
This means that the only way $\mathcal{N}'$ can fail anti-symmetry is if there exist distinct $x_1,x_2\in G$
  with $\lambda(W(x_1,x_2))>0$ and $\lambda(W(x_2,x_1))>0$. But since $x_1\in G$, we know that for almost
  every $(y,z)\in W(x_2,x_1)\times W(x_1,x_2)$, we have $(y,z)\in\mathcal{N}$. On the other hand, since
  $x_2\in G$, we know that for almost every $(z,y)\in W(x_1,x_2)\times W(x_2,x_1)$, we have
  $(z,y)\in\mathcal{N}$. Hence if $\lambda(W(x_1,x_2))>0$ and $\lambda(W(x_2,x_1)) > 0$, then almost every point of the positive measure set $W(x_1,x_2)\times W(x_2,x_1)$ violates the anti-symmetry of $\mathcal{N}$, a
  contradiction. Therefore $\mathcal{N}'$ is anti-symmetric.

  \medskip

  Let us now show that $\lambda(\mathcal{N}\symmdiff\mathcal{N}') = 0$.
To do so, let us first show that $\lambda(E) = 1$. Note that since $\lambda(G)=1$ and $E\subseteq G\times
  G$, it is enough to show that the set
  \begin{align*}
    Z & \df (G\times G)\setminus E
  \end{align*}
  has zero measure. For every $x_1\in G$, define
  \begin{align*}
    Z_1(x_1) & \df \{x_2\in G \mid (x_1,x_2)\in Z\cap\mathcal{N}\};\\
    Z_2(x_1) & \df \{x_2\in G \mid (x_1,x_2)\in Z\setminus\mathcal{N}\}.
  \end{align*}

  We claim that $\lambda(Z_1(x_1)) = 0$ for every $x_1\in G$. Indeed, otherwise, by
  Lemma~\ref{lem:antisymvertex}, there would exist $y\in Z_1(x_1)$ such that $\lambda(\{z\in Z_1(x_1) \mid
  (z,y)\in\mathcal{N}\}) > 0$, which would imply that $\lambda(W(x_1,y)) > 0$, contradicting $(x_1,y)\notin
  E$.

  Using anti-symmetry, we also conclude that $\lambda(Z_2(x_1))=0$ for every $x_1\in G$, and by Fubini's
  Theorem, we get $\lambda(E) = 1$. This means that to show that $\lambda(\mathcal{N}\symmdiff\mathcal{N}')=0$ it is
  sufficient to prove that $\lambda((\mathcal{N}\symmdiff\mathcal{N}')\cap E) = 0$.

  If $(x_1,x_2)\in(\mathcal{N}'\setminus\mathcal{N})\cap
  E$, then for every $y\in W(x_1,x_2)$ we have $(x_1,y,x_2)\in T(\neg F,\mathcal{N})$. But then Fubini's
  Theorem implies that
   $\lambda(W(x_1,x_2)) = 0$ for almost every $(x_1,x_2)\in(\mathcal{N}'\setminus\mathcal{N})\cap E$. Now
  the definition of $\mathcal{N}'$ gives $\lambda(W(x_1,x_2))>0$ for every $(x_1,x_2)\in\mathcal{N}'\cap E$,
  so we must have $\lambda((\mathcal{N}'\setminus\mathcal{N})\cap E)=0$.

  On the other hand, by anti-symmetry, we have $(x_1,x_2)\in(\mathcal{N}\setminus\mathcal{N}')\cap E\equiv
  (x_2,x_1)\in(\mathcal{N}'\setminus\mathcal{N})\cap E$, so it follows that
  $\lambda((\mathcal{N}\setminus\mathcal{N}')\cap E)=0$ as well.
Therefore $\lambda(\mathcal{N}\symmdiff\mathcal{N}')=0$.
  It remains to prove that $\mathcal{N}'$ is also transitive.
Fix $(x_1,x_2,x_3)\in[0,1]^3$.

  If at least one of $x_1,x_2,x_3$ is not in $G$, then clearly $\mathcal{N}'$ satisfies transitivity for (all
  permutations of) $(x_1,x_2,x_3)$. The same holds if none of the pairs $(x_1,x_2),(x_2,x_3),(x_3,x_1)$ is in
  $E$.

  Suppose then that $x_1,x_2,x_3\in G$ and that at least one pair is in $E$. Without loss of generality,
  let us suppose that $\lambda(W(x_1,x_2)) > 0$.

  By anti-symmetry of $\mathcal{N}$, for every $z\in W(x_1,x_2)\setminus\{x_3\}$, we either have
  $(z,x_3)\in\mathcal{N}$ or $(x_3,z)\in\mathcal{N}$. Since $\lambda(W(x_1,x_2)) > 0$, at least one of these
  possibilities must have positive probability, so we either have $\lambda(W(x_1,x_3)) > 0$ or
  $\lambda(W(x_3,x_2)) > 0$; in both cases, it follows that $\mathcal{N}'$ satisfies transitivity for
  $(x_1,x_2,x_3)$.
Therefore $\mathcal{N}'$ is a strong Borel $\mathcal{E}_2^*$-measurable $T_{\operatorname{LinOrder}}$-on such
  that $\lambda(\mathcal{N}\symmdiff\mathcal{N}')=0$.
\end{proof}

\begin{corollary}
  If $\mathcal{N}$ is a weak $T_{\operatorname{Perm}}$-on, then there exists a strong Borel
  $T_{\operatorname{Perm}}$-on $\mathcal{N}'$ such that $\mathcal{N}_{\prec_i}$ is
  $\mathcal{E}_2^*$-measurable and $\lambda(\mathcal{N}_{\prec_i}\symmdiff\mathcal{N}'_{\prec_i})=0$ for all
  $i\in[2]$.
\end{corollary}

\begin{proof}
  Simply apply Theorem~\ref{thm:strongLinOrder} to each of the peons separately. The resulting $T_{\operatorname{Perm}}$-on is strong since $T_{\operatorname{Perm}}$ is the {\em disjoint} union of two copies of $T_{\operatorname{LinOrder}}$.
\end{proof}

\section{Existence and uniqueness}\label{sec:existence-uniqueness}

The objective of this section is to prove Theorems~\ref{thm:theoncryptomorphism} and~\ref{thm:TheonUniqueness}, but before we do so, we will prove that yet another object is cryptomorphic to limits of convergent sequences of models.
Throughout this section, random variables will be identified with their distributions, that is we do not distinguish between random variables corresponding to the same probability measure.

Let us first recall the definition of weak convergence.

\begin{definition}\label{def:weakconv}
A {\em Polish space} is a separable completely metrizable topological space. Let $S$ be a Polish space endowed with the Borel $\sigma$-algebra $\mathcal{B}(S)$; we view it as a (standard) Borel space. A
  sequence $(\rn{X_n})_{n\in\mathbb{N}}$ of random $S$-valued variables \emph{weakly converges} (or
  \emph{converges in law}, or \emph{converges in distribution}) to another $S$-valued random variable $\rn{X}$
  if for any bounded continuous function $f\in C(S)$, $\lim_{n\to\infty}\expect{f(\rn{X_n})} = \expect{f(\rn
    X)}$.

It is a direct consequence of Prokhorov's Theorem that the random variable $\rn X$ is uniquely defined (due to our convention),
we will denote it by $\lim_{n\to\infty}\rn{X_n}$ and say that the sequence $(\rn{X_n})_{n\in\mathbb{N}}$ \emph{weakly converges} if this limit exists.
\end{definition}

Let us recall two important theorems on weak convergence. A {\em continuity set of $\rn{X}$} is a Borel set $B$ such that $\prob{\rn X\in\partial B}=0$, where $\partial B$ is the boundary of $B$.

\begin{theorem}[Portmanteau]\label{thm:port}
  If $\rn{X}$ and $\rn{X_n}$ ($n\in\mathbb{N}$) are random variables, then the following are equivalent.
  \begin{itemize}
  \item The sequence $(\rn{X_n})_{n\in\mathbb{N}}$ weakly converges to $\rn{X}$.
  \item For every continuity set $B$ of $\rn X$ we have
    \begin{align*}
      \lim_{n\to\infty}\prob{\rn{X_n}\in B} & = \prob{\rn X\in B}.
    \end{align*}
  \item For every open set $U\subseteq S$, we have
    \begin{align*}
      \liminf_{n\to\infty}\prob{\rn{X_n}\in U} & \geq \prob{\rn{X}\in U}.
    \end{align*}
  \item For every closed set $C\subseteq S$, we have
    \begin{align*}
      \limsup_{n\to\infty}\prob{\rn{X_n}\in C} & \leq \prob{\rn{X}\in C}.
    \end{align*}
  \end{itemize}
\end{theorem}

\begin{theorem}[Method of moments]\label{thm:moments}
Let $I$ be a finite or countable set of indices, and let $S\df [0,1]^I$ be endowed with product topology. Let $(\rn{X_n})_{n\in\mathbb{N}}$ be a sequence of $S$-valued random variables such that all joint moments converge, that is, for every finite $I'\subseteq I$ and every $k\function{I'}{\mathbb{N}}$, the limit
  \begin{equation} \label{eq:moments}
    \lim_{n\to\infty}\expect{\prod_{i\in I'}\pi_i(\rn{X_n})^{k(i)}}
  \end{equation}
  exists {\rm (}$\pi_i\function{[0,1]^I}{[0,1]}$ denotes the projection on the $i$th coordinate{\rm )}.

  Then $(\rn{X_n})_{n\in\mathbb{N}}$ weakly converges and $\lim_{n\to\infty} \rn{X_n}$ depends only on the moments~\eqref{eq:moments}.
\end{theorem}

After these preliminaries, let us get to our framework. First, we prove a technical lemma that says that weak convergence of sequences of random models is the same as convergence in expectation. This is a far-reaching generalization of Theorem~\ref{thm:cryptomorphism} (the latter corresponds to deterministic sequences).

\begin{lemma}\label{lem:convexpect}
  Let $(s_n)_{n\in\mathbb{N}}$ be an increasing sequence of integers and $(\rn{N_n})_{n\in\mathbb{N}}$ be a
  sequence of random models of a theory $T$ such that $\lvert V(\rn{N_n})\rvert=s_n$ for every $n\in\mathbb{N}$.
   Then the following are equivalent.
\begin{enumerate}[label={\arabic*.}, ref={\arabic*)}]
\item \label{a} $\lim_{n\to\infty}\expect{p(M,\rn{N_n})}$
  exists for every $M\in\mathcal{M}$.

\item \label{b} The sequence $(p(\place,\rn{N_n}))_{n\in\mathbb{N}}$ of $[0,1]^{\mathcal M}$-valued
random variables weakly converges.

\item \label{c} The sequence $(p(\place,\rn{N_n}))_{n\in\mathbb{N}}$ weakly converges to a random variable supported on $\Hom$.
\end{enumerate}
Moreover, if the above holds then $\lim_{n\to\infty} p(\place,\rn{N_n})$ is uniquely determined by the limits $\lim_{n\to\infty}\expect{p(M,\rn{N_n})}$.
\end{lemma}
\begin{proof}
  \ref{c}~$\Longrightarrow$~\ref{b} and~\ref{b}~$\Longrightarrow$~\ref{a} are obvious. Let us
  prove~\ref{a}~$\Longrightarrow$~\ref{c}.

  Let $M_1,M_2,\ldots,M_t\in\mathcal{M}$ be arbitrary fixed models of $T$. By Lemma~\ref{lem:flaglowcollision}, we know that
  \begin{equation} \label{eq:models_moments}
  \lim_{n\to\infty}\max_{N_n\in\mathcal M_n} \left\lvert p(M_1,M_2,\ldots,M_t; N_n) - \prod_{i=1}^tp(M_i, N_n)\right\rvert = 0.
  \end{equation}
  But by Lemma~\ref{lem:chain2} we know that $p(M_1,M_2,\ldots,M_t;\rn{N_n})$ can be written as a linear combination of
  $(p(M,\rn{N_n}))_{M\in\mathcal{M}}$, which by linearity of expectation implies that the limit
  $\lim_{n\to\infty}\expect{p(M_1,M_2,\ldots,M_t;\rn{N_n})}$ (and hence also
  $\lim_{n\to\infty}\expect{\prod_{i=1}^tp(M_i,\rn{N_n})}$) exists.

  Since all joint moments of $(p(\place,\rn{N_n}))_{n\in\mathbb{N}}$ converge, by Theorem~\ref{thm:moments}
  the sequence $(p(\place,\rn{N_n}))_{n\in\mathbb{N}}$ weakly converges and the limit distribution $\rn{\phi}$ is
  completely determined by its joint moments, which, by~\eqref{eq:models_moments}, are completely determined by $(\lim_{n\to\infty}\expect{p(M,\rn{N_n})})_{M\in\mathcal{M}}$.

  It remains to prove that $\prob{\rn{\phi}\in\Hom}=1$. Since $\mathcal{M}$ is countable, it is enough to prove that $\rn{\phi}$ a.e.\
  satisfies the relations defining $\Hom$, that is:
  \begin{itemize}
  \item for every $M\in\mathcal{M}$ and every $\ell\geq\lvert V(M)\rvert$, we have
    \begin{align}\label{eq:chainphi}
      \prob{\rn{\phi}(M) = \sum_{M'\in\mathcal{M}_\ell}p(M,M')\rn{\phi}(M')} & = 1;
    \end{align}
  \item for every $M_1,M_2\in\mathcal{M}$ and every $\ell\geq\lvert V(M_1)\rvert + \lvert V(M_2)\rvert$, we have
    \begin{align}\label{eq:prodphi}
      \prob{\rn{\phi}(M_1)\rn{\phi}(M_2) = \sum_{M'\in\mathcal{M}_\ell}p(M_1,M_2;M')\rn{\phi}(M')} & = 1.
    \end{align}
  \end{itemize}

  Note that the event in~\eqref{eq:chainphi} is a closed set in $[0,1]^{\mathcal{M}}$ and note that if $s_n\geq\ell$, then
  \begin{align*}
    \prob{p(M,\rn{N_n}) = \sum_{M'\in\mathcal{M}_\ell}p(M,M')p(M',\rn{N_n})} & = 1,
  \end{align*}
  hence~\eqref{eq:chainphi} follows by Theorem~\ref{thm:port}.

  On the other hand, fixing $M_1,M_2\in\mathcal{M}$, $\ell\geq\lvert V(M_1)\rvert + \lvert V(M_2)\rvert$ and $\epsilon>0$, the set
  \begin{align*}
    U_\epsilon
    & =
    \left\{
    \psi\in[0,1]^{\mathcal{M}} \middle\vert
    \left\lvert\psi(M_1)\psi(M_2) - \sum_{M'\in\mathcal{M}_\ell}p(M_1,M_2;M')\psi(M')\right\rvert > \epsilon
    \right\}
  \end{align*}
  is open, and by Lemma~\ref{lem:flaglowcollision} we know that
  \begin{align*}
    \lim_{n\to\infty}\prob{p(\place,\rn{N_n})\in U_\epsilon} & = 0.
  \end{align*}

  By Theorem~\ref{thm:port}, we get $\prob{\rn{\phi}\in U_\epsilon}=0$ for every $\epsilon>0$, which implies that~\eqref{eq:prodphi} holds.
\end{proof}

Next we need to make precise the notion of a random canonical structure on an infinite countable set that we choose to
be $\mathbb{N}_+ = \mathbb{N}\setminus\{0\}$.

\begin{definition}
  Let $V$ be a set, and let $(V)_{<\omega}\df \bigcup_{k\in\mathbb N}(V)_k$.
  Let us also denote by $\mathcal{K}_V[\mathcal{L}]$ the set of all canonical structures in the language $\mathcal{L}$ with vertex set
  $V$ (we do \emph{not} identify isomorphic canonical structures). As usual, we will drop $[\mathcal{L}]$ from the notation
  when it is clear from context, and we will denote $\mathcal{K}_{[n]}$ by $\mathcal{K}_n$.
  For $K\in\mathcal K_{\mathbb N_+}$ and $V\subseteq\mathbb N_+$, we denote by $K\vert_V\in \mathcal K_V$ the structure
  induced by $K$ on the set $V$.
\end{definition}

The set $\mathcal K_{\mathbb N_+}$ can be naturally identified with $\{0,1\}^{\{(P,\alpha) \mid P\in\mathcal L, \
\alpha\in (\mathbb N_+)_{k(P)}\}}$ and, in particular, it inherits the ordinary product topology from that space.
The same topology can be alternatively described by the basis
$\{U_K \mid K\in\mathcal{K}_\ell, \ell\in\mathbb{N}\}$, where
\begin{align*}
  U_K & = \{L\in\mathcal{K}_{\mathbb{N}_+} \mid L\vert_{[\ell]} = K\}.
\end{align*}
Note that each $U_K$ is a clopen set.
This immediately implies that a random structure $\rn{K}\in \mathcal{K}_{\mathbb{N}_+}$
is uniquely determined by its marginals $(\rn{K}\vert_{[\ell]})_{\ell\in\mathbb{N}}$.

\begin{lemma}\label{lem:arrayweakconv}
  Let $(\rn{K_n})_{n\in\mathbb N}$ be $\mathcal K_{\mathbb N_+}$-random variables.
  Then the sequence $(\rn{K_n})_{n\in\mathbb N}$ weakly converges if and only if the limit $\lim_{n\to\infty} \prob{\rn{K_n}\vert_{[\ell]}=K}$ exists for every $\ell\in\mathbb N_+$ and every $K\in\mathcal K_\ell$.

  Moreover, if this is the case and $\rn K=\lim_{n\to\infty} \rn{K_n}$ then
  $$
  \lim_{n\to\infty} \prob{\rn{K_n}\vert_{[\ell]}=K} = \prob{\rn K\vert_{[\ell]}=K},
  $$
  again for every $\ell\in\mathbb N_+$ and every $K\in\mathcal K_\ell$.
\end{lemma}
\begin{proof}
  ``Only if'' part readily follows from Theorem~\ref{thm:port} and the observation that every clopen set is a continuity set.

  For the ``if'' part we have to invoke Prokhorov's theorem again. The space of all probability measures on
  $\mathcal K_{\mathbb N_+}$ (that we identify with random variables) with the topology given by weak convergence is
  compact. Hence, for any $(\rn{K_n})_{n\in\mathbb N}$ there exists a subsequence $(\rn{K_{n_m}})$ weakly converging to a
  random variable $\rn K$. But this space is also metrizable. Hence if the whole sequence  $(\rn{K_{n}})$ would not have converged to $\rn K$, we could have found in it another subsequence $(\rn{K_{n_m'}})$ converging to a
 {\em different} probability measure $\rn L$, This, however, is absurd since
  \begin{align*}
    \prob{\rn K\vert_{[\ell]}=K} & = \lim_{m\to\infty} \prob{\rn{K_{n_m}}\vert_{[\ell]}=K}
    \\
    & = \lim_{m\to\infty} \prob{\rn{K_{n_m'}}\vert_{[\ell]}=K} = \prob{\rn L\vert_{[\ell]}=K},
  \end{align*}
  and, as we remarked above, a probability distribution over $\mathcal K_{\mathbb N_+}$ is completely determined by its finite marginals. This contradiction shows that in fact $\rn K=\lim_{n\to\infty} \rn{K_n}$.

    The second part of the lemma is again immediate from Theorem~\ref{thm:port}.
\end{proof}

\begin{definition}
  For a fixed target set $\Omega$, an array (indexed by $(\mathbb{N}_+)_{<\omega}$) is a
  function $X\function{(\mathbb{N}_+)_{<\omega}}{\Omega}$.

  Let $S_{\mathbb{N}_+}$ denote the symmetric group over $\mathbb{N}_+$ and define the (right) action of $S_{\mathbb{N}_+}$ on the set of arrays indexed by $(\mathbb{N}_+)_{<\omega}$ by letting
  \begin{align*}
    (X\cdot\sigma)_\alpha \df X_{\sigma\circ\alpha}\quad (\alpha\in (\mathbb N_+)_{<\omega}),
  \end{align*}
  for every permutation $\sigma\in S_{\mathbb{N}_+}$ and every array $X$.

  A random array $\rn{X}$ indexed by $(\mathbb{N}_+)_{<\omega}$ is \emph{(jointly) exchangeable} if for every $\sigma\in S_{\mathbb{N}_+}$ we
  have $\rn{X}\sim\rn{X}\cdot\sigma$.
\end{definition}

Let now $\Omega\df \{0,1\}^{\mathcal{L}}$. The elements of the set
$\mathcal{K}_{\mathbb{N}_+}$ that was previously identified with
$\{0,1\}^{\{(P,\alpha) \mid P\in\mathcal L, \ \alpha\in (\mathbb
N_+)_{k(P)}\}}$, will be now viewed as $\Omega$-valued arrays $X$, where we
for definiteness put $(X_\alpha)_P \df 0$ whenever $P\in\mathcal{L}$ and
$\alpha\injection{[k]}{\mathbb{N}_+}$ are such that $k\neq k(P)$.

A random structure $\rn{K}$ in $\mathcal{K}_{\mathbb{N}_+}$ is \emph{exchangeable} if the associated random array is exchangeable.

\begin{remark}\label{rmk:exchchar}
  Since $\rn K$ is completely determined by its finite marginals, to check whether $\rn{K}$ is exchangeable, it is sufficient
  to check that $\rn{K}\sim\rn{K}\cdot\sigma$ only for those $\sigma\in S_{\mathbb{N}_+}$ for which $\{n\in\mathbb{N}_+ \mid \sigma(n)\neq n\}$ is finite.

  This further implies that $\rn{K}$ is exchangeable if and only if for every $\ell\in\mathbb{N}$ and every $K,L\in\mathcal{K}_\ell$ with $K$ and $L$ isomorphic we have
  \begin{equation} \label{eq:extensions}
    \prob{\rn{K}\vert_{[\ell]}=K} = \prob{\rn{K}\vert_{[\ell]}=L}.
  \end{equation}

  Indeed, given $\alpha\in (\mathbb N_+)_{< \omega}$ and $\sigma\in S_{\mathbb N_+}$, pick
  $\ell$ to be large enough so that $\im(\alpha)\cup \im(\sigma\comp\alpha)\subseteq [\ell]$.
  Then we only have to replace $\sigma$ with an arbitrarily chosen $\sigma'\in S_\ell$ such that
  $\sigma\comp\alpha = \sigma'\comp\alpha$, and apply~\eqref{eq:extensions} to all possible
  extensions of $\rn{K}_\alpha$ to $\rn K\vert_{[\ell]}$.
\end{remark}

\begin{definition}
  Let $N$ be a canonical structure on $n$ vertices. The distribution $\rn{R}(N)$ over $\mathcal{K}_{\mathbb{N}_+}$ is defined by picking
  uniformly at random a labeling of $N$ by $[n]$ and completing it with isolated vertices. Formally, we pick $\rn{f}\injection{[n]}{V(N)}$ uniformly at random and define the random structure $\rn{R}(N)$ on $\mathbb{N}_+$ by letting
  \begin{align*}
    \alpha\in R_P(\rn{R}(N)) & \equiv \im(\alpha)\subseteq [n]\land \rn{f}\comp\alpha\in R_P(N),
  \end{align*}
  for every $P\in\mathcal{L}$ and every $\alpha\injection{[k(P)]}{\mathbb{N}_+}$.

  Furthermore, if $\rn{N}$ is itself a random canonical structure, then we define $\rn{R}(\rn{N})$ by picking $\rn{f}$ independently from $\rn{N}$.
\end{definition}

The next theorem (or, more exactly, its Corollary~\ref{cor:cryptoexch}) add extreme distributions of exchangeable
random structures in $\mathcal{K}_{\mathbb{N}_+}$ to the list of objects
cryptomorphic to convergent sequences (this connection was originally pointed
out independently by Diaconis and Janson~\cite{DiJa} and Austin\footnote{Unlike us, Austin also covers the case of flag algebra homomorphisms of non-zero
types.}~\cite{Aus}).

\begin{theorem}\label{thm:cryptoexch}
  If $\rn{\phi}$ is a probability distribution on the set $\HomT{T}\subseteq[0,1]^{\mathcal{M}[T]}$, then there
  exists an exchangeable
  probability distribution $\rn{K}$ over $\mathcal{K}_{\mathbb{N}_+}$ satisfying
  \begin{align}\label{eq:cryptoexch}
    \prob{\rn{K}\vert_{[m]} \cong M} & = \expect{\rn{\phi}(M)}
  \end{align}
  for every $M\in\mathcal{M}_m$. In particular, almost surely $\rn{K}$ is a model of $T$.

  Conversely, for every exchangeable probability distribution $\rn{K}$ over $\mathcal{K}_{\mathbb{N}_+}$ that is almost surely a model of $T$, there exists a probability distribution $\rn{\phi}$ over $\HomT{T}$ such that~\eqref{eq:cryptoexch} holds.

  Furthermore,~\eqref{eq:cryptoexch} gives a one-to-one correspondence between probability distributions over $\HomT{T}$ and distributions of exchangeable random structures in $\mathcal{K}_{\mathbb{N}_+}$ that are almost surely models of $T$.
\end{theorem}

\begin{proof}
  Suppose first that $\rn{\phi}$ is a probability distribution over $\HomT{T}$. For every $n\in\mathbb{N}$, we define the probability distribution $\rn{N_n}$ over $\mathcal{M}_n$ by
  \begin{align*}
    \prob{\rn{N_n} = N} & \df \expect{\rn{\phi}(N)}.
  \end{align*}

  Note furthermore that for every $M\in\mathcal{M}_m$ with $m\leq n$, we have
  \begin{align*}
    \expect{p(M,\rn{N_n})}
    & =
    \sum_{N\in\mathcal{M}_n} p(M,N)\expect{\rn{\phi}(N)}
    =
    \expect{\rn{\phi}(M)}.
  \end{align*}

  On the other hand, for $K\in\mathcal{K}_\ell$ and $N\in\mathcal{M}$ with $\lvert V(N)\rvert\geq\ell$ we have $\prob{\rn{R}(N)\vert_{[\ell]} = K} = \tind(K,N)\ (= (\lvert\Aut(K)\rvert/\ell!)\cdot p(K,N))$. Therefore, for every $n\geq\ell$, we have
  \begin{align}\label{eq:arrayversushom}
    \prob{\rn{R}(\rn{N_n})\vert_{[\ell]} = K}
    & =
    \expect{\tind(K,\rn{N_n})}
    =
    \frac{\lvert\Aut(K)\rvert}{\ell!}\expect{\rn{\phi}(K)}.
  \end{align}

  By Lemma~\ref{lem:arrayweakconv}, it follows that $(\rn{R}(\rn{N_n}))_{n\in\mathbb{N}}$ is weakly convergent; let $\rn{K}$ be its limit. By~\eqref{eq:arrayversushom}, it follows that if $K_1,K_2\in\mathcal{K}_\ell$ are isomorphic then $\prob{\rn{K}\vert_{[\ell]}=K_1}=\prob{\rn{K}\vert_{[\ell]}=K_2}$, hence $\rn{K}$ is exchangeable by Remark~\ref{rmk:exchchar}. Furthermore~\eqref{eq:arrayversushom} also implies~\eqref{eq:cryptoexch}, from which it follows that $\rn{K}$ is almost surely a model of $T$.

  \medskip

  Let us now prove the converse. Suppose $\rn{K}$ is an exchangeable probability distribution over $\mathcal{K}_{\mathbb{N}_+}$ that is almost surely a model of $T$. Then we define the probability distributions $\rn{K_n}$ over $\mathcal{K}_n$ as its marginals:
  \begin{align*}
    \prob{\rn{K_n} = K} & \df \prob{\rn{K}\vert_{[n]} = K},
  \end{align*}
  and we note that since $\rn{K}$ is exchangeable, for every $L\in\mathcal{K}_\ell$ with $\ell\leq n$, we have
  \begin{align*}
    \expect{p(L,\rn{K_n})}
    & =
    \frac{\ell!}{\lvert\Aut(L)\rvert}\sum_{K\in\mathcal{K}_n}\tind(L,K)\prob{\rn{K}\vert_{[n]} = K}
    \\
    & =
    \frac{\ell!}{\lvert\Aut(L)\rvert}\prob{\rn{K}\vert_{[\ell]} = L}
    \\
    & =
    \prob{\rn{K}\vert_{[\ell]} \cong L}.
  \end{align*}
  Hence by Lemma~\ref{lem:convexpect} $(p(\place,\rn{K_n}))_{n\in\mathbb{N}}$ is weakly convergent and its limit $\rn{\phi}$ is supported on $\Hom$ and satisfies~\eqref{eq:cryptoexch}.

  Finally, the one-to-one correspondence follows from the uniqueness statement of Lemma~\ref{lem:convexpect} and the fact that the distribution of $\rn{K}$ is uniquely determined by its marginals $\rn{K}\vert_{[\ell]}$.
\end{proof}

\begin{definition}
Recall that an {\em extreme point} of a convex set $S$ is a point that does not lie in any open segment joining two distinct points of $S$. We say that an exchangeable random structure $\rn{K}$ in $\mathcal{K}_{\mathbb{N}_+}$ is {\em extreme} if its distribution is an extreme point in the set of distributions of all such structures.
\end{definition}

\begin{corollary}\label{cor:cryptoexch}
  Let $\rn{\phi}$ be a random homomorphism in $\Hom$ and let $\rn{K}$ be an exchangeable random structure in $\mathcal{K}_{\mathbb{N}_+}$ with distribution corresponding to $\rn{\phi}$ according to Theorem~\ref{thm:cryptoexch}. Then $\rn{K}$ is extreme if and only if there exists $\phi\in\Hom$ such that $\rn{\phi}=\phi$ almost surely.
\end{corollary}

\begin{proof}
  Follows directly from the fact that the correspondence~\eqref{eq:cryptoexch} is linear w.r.t.\ convex combinations of probability measures and the obvious observation that extreme points in the space of {\em all} probability
  distributions on the set $\Hom$ are precisely as described.
\end{proof}

\begin{definition}
  Given an array $X$ indexed by $(\mathbb{N}_+)_{<\omega}$ and a set $I\subseteq\mathbb{N}_+$ (possibly, infinite), we define the restriction $X\vert_I$ as the restriction of $X$ to $(I)_{<\omega}$.

  A random array $\rn{X}$ indexed by $(\mathbb{N}_+)_{<\omega}$ is \emph{local} (or \emph{dissociated}) if for every pairwise disjoint $I_1,I_2,\ldots,I_k\subseteq\mathbb{N}_+$ the restrictions $\rn{X}\vert_{I_i}$ are mutually independent.

  We extend the definition of locality to random structures in $\mathcal{K}_{\mathbb{N}_+}$ via the correspondence with arrays.
\end{definition}

\begin{remark}\label{rmk:local}
  Note that if $\rn{X}$ is exchangeable, then to check whether $\rn{X}$ is local it is enough to test only the case $I_1=[m_1]$ and $I_2=\{m_1+1,\ldots,m_1+m_2\}$ for every $m_1,m_2\in\mathbb{N}$.

  Analogously, if $\rn{K}$ is an exchangeable random structure, then to check whether $\rn{K}$ is local it is enough to test if $\rn{K}\vert_{[m_1]}$ and $\rn{K}\vert_{\{m_1+1,\ldots,m_1+m_2\}}$ are independent for every $m_1,m_2\in\mathbb{N}$. Furthermore, by exchangeability it is sufficient to check
independence of the events $\rn K\vert_{[m_1]}\cong M_1$ and $\rn K\vert_{\{m_1+1,\ldots, m_1+m_2\}}\cong M_2$ for every pair $M_1\in\mathcal M_{m_1}$ and $M_2\in\mathcal M_{m_2}$.
\end{remark}

\begin{proposition}[\protect{cf.~\cite[Proposition~14.62]{Lov4}}]\label{prop:extremelocal}
  Suppose $\rn{K}$ is an exchangeable random structure in $\mathcal{K}_{\mathbb{N}_+}$. Then $\rn{K}$ is local if and only if it is extreme.
\end{proposition}

\begin{proof}
  Suppose first that $\rn{K}$ is extreme. Then by Corollary~\ref{cor:cryptoexch} we know that~\eqref{eq:cryptoexch} holds with a single element $\phi\in\Hom$. Now, for every $M_1\in\mathcal M_{m_1}$ and $M_2\in\mathcal M_{m_2}$, the desired equality
  \begin{align*}
    & \!\!\!\!\!\!
    \prob{\rn K\vert_{[m_1]}\cong M_1 \land \rn K\vert_{\{m_1+1,\ldots,m_1+m_2\}}\cong M_2}
    = \phi(M_1M_2)
    \\
    & = \phi(M_1)\phi(M_2)
    = \prob{\rn K\vert_{[m_1]}\cong M_1}\prob{\rn K\vert_{\{m_1+1,\ldots,m_1+m_2\}}\cong M_2}
  \end{align*}
  immediately follows from Definition~\ref{def:mult} of the product in flag algebras (and the fact that $\phi$ is an algebra homomorphism).

  \medskip

  In the opposite direction, if $\rn{K}$ is not extreme then~\eqref{eq:cryptoexch} holds for a distribution $\rn\phi$ on $\Hom$ that is not supported on any single point. The latter implies that for some $M\in\mathcal M_m$, $\rn\phi(M)$ is not supported on any single point and hence that $\text{Var}(\rn\phi(M))>0$. But now we have $\prob{\rn K\vert_{[m]}\cong M \land \rn K\vert_{[m+1,\ldots,2m]}\cong M} = \expect{\rn\phi(M^2)} > \expect{\rn\phi(M)}^2$. Hence $\rn K$ is not local.
\end{proof}

The next two theorems on exchangeable arrays are key for the theon existence. Recall from Definition~\ref{def:EV} that $r(V)$ and $\mathcal E_V$ are well-defined even if $V$ is countable, but even in that case $r(V)$ stands for the collection of non-empty {\em finite} subsets of $V$. We let $\rn\xi = (\rn \xi_A)_{A\in r(\mathbb N_+)}$ be drawn uniformly (w.r.t.\ the Lebesgue measure) from $\mathcal E_{\mathbb N_+}$ and let $\rn\eta$ be drawn uniformly from $[0,1]$, independently of $\rn\xi$. We also let $\mathcal E_V^+\df [0,1]\times \mathcal E_V$.

\begin{theorem}[Hoover~\cite{Hoo}, see also~{\cite[Theorem~7.22]{Kall}}]\label{thm:exchrepr}
  Let $\Omega$ be a Polish space and let
  $\rn{X}=(\rn{X_\alpha})_{\alpha\in(\mathbb{N}_+)_{<\omega}}$ be an $\Omega$-valued exchangeable random array
  indexed by $(\mathbb{N}_+)_{<\omega}$.

  Then there exist measurable functions
  \begin{align*}
    \chi_k\function{\mathcal E_k^+}{\Omega}\quad (k\in\mathbb N_+)
  \end{align*}
such that $\rn X$ is equidistributed with the random array $\rn Y$ given by
  \begin{align}\label{eq:exchrepr}
    \rn{Y_\alpha} & \df \chi_{\lvert\alpha\rvert}(\rn\eta,\alpha^*(\rn{\xi}))\quad (\alpha\in(\mathbb{N}_+)_{<\omega}).
  \end{align}
\end{theorem}

The theorem below proved first by Aldous~\cite{Ald2} for arrays indexed by $\mathbb{N}_+^2$ and extended for
arrays indexed by $\mathbb{N}_+^k$ ($k\geq 1$) by Kallenberg~\cite[Lemma~7.35]{Kall} says that in the local
case, we can remove the dependency on $\rn{\eta}$. We provide an ad hoc proof from
Theorem~\ref{thm:exchrepr} for the case of arrays indexed by $(\mathbb{N}_+)_{<\omega}$.

\begin{theorem}[Aldous~\cite{Ald2}, Kallenberg~{\cite[Lemma~7.35]{Kall}}]\label{thm:localexchrepr}
  Under the assumptions of Theorem~\ref{thm:exchrepr}, the array $\rn{X}$ is local if and only if there exist
  measurable functions
  \begin{align*}
    \chi_k\function{\mathcal E_k}{\Omega}\quad (k\in\mathbb N_+)
  \end{align*}
such that $\rn X$ is equidistributed with the random array $\rn Y$ given by
  \begin{align}\label{eq:localexchrepr}
    \rn{Y_\alpha} & \df \chi_{\lvert\alpha\rvert}(\alpha^*(\rn{\xi}))\quad (\alpha\in(\mathbb{N}_+)_{<\omega}).
  \end{align}
\end{theorem}

\begin{proof}
  For the ``if'' part, note that~\eqref{eq:localexchrepr} implies that $\rn{Y}\vert_I$ depends only on $\{\rn{\xi_A} \mid A\in r(I)\}$, hence if $I_1\cap I_2 = \emptyset$, then $\rn{Y}\vert_{I_1}$ and $\rn{Y}\vert_{I_2}$ are independent.

  For the ``only if'' part, let $\chi_k$ be as in Theorem~\ref{thm:exchrepr}, so that the random array $\rn Y$ defined by~\eqref{eq:exchrepr} is local. First, we claim that $\rn{Y}$ is independent of $\rn{\eta}$.
To prove this, we need to show that for every Borel set $C\subseteq[0,1]$ with $0<\lambda(C)<1$, for every $\ell\in\mathbb{N}_+$ and every Borel set $B\subseteq
  \Omega^{([\ell])_{<\omega}}$, we have
  \begin{align*}
    \prob{\rn{Y}\vert_{[\ell]} \in B \given \rn{\eta}\in C} & =
    \prob{\rn{Y}\vert_{[\ell]} \in B \given \rn{\eta}\notin C}
  \end{align*}
  (recall once more that random arrays are uniquely determined by their finite marginals).

  Let $\alpha$ and $\beta$ be the left and right-hand sides in the above expression respectively. Then
  \begin{align*}
   \prob{\rn{Y}\vert_{[\ell]}\in B} &= p\alpha + (1-p)\beta,
  \end{align*}
where $p\df\lambda(C)$.

  Let now $\sigma\in S_{\mathbb{N}_+}$ be a permutation such that $\sigma(i) = \ell+i$ for every $i\in[\ell]$
  and note that the events $\rn{Y}\vert_{[\ell]}\in B$, $ (\rn{Y}\cdot\sigma)\vert_{[\ell]}\in B$
  are conditionally independent given $\rn{\eta}$ since by~\eqref{eq:exchrepr}, the first one depends only on $\{\rn{\eta}\}\cup\{\rn{\xi_A} \mid A\in r(\ell)\}$, and the second one only on $\{\rn{\eta}\}\cup\{\rn{\xi_A} \mid A\in r(\{\ell+1,\ldots,2\ell\})\}$. Furthermore,~\eqref{eq:exchrepr} readily implies that
  $\rn Y$ remains exchangeable after conditioning on any event that depends on $\rn\eta$ only.
  In particular,
  \begin{align*}
    & \!\!\!\!\!\!
    \prob{\rn{Y}\vert_{[\ell]}\in B\land(\rn{Y}\cdot\sigma)\vert_{[\ell]}\in B}
    \\
    & =
    p\prob{\rn{Y}\vert_{[\ell]}\in B\land(\rn{Y}\cdot\sigma)\vert_{[\ell]}\in B \given \rn{\eta}\in C}
    \\ & \qquad
    +
    (1-p)
    \prob{\rn{Y}\vert_{[\ell]}\in B\land(\rn{Y}\cdot\sigma)\vert_{[\ell]}\in B \given \rn{\eta}\notin C}
    \\
    & =
    p\alpha^2 + (1-p)\beta^2.
  \end{align*}

  From strict convexity of $x\mapsto x^2$, it follows that
  \begin{align*}
    \prob{\rn{Y}\vert_{[\ell]}\in B\land(\rn{Y}\cdot\sigma)\vert_{[\ell]}\in B}
    & \geq
    (p\alpha + (1-p)\beta)^2
    =
    \prob{\rn{Y}\vert_{[\ell]}\in B}^2,
  \end{align*}
  with equality if and only if $\alpha=\beta$. But the locality of $\rn X$ implies that
  we do have equality here, hence indeed $\alpha=\beta$ and thus $\rn Y$ is independent of
  $\rn\eta$.

  The rest is a routine exercise in measure theory. First of all, since $\Omega$ is Polish,
  it has a countable base and hence Fubini's theorem implies that for almost all $x\in [0,1]$,
  all fiber functions $\chi_k^x\df \chi_k(x,\place)$ are measurable and hence we can form random arrays $\rn Y(x)$
  by $\rn Y_\alpha(x) = \chi_{\lvert\alpha\rvert}(x,\alpha^\ast(\rn\xi))$. We claim that for almost all $x$, $\rn Y(x)$ is equidistributed with $\rn Y$.

  Since the space of $\Omega$-valued arrays indexed by $(\mathbb N_+)_{<\omega}$ is also
  Polish, it suffices to check that $\prob{\rn Y\in A}= \prob{\rn Y(x)\in A}$ a.e.\ for any fixed
  Borel set $A$ in this space. But for any fixed $n$, the sets $\{x\in [0,1] \mid \prob{\rn Y(x)\in A} > \prob{\rn Y\in A}+ 1/n\}$ and $\{x\in [0,1] \mid \prob{\rn Y\in A} > \prob{\rn Y(x)\in A}+ 1/n\}$ must
  have measure~0 since otherwise by Fubini's theorem we would get a contradiction with the independence
  we have just proven. Hence indeed $\prob{\rn Y(x)\in A} =\prob{\rn Y\in A}$ for almost all
  $x\in [0,1]$, which completes the proof.
\end{proof}

Let us finally show how Theorem~\ref{thm:localexchrepr} implies theon existence.

\begin{proofof}{Theorem~\ref{thm:theoncryptomorphism}}
  For an element $\phi\in\Hom$, by Corollary~\ref{cor:cryptoexch} and Proposition~\ref{prop:extremelocal}, let $\rn{K}$ be a local exchangeable random structure in $\mathcal{K}_{\mathbb{N}_+}$ with distribution corresponding to $\phi$ via~\eqref{eq:cryptoexch}.

  Recall that we view $\rn K$ as a local exchangeable random array $\rn X$ indexed by $(\mathbb{N}_+)_{<\omega}$ with values in $\Omega\df \{0,1\}^{\mathcal L}$, where $(\rn X_\alpha)_P$ is the
  characteristic function of the event $\alpha\in R_P(\rn K)$ if $k(P)= \lvert\alpha\rvert$ and defined
  arbitrarily if $k(P)\neq \lvert\alpha\rvert$. By Theorem~\ref{thm:localexchrepr}, there exist measurable functions $\chi_k\function{\mathcal{E}_k}{\Omega}$ such that~\eqref{eq:localexchrepr} holds.

  Define the Euclidean structure $\mathcal{N}$ by letting
  \begin{align*}
    \mathcal{N}_P & \df \{x\in\mathcal{E}_{k(P)} \mid \chi_{k(P)}(x)_P = 1\},
  \end{align*}
  for every $P\in\mathcal{L}$, where $\chi_k(x)_P$ denotes the $P$th coordinate of $\chi_k(x)$.

  By~\eqref{eq:cryptoexch}, for every (unordered) model $M\in\mathcal M_m$, $\phi(M) = \prob{\rn K\vert_{[m]} \cong  M}$, and (see Definition~\ref{def:peons}), $p(M,\mathcal N)= \frac{\lvert V(M)\rvert!}{\lvert\Aut(M)\rvert}\lambda(\Tind(M,\mathcal N))$. Thus, passing to the labeled case, we have to prove that
\begin{equation} \label{eq:tobeproven}
\prob{\rn K\vert_{[m]} =K} = \lambda(\Tind(K, \mathcal N))
\end{equation}
for any $K\in\mathcal K_m$.
The latter quantity, however, can be interpreted as $\prob{i_m^\ast(\rn\xi) \in \Tind(K,\mathcal N)}$, where $i_m\injection{[m]}{\mathbb N_+}$ is the natural inclusion, and now the events on both sides of~\eqref{eq:tobeproven} are identical. Namely, for every $P\in\mathcal L$ and every $\alpha\injection{[k(P)]}{[m]}$, $\chi_{k(P)}(\alpha^\ast(\rn\xi))_P=1$ if and only if $\alpha\in R_P(\rn K)$.

  In the opposite direction, if $\mathcal{N}$ is a weak $T$-on, then we define the random structure $\rn{K}$ in $\mathcal{K}_{\mathbb{N}_+}$ by letting
  \begin{align*}
    \alpha\in R_P(\rn{K}) & \equiv \alpha^*(\rn{\xi})\in\mathcal{N}_P,
  \end{align*}
  for every $P\in\mathcal{L}$ and every $\alpha\injection{[k(P)]}{\mathbb{N}_+}$.

  Reversing the above argument, for every structure $K\in\mathcal{K}_m$ we have
  \begin{align*}
    \tind(K,\mathcal{N}) & = \prob{\rn{K}\vert_{[m]} = K}
  \end{align*}
  and hence also
  \begin{align*}
    \tinj(K,\mathcal{N}) & = \prob{\rn{K}\vert_{[m]}\supseteq K};\\
    p(K,\mathcal{N}) & = \prob{\rn{K}\vert_{[m]} \cong K}.
  \end{align*}

  By Theorem~\ref{thm:localexchrepr}, this implies that the corresponding $\rn X$ is a local exchangeable array, and by Proposition~\ref{prop:extremelocal} and Corollary~\ref{cor:cryptoexch}, we know that $\rn{X}$ corresponds to a homomorphism $\phi\in\Hom$. Theorem~\ref{thm:cryptomorphism} then concludes the proof.
\end{proofof}

\bigskip

The next task is to prove theon uniqueness that will require extending Definition~\ref{def:hoa}.

\begin{definition}[Definition~\ref{def:hoa}, cntd.] For a function $f$ on $\mathcal E_V^+\ (= [0,1]\times \mathcal E_V)$, we will abbreviate its $x$th fiber $f(x,\place)$ as $f^x$. A function $f\function{\mathcal E_V^+}{[0,1]}$
is \emph{symmetric} if $f^{x}$ is symmetric for every $x\in[0,1]$. Furthermore, the function $f$ is \emph{measure preserving on highest order argument} (h.o.a.) if it is measurable and $f^{x}$ is measure preserving on h.o.a.\ for every $x\in[0,1]$.

  Suppose now that $f = (f_d)_{d\in\mathbb{N}}$ is a family of symmetric functions with $f_d\function{\mathcal{E}^+_d}{[0,1]}$. Then we define a new sequence $\widehat{f} = (\widehat{f}_d)_{d\in\mathbb{N}}$ with $\widehat{f}_d\function{\mathcal{E}^+_d}{\mathcal{E}^+_d}$ by
  \begin{equation} \label{eq:locality}
    \widehat{f}_d(x,y) \df (f_0(x), \widehat{(f^{x})}_d(y)).
  \end{equation}

  Analogously, for $h\function{\mathcal{E}_V^+\times \mathcal{E}_V^+}{[0,1]}$ and for $(x,x')\in [0,1]^2$, we define the function
  \begin{align*}
    \begin{functiondef}
      h^{x,x'}\colon & \mathcal{E}_V\times \mathcal{E}_V & \longrightarrow & [0,1]\\
      & (y,y') & \longmapsto & h((x,y),(x',y')).
    \end{functiondef}
  \end{align*}

  The function $h$ is \emph{symmetric} if $h^{x,x'}$ is symmetric for every $(x,x')\in[0,1]^2$. Furthermore, the function $h$ is \emph{measure preserving on highest order argument} (h.o.a.) if it is measurable and $h^{x,x'}$ is measure preserving on h.o.a.\ for every $(x,x')\in[0,1]^2$.

  If $h = (h_d)_{d\in\mathbb{N}}$ is a family of symmetric functions with $h_d\function{\mathcal{E}_d^+\times \mathcal{E}_d^+}{[0,1]}$, then we define the sequence $\widehat{h}=(\widehat{h}_d)_{d\in\mathbb{N}}$ by
  \begin{align*}
    \widehat{h}_d((x,y),(x', y'))
    & \df
    (h_0(x, x'), \widehat{(h^{x,x'})}_d(y,y')).
  \end{align*}
\end{definition}

The following theorem by Hoover~\cite{Hoo} and Kallenberg~\cite{Kall2} (see also~\cite[Lemma~7.28]{Kall}) characterizes equidistributed exchangeable arrays. Recall that $\rn\xi$ is uniformly distributed in $\mathcal E_{\mathbb N_+}$ and $\rn\eta$ is uniformly distributed in $[0,1]$; thus, the pair $(\rn\eta, \rn\xi)$ defines a uniform distribution over $\mathcal E_{\mathbb N_+}^+$.

\begin{theorem}[Hoover~\cite{Hoo}, Kallenberg~\cite{Kall2}]\label{thm:exchuniqueness}
  Let $\Omega$ be a Polish space and let
  \begin{align*}
    \chi_d,\chi_d'\function{{\mathcal{E}}^+_d}{\Omega}
  \end{align*}
  be measurable functions. Define the random (exchangeable) arrays $\rn{X}$ and $\rn{X'}$ indexed by $(\mathbb{N}_+)_{<\omega}$ by letting
  \begin{align*}
    \rn{X}_\alpha & = \chi_{\lvert\alpha\rvert}(\rn\eta,\alpha^*(\rn{\xi})); &
    \rn{X'}_\alpha & = \chi_{\lvert\alpha\rvert}'(\rn{\eta},\alpha^*(\rn{\xi})).
  \end{align*}

  Then the following are equivalent.
  \begin{itemize}
  \item The arrays $\rn{X}$ and $\rn{X'}$ have the same distribution.

  \item There exist families $f = (f_d)_{d\in\mathbb{N}}$ and $g = (g_d)_{d\in\mathbb{N}}$ of symmetric functions measure preserving on h.o.a., $f_d,g_d\function{\mathcal{E}^+_d}{[0,1]}\ (d\in\mathbb{N})$, such that
    \begin{align*}
      \chi_d(\widehat{f}_d(x,y)) & = \chi_d'(\widehat{g}_d(x,y)),
    \end{align*}
    for every $d\in\mathbb{N}$ and almost every $x\in [0,1],\ y\in\mathcal{E}_d$.
  \item There exists a family $h = (h_d)_{d\in\mathbb{N}}$ of symmetric functions measure preserving on h.o.a., $h_d\function{\mathcal{E}^+_d\times \mathcal{E}^+_d}{[0,1]}\ (d\in\mathbb{N})$, such that
    \begin{align*}
      \chi_d(\widehat{h}_d((x,y), (x',y'))) & = \chi'_d(x,y)
    \end{align*}
    for every $d\in\mathbb{N}$ and almost every $x,x'\in [0,1]$ and $y,y'\in \mathcal{E}_d$.
  \end{itemize}
\end{theorem}

As the reader may have noticed, Theorems~\ref{thm:TheonUniqueness} and~\ref{thm:exchuniqueness} are very similar, with the difference that the latter uses extra variables $x\in [0,1]$. The proof of theon uniqueness below consists of a standard measure-theoretic trick to remove these extra variables (cf.~the proof of Theorem~\ref{thm:localexchrepr}).

\begin{proofof}{Theorem~\ref{thm:TheonUniqueness}}
  By the definition of $\phi_{\mathcal{N}}$ and $\phi_{\mathcal{N}'}$, \ref{it:fgequiv}~$\Longrightarrow$~\ref{it:densityequiv} is straightforward.

  Assume now that $\phi_{\mathcal{N}}=\phi_{\mathcal{N}'}$, and let us prove~\ref{it:fgequiv}. Define random canonical structures $\rn{K}$ and $\rn{K}'$ from our theons as in the proof of Theorem~\ref{thm:theoncryptomorphism}. That is, for every $P\in\mathcal{L}$ and every $\alpha\injection{[k(P)]}{\mathbb{N}_+}$, we let
  \begin{align*}
    \alpha\in R_P(\rn{K}) & \equiv \alpha^*(\rn{\xi})\in\mathcal{N}_P; &
    \alpha\in R_P(\rn{K}') & \equiv \alpha^*(\rn{\xi})\in\mathcal{N}'_P.
  \end{align*}

  Note that for every $L\in\mathcal{K}_\ell$, we have
  \begin{align*}
    \prob{\rn{K}\vert_{[\ell]}=L} & = \frac{\ell!}{\lvert\Aut(L)\rvert} \phi_{\mathcal{N}}(L)
    = \frac{\ell!}{\lvert\Aut(L)\rvert} \phi_{\mathcal{N}'}(L) = \prob{\rn{K}'\vert_{[\ell]}=L},
  \end{align*}
  hence $\rn{K}\sim\rn{K}'$. In particular, random exchangeable arrays $\rn X$ and $\rn X'$ derived from $\rn K$ and $\rn K'$ also have the same distribution. But the specific way in which $\rn K,\ \rn K'$ were constructed also provides us with a natural representation of the arrays $\rn X,\ \rn X'$ as required in Theorem~\ref{thm:exchuniqueness}. Namely, define first the functions $\zeta_d,\zeta_d'\function{\mathcal E_d}{\Omega}$ by letting
  \begin{align*}
    \zeta_d(y)_P & \df
    \begin{dcases*}
      1, & if $y\in\mathcal{N}_P$;\\
      0, & otherwise;
    \end{dcases*}
    &
    \zeta_d'(y)_P & \df
    \begin{dcases*}
      1, & if $x\in\mathcal{N}'_P$;\\
      0, & otherwise
    \end{dcases*}
  \end{align*}
  for every $y\in\mathcal E_d$ and $P\in\mathcal{L}$ with $k(P) = d$; as always, the values $\zeta_d(y)_P$ can be chosen arbitrarily when $k(P)\neq d$.

  Next, define the functions $\chi_d,\chi_d'\function{\mathcal{E}^+_d}{\Omega}$ by adding $x$ as a dummy variable:
  \begin{align*}
    \chi_d(x,y) & = \zeta_d(y); &
    \chi_d'(x,y) & = \zeta_d'(y).
  \end{align*}
  for every $d\in\mathbb{N}$, every $x\in [0,1]$ and every $y\in\mathcal{E}_d$.

  By Theorem~\ref{thm:exchuniqueness}, there exist families $f = (f_d)_{d\in\mathbb{N}}$ and $g = (g_d)_{d\in\mathbb{N}}$ of symmetric functions measure preserving on h.o.a.\ with $f_d,g_d\function{\mathcal{E}^+_d}{[0,1]}$ such that
  \begin{equation} \label{eq:chi}
    \chi_d(\widehat{f}_d(x,y)) = \chi_d'(\widehat{g}_d(x,y)),
  \end{equation}
  for every $d\in\mathbb{N}$ and almost every $(x,y)\in\mathcal{E}^+_d$.
  As in the proof of Theorem~\ref{thm:localexchrepr}, we have to get rid of the first argument but this time it more or less immediately follows from the fact that the definition~\eqref{eq:locality} is local. Formally,
$$
\chi_d(\widehat f_d(x,y)) = \chi_d(f_0(x), (\widehat{f^x})_d(y)) = \zeta_d((\widehat{f^x})_d(y)),
$$
and likewise for $\chi_d'$. Hence we have
\begin{equation} \label{eq:33}
\zeta_d((\widehat{f^x})_d(y)) = \zeta'_d((\widehat{g^x})_d(y))\quad \text{a.e.}
\end{equation}
Since all our spaces are Polish, we can apply Fubini's theorem and find a particular $x_0$ such that for every $d$ the functions $f^{x_0}_d,\ g^{x_0}_d$ are measurable and~\eqref{eq:33} holds for almost every $y\in\mathcal E_d$.
The families $(f^{x_0}_d)_{d=1}^k$ and $(g^{x_0}_d)_{d=1}^k$, where $k=\max\{k(P) \mid P\in\mathcal{L}\}$ have the required properties, and the $P$-on $\mathcal N_P''$ is provided by~\eqref{eq:chi}: $y\in\mathcal N''_P$ if and only if (say) $\chi_{k(P)}(x_0,y)_P=1$.
\end{proofof}

The proof of Theorem~\ref{thm:TheonUniquenessSecond} from Theorem~\ref{thm:exchuniqueness} is analogous.

\section{One final cryptomorphism: ergodicity}
\label{sec:othercrypto}

In this section we will prove an ergodicity property and wrap up the list of cryptomorphic objects in the style of~\cite[Theorem~11.52]{Lov4}.

\begin{definition}
  Let us call a random structure $\rn{K}$ in $\mathcal{K}_{\mathbb{N}_+}$ \emph{weakly ergodic} if for every $S_{\mathbb{N}_+}$-invariant Borel set $A$ (i.e., $A = A\cdot\sigma$ for every $\sigma\in S_{\mathbb{N}_+}$), the event $\rn{K}\in A$ is trivial, that is $\prob{\rn{K}\in A}\in\{0,1\}$.

  Let $S_{\mathbb{N}_+}^*$ be the subgroup of all permutations in $S_{\mathbb{N}_+}$ that fix all but finitely many elements of $\mathbb{N}_+$.
  A random structure $\rn{K}$ in $\mathcal{K}_{\mathbb{N}_+}$ is \emph{strongly ergodic} if for every $S_{\mathbb{N}_+}^*$-invariant Borel set $A$, the event $\rn{K}\in A$ is trivial.
\end{definition}

Since an $S_{\mathbb{N}_+}$-invariant set is clearly an $S_{\mathbb{N}_+}^*$-invariant set, it follows that strong ergodicity
implies weak ergodicity. It is also important to note that although the concepts of $S_{\mathbb{N}_+}$-invariant and $S_{\mathbb{N}_+}^*$-invariant \emph{distributions} over $\mathcal{K}_{\mathbb{N}_+}$ are equivalent (cf.~Remark~\ref{rmk:exchchar}), the notions of $S_{\mathbb{N}_+}$-invariant and $S_{\mathbb{N}_+}^*$-invariant sets are not the same, see Examples~\ref{ex:bernoulli} and~\ref{ex:rado} below.

\begin{proposition}\label{prop:ergodic}
  Let $\rn{K}$ be an exchangeable random structure in $\mathcal{K}_{\mathbb{N}_+}$. The following are equivalent.
  \begin{enumerate}[label={\arabic*.}, ref={\arabic*)}]
  \item The structure $\rn{K}$ is strongly ergodic.\label{it:erg}
  \item The structure $\rn{K}$ is local.\label{it:local}
  \item The structure $\rn{K}$ is extreme.\label{it:extr}
  \end{enumerate}
\end{proposition}

\begin{proof}
  The equivalence $\ref{it:local}\equiv\ref{it:extr}$ is the content of Proposition~\ref{prop:extremelocal}.

  Suppose the distribution $D$ of $\rn{K}$ is not an extreme point, then there are distinct distributions $D_1$ and $D_2$ of exchangeable random structures $\rn{K_1}$ and $\rn{K_2}$ such that $D=(D_1+D_2)/2$.

  Let $(P,N)$ be a Hahn decomposition (see e.g.~\cite[Theorem~I.3.1.1]{Bog}) of the signed measure $D_1-D_2$, that is, $P$ and $N$ are Borel sets
  such that
  \begin{itemize}
  \item $\mathcal{K}_{\mathbb{N}_+} = P\mathbin{\stackrel\cdot\cup} N$;
  \item for every Borel set $A\subseteq P$, we have $D_1(A)\geq D_2(A)$;
  \item for every Borel set $A\subseteq N$, we have $D_1(A)\leq D_2(A)$.
  \end{itemize}

  Let then
  \begin{align*}
    P' & \df \bigcup_{\sigma\in S_{\mathbb{N}_+}^*}\sigma\cdot P, &
    N' & \df \mathcal{K}_{\mathbb{N}_+}\setminus P' = \bigcap_{\sigma\in S_{\mathbb{N}_+}^*}\sigma\cdot N,
  \end{align*}
  and note that since $S_{\mathbb{N}_+}^*$ is countable (this is how we use \emph{strong} ergodicity), these sets are Borel. Clearly these sets are also $S_{\mathbb{N}_+}^*$-invariant.
  We claim that $(P',N')$ is another Hahn decomposition of $D_1-D_2$.

  Firstly, since $N'\subseteq N$, if $A\subseteq N'$ is a Borel set, then clearly $D_1(A)\leq D_2(A)$. Thus, it remains to prove that if $A\subseteq P'$ is a Borel set, then $D_1(A)\geq D_2(A)$.

  Fix an enumeration $(\sigma_n)_{n\in\mathbb{N}}$ of $S_{\mathbb{N}_+}^*$ and define the sets
  \begin{align*}
    A_n & \df (\sigma_n\cdot P)\cap \left(A\setminus\bigcup_{m=0}^{n-1} A_m\right)
  \end{align*}
  inductively. Since $\sigma_n^{-1}\cdot A_n\subseteq P$ and $D_1$ and $D_2$ are $S_{\mathbb{N}_+}$-invariant, we have $D_1(A_n)\geq D_2(A_n)$. Since $A = \mathop{\stackrel\cdot\bigcup}_{n\in\mathbb{N}} A_n$, we get
  \begin{align*}
    D_1(A) & = \sum_{n\in\mathbb{N}} D_1(A_n) \geq \sum_{n\in\mathbb{N}} D_2(A_n) = D_2(A)
  \end{align*}
  as desired. Therefore $(P',N')$ is also a Hahn decomposition of $D_1-D_2$ as claimed above.

  We claim now that $D(P')\notin\{0,1\}$. Indeed, if $D(P')=1$, then $D(N')=0$, which implies that $D_1(A)\geq D_2(A)$ for every Borel set $A\subseteq\mathcal{K}_{\mathbb{N}_+}$. Since $D_1$ and $D_2$ are probability measures, by taking complements we get $D_1=D_2$, contradicting our assumption. Analogously, $D(P')=0$ implies the same contradiction $D_1=D_2$.
Therefore $D(P')\notin\{0,1\}$, hence $D$ is not strongly ergodic as $P'$ is $S_{\mathbb{N}_+}^*$-invariant.

   \medskip
   Conversely, if $D$ is not strongly ergodic and $A$ is an $S_{\mathbb{N}_+}^*$-invariant set with $0 < D(A) < 1$, then $D$ is a convex combination of the exchangeable distributions $D_1$ and $D_2$ defined by
   \begin{align*}
     D_1(B) & = \frac{D(B\cap A)}{D(A)}; &
     D_2(B) & = \frac{D(B\setminus A)}{1 - D(A)};
   \end{align*}
   for every Borel set $B\subseteq\mathcal{K}_{\mathbb{N}_+}$.
\end{proof}

The following examples show that not all weakly ergodic random structures are local.

\begin{example} \label{ex:bernoulli}
  Consider the theory $T_{2\operatorname{-Coloring}}$ and note that $\mathcal{K}_{\mathbb{N}_+}$ is naturally
  identified with $\{0,1\}^{\mathbb{N}_+}$. Note that $x,y\in\{0,1\}^{\mathbb{N}_+}$ are in the same $S_{\mathbb{N}_+}$-orbit if and only if $\lvert x^{-1}(0)\rvert = \lvert y^{-1}(0)\rvert$ and $\lvert x^{-1}(1)\rvert = \lvert y^{-1}(1)\rvert$ (either one of these four quantities can be infinite, of course). In particular, there are countably many orbits.

  For $p\in(0,1)$, let $D_p$ be the distribution of the random structure in
  $\mathcal{K}_{\mathbb{N}_+}$ corresponding to the homomorphism
  $\phi_p\in\HomT{T_{2\operatorname{-Coloring}}}$ in which a fraction $p$ of the vertices
  has color $\chi_0$ (via the identification made above, $D_p$ is simply the product of Bernoulli
  distributions with parameter $p$). Note that the $S_{\mathbb{N}_+}$-orbit of sequences that have infinitely many zeros and ones has $D_p$-measure $1$ and all other orbits have $D_p$-measure $0$.

  Take $p,q\in(0,1)$ distinct and let $D = (D_p + D_q)/2$. Then any $S_{\mathbb{N}_+}$-invariant Borel set $A\subseteq\mathcal{K}_{\mathbb{N}_+}$ must be a union of orbits, hence must have $D$-measure either $0$ or $1$ depending only on whether $A$ contains the orbit of infinitely many zeros and ones. Therefore, $D$ is weakly ergodic. It is not strongly ergodic by Proposition~\ref{prop:ergodic}; more explicitly, the $D$-measure of the $S_{\mathbb{N}_+}^*$-invariant set
  $$
\left\{ x\in \{0,1\}^{\mathbb N^+} \;\middle\vert\; \limsup_{n\to\infty} \frac{\lvert x^{-1}(1)\cap [n]\rvert}{n} \geq \frac{p+q}{2} \right\}
  $$
  is $1/2$.
\end{example}

\begin{example} \label{ex:rado}
  Consider the theory $T_{\operatorname{Graph}}$ and let $D_p$ be the distribution of the random structure in $\mathcal{K}_{\mathbb{N}_+}$ corresponding to the almost sure limit of $(\rn{G_{n,p}})_{n\in\mathbb{N}}$ (cf.~Example~\ref{ex:ErdosRenyi}). For $p\in(0,1)$, the distribution $D_p$ is concentrated on the $S_{\mathbb{N}_+}$-orbit of Rado graphs (see e.g.~\cite{Cam}). As in Example~\ref{ex:bernoulli}, the distribution $D = (D_p + D_q)/2$ for $p,q\in(0,1)$ distinct is non-local (and hence not strongly ergodic), but it is weakly ergodic.

\smallskip

  One can also generalize this example to the theory $T_{k\operatorname{-Hypergraph}}$. Let $D_p$ be the distribution of the random model $\rn{K_p}$ over $\mathbb{N}_+$ in which each hyperedge is present independently with probability $p\in(0,1)$. Just as in the graph case, one can show that $\rn{K_p}$ satisfies the following extension property with probability $1$: for every finite $k$-hypergraph $H$, every $W\subseteq V(H)$ and every embedding $f\function{W}{\mathbb{N}_+}$ of $H\vert_W$ in $\rn{K_p}$, we can extend $f$ to an embedding of $H$ in $\rn{K_p}$.

  Then, by a straightforward application of the back-and-forth method, one can prove that all hypergraphs over $\mathbb{N}_+$ that satisfy this extension property are isomorphic to each other and form an $S_{\mathbb{N}_+}$-orbit; they make a perfect hypergraph analogue of Rado graphs. Again, the same construction $D = (D_p + D_q)/2$ yields a non-local weakly ergodic random model in the theory $T_{k\operatorname{-Hypergraph}}$.
\end{example}

\begin{theorem}\label{thm:allcrypto}
  Consider the following objects for a theory $T$.
  \begin{enumerate}[label={\arabic*.}, ref={\arabic*)}]
  \item A convergent sequence of models $(N_n)_{n\in\mathbb{N}}$.
  \item A flag algebra homomorphism $\phi\in\HomT{T}$.
    \label{it:flagalgebracrypto}
  \item A $T$-on $\mathcal{N}$.
  \item A local exchangeable random structure $\rn{K}$ in $\mathcal{K}_{\mathbb{N}_+}$ supported on models of $T$.
  \item A strongly ergodic exchangeable random structure $\rn{K}$ in $\mathcal{K}_{\mathbb{N}_+}$ supported on models of $T$.
  \item An extreme exchangeable random structure $\rn{K}$ in $\mathcal{K}_{\mathbb{N}_+}$ supported on models of $T$. \label{it:crypto}
  \end{enumerate}

  The objects above are cryptomorphic in the sense that given an instance of one of them,
  one can construct instances of the others that satisfy the following for every $K\in\mathcal{K}_\ell$:
  \begin{gather*}
    \lim_{n\to\infty}p(K,N_n)
    = \phi(K)
    = p(K,\mathcal{N})
    = \prob{\rn{K}\vert_{[\ell]}\cong K}.
    \\
    \lim_{n\to\infty}\tind(K,N_n)
    = \frac{\lvert\Aut(K)\rvert}{\ell!}\phi(K)
    = \tind(K,\mathcal{N})
    = \prob{\rn{K}\vert_{[\ell]}= K}.
    \\
    \lim_{n\to\infty}\tinj(K,N_n)
    = \sum_{K'\supseteq K}\frac{\lvert\Aut(K')\rvert}{\ell!}\phi(K')
    = \tinj(K,\mathcal{N})
    = \prob{\rn{K}\vert_{[\ell]}\supseteq K}.
  \end{gather*}
\end{theorem}

\begin{proof}
  This is the content of Theorem~\ref{thm:cryptomorphism}, Theorem~\ref{thm:theoncryptomorphism}, Corollary~\ref{cor:cryptoexch}, Proposition~\ref{prop:extremelocal} and Proposition~\ref{prop:ergodic}.
\end{proof}

Comparing with Lov\'{a}sz's list for ordinary graphs~\cite[Theorem~11.52]{Lov4}, we see two
omissions in our treatment: consistent \emph{finite} random models and convergence w.r.t.\
the cut-distance. The former are omitted because they are ``trivially'' cryptomorphic to
homomorphisms in flag algebras (item~\ref{it:flagalgebracrypto} in Theorem~\ref{thm:allcrypto}). The situation with
cut-distance is, however, way more intriguing, and to the best of our knowledge, no
unambiguous and useful analogue of it is known even for 3-graphs. We will return to this discussion in the concluding section~\ref{sec:conclusion}.

\section{Other limit objects} \label{sec:otherobjects}

Let us now see a few concrete examples of how to connect theons with several
limit objects previously considered in the literature other than graphons,
digraphons and hypergraphons. Most of them are defined on ``nice'' $\sigma$-algebras, but the probability measures involved are normally entirely out of our control. In particular, a priori we do not have any idea how their completion may look like, and that adds additional measure-theoretical subtleties to be taken care of. This is our first order of business.

\subsection{Measure-theoretic background} \label{sec:measure_background}

\smallskip

In all definitions and results presented so far, there is nothing special about using the unit interval $[0,1]$ as the underlying space for the coordinates of Euclidean structures and peons. In fact,
in definitions we can use instead any probability space $\Omega=(X,\mathcal{A},\mu)$. In order for our results to hold, however, we need a few more assumptions on $\Omega$.

\begin{definition}
 Recall that an \emph{atom} of $\Omega$ is a measurable set $A\in\mathcal{A}$ such that $\mu(A)>0$ and every measurable set $B\in\mathcal{A}$ contained in $A$ has either measure $0$ or $\mu(A)$. The space $\Omega$ is called \emph{atomless} if it does not have any atoms.
\end{definition}

{\bf Assumption~P.} {\em The space $X$ can be endowed with the structure of a Polish space such that $\mathcal A$ is the $\sigma$-algebra consisting of its Borel sets and $\mu(\{x\}) = 0$ for every $x\in X$.}

\medskip
Let us note at once that $\sigma$-algebras appearing in Assumption~P are automatically atomless.

\begin{lemma}\label{lem:atomless}
  Let $X$ be a Polish space, let $\mathcal{A}$ be its Borel $\sigma$-algebra and let $\mu$ be a probability measure on $(X,\mathcal{A})$. Then $(X,\mathcal{A},\mu)$ is atomless if and only if $\mu(\{x\})=0$ for every $x\in X$.
  Also, in that case the diagonal $\set{(x,x)\in X^2}{x\in X}$ has measure~0 w.r.t.\ the product measure $\mu^2$.
\end{lemma}
We defer the proof of this lemma to Appendix~\ref{sec:measure}.

We can now define a version of all our concepts by replacing $[0,1]$ with $\Omega$. Some care, however, must be taken with respect to the type of measurability
required.

\begin{definition}
  Let $\Omega=(X,\mathcal{A},\mu)$ be a probability space satisfying Assumption~P. For a finite set $V$, we define $\mathcal{E}_V(\Omega) \df X^{r(V)}$ and we let $\mathcal{B}_V(\Omega)$ be
  the product $\sigma$-algebra of $\lvert r(V)\rvert$ copies of $\mathcal{A}$. Let also $\mathcal{L}_V(\Omega)$ be
  the completion of $\mathcal{B}_V(\Omega)$ with respect to the product $\mu^{r(V)}$ of $\lvert r(V)\rvert$
  copies of the measure $\mu$. Note for the record that the space $(\mathcal E_V(\Omega), \mathcal B_V(\Omega), \mu^{r(V)})$ also satisfies Assumption~P.  We will sometimes abuse the notation  denoting by $\mu^{r(V)}$ the completion of this measure as well.

  We use the same shorthand conventions when $V=[k]$.
  For a predicate symbol $P$ of arity $k$, a \emph{$P$-on over $\Omega$} is a set in $\mathcal{L}_k(\Omega)$.
  Define $\mathcal{D}_V(\Omega)$, a (weak or strong) $T$-on over $\Omega$ and related
  concepts such as $\Tinj(M,\mathcal{N})$, $\Tind(M,\mathcal{N})$, $T(F,\mathcal{N})$ etc.\ by replacing
  $[0,1]$ with $\Omega$ in Definitions~\ref{def:EV}, \ref{def:peons}, \ref{def:theons}, \ref{def:truth},
  \ref{def:hoa}, and~\ref{def:hoa2}, saying that a theon $\mathcal{N}$ over $\Omega$ is \emph{Borel} if
  $\mathcal{N}_P$ is a set in $\mathcal{B}_{k(P)}(\Omega)$ (rather than just in $\mathcal{L}_{k(P)}(\Omega)$)
  for every $P\in\mathcal{L}$. Measurability of functions $f\function{\mathcal{E}_V(\Omega)}{\Omega}$
  and $h\function{\mathcal{E}_V(\Omega)\times\mathcal{E}_V(\Omega)}{\Omega}$ in Definitions~\ref{def:hoa} and~\ref{def:hoa2} is taken with respect
  to $\mathcal{A}$ in the codomain and $\mathcal{L}_V(\Omega)$ and $\mathcal{L}_V(\Omega\times\Omega)$ (via the natural identification of $\mathcal{E}_V(\Omega)\times\mathcal{E}_V(\Omega)$ with $\mathcal{E}_V(\Omega\times\Omega)$) in the
  domain respectively. Note that for $\Omega_1 = ([0,1],\mathcal{B}_1,\lambda)$, where $\mathcal{B}_1$ is the $\sigma$-algebra of Borel sets, we recover all our previous notions.
\end{definition}

We claim that all the previous results continue to hold for an arbitrary probability space $\Omega=(X,\mathcal{A},\mu)$ satisfying Assumption~P. One way to verify this is by a direct inspection of proofs. Alternately, we can do it in a more intelligent way by invoking relatively deep results from measure theory.

\begin{definition}
  Let $\Omega=(X,\mathcal{A},\mu)$ and $\Omega'=(X',\mathcal{A}',\mu')$ be measure spaces.

  A \emph{measure-isomorphism} between $\Omega$ and $\Omega'$ is a bijection $F\function{X}{X'}$ such that both $F$ and $F^{-1}$ are measurable and measure preserving. Two spaces are said to be \emph{measure-isomorphic} if there exists a measure-isomorphism between them.

  The spaces $\Omega$ and $\Omega'$ are said to be \emph{measure-isomorphic modulo~0} if there exist $A\in\mathcal{A}$ and $A'\in\mathcal{A}'$ such that $\mu(X\setminus A)=\mu'(X'\setminus A')=0$ and the spaces $(A,\mathcal{A}\vert_A,\mu\vert_{A})$ and $(A',\mathcal{A}'\vert_{A'},\mu'\vert_{A'})$ are measure-isomorphic.

  We will denote by $\mathcal B_t$ the $\sigma$-algebra on $[0,1]^t$ that consists of all Borel sets, by $\mathcal L_t$ the $\sigma$-algebra consisting of Lebesgue measurable sets, and by $\lambda^t$ the Lebesgue measure itself. Let also $\Omega_t\df ([0,1], \mathcal B_t, \lambda^t)$.  Finally, we denote by $\pi_i\function{[0,1]^t}{[0,1]}$ the projection on the $i$th coordinate.
\end{definition}

\begin{theorem}[\protect{\cite[Theorem~9.2.2]{Bog}}] \label{thm:bogachev}
Every probability space $\Omega=(X,\mathcal{A},\mu)$ satisfying Assumption~P
is measure-isomorphic modulo~0 to $\Omega_1$.
\end{theorem}

Measure-isomorphism modulo~0 is clearly sufficient for transferring results about weak theons, i.e., those concerning only
densities of models. For results involving \emph{points} in $\mathcal{E}_V(\Omega)$ such as the ones involving
strong theons or measurable functions, measure-isomorphism modulo~0 is not enough. Note, however,
that since $(\mathcal{E}_V(\Omega),\mathcal{B}_V(\Omega),\mu^{r(V)})$ also satisfies Assumption~P and since we
define peons over $\Omega$ to be measurable with respect to the
\emph{completion} $(\mathcal{E}_V(\Omega),\mathcal{L}_V(\Omega),\mu^{r(V)})$ of this space, we can use the
following neat characterization.

\begin{theorem}\label{thm:measiso}
A probability space can be represented as the completion of a space satisfying Assumption~P if and only if it is measure-isomorphic to $([0,1], \mathcal L_1,\lambda^1)$.
\end{theorem}
Since we have not been able to find this statement in the measure-theoretic literature, we defer its simple (that is, modulo Theorem~\ref{thm:bogachev}) proof to Appendix~\ref{sec:measure} as well.

As a corollary of the above, the spaces $(\mathcal{E}_k,\mathcal{L}_k,\lambda^{r(k)})$
and $(\mathcal{E}_k(\Omega),\mathcal{L}_k(\Omega),\mu^{r(k)})$ are measure-isomorphic, where $\mathcal{L}_k$
is the $\sigma$-algebra of Lebesgue measurable sets of $\mathcal{E}_k$. Using this measure-isomorphism,
results such as Theorems~\ref{thm:ierl} and~\ref{thm:characterization} continue to hold for spaces
satisfying Assumption~P (see also Proposition~\ref{prop:uniquenessother} below).

For the uniqueness results, Theorems~\ref{thm:TheonUniqueness} and~\ref{thm:TheonUniquenessSecond}, we can mix
several different spaces (it will turn out handy for comparing ``theonic'' definitions of various objects with
original ones). Notably, for spaces $\Omega=(X,\mathcal{A},\mu)$ and $\Omega'=(X',\mathcal{A}',\mu')$
satisfying Assumption~P, we can extend Definition~\ref{def:hoa} for
functions $f\function{\mathcal{E}_V(\Omega)}{\Omega'}$
and $h\function{\mathcal{E}_V(\Omega)\times\mathcal{E}_V(\Omega)}{\Omega'}$, where measurability is taken with
respect to $\mathcal{A}'$ in the codomain and $\mathcal{L}_V(\Omega)$ and $\mathcal{L}_V(\Omega\times\Omega)$
in the domain respectively (that is, $f^{-1}(A)\in\mathcal L_V(\Omega)$ and $h^{-1}(A)\in\mathcal L_V(\Omega\times\Omega)$ for any $A\in\mathcal A'$). This implies that $\widehat{f}_d\function{\mathcal{E}_d(\Omega)}{\mathcal{E}_d(\Omega')}$ and $\widehat{h}_d\function{\mathcal{E}_d(\Omega)\times\mathcal{E}_d(\Omega)}{\mathcal{E}_d(\Omega')}$ are measurable with respect to $\mathcal{B}_d(\Omega')$ in the codomain and $\mathcal{L}_d(\Omega)$ and $\mathcal{L}_d(\Omega\times\Omega)$ in the domain respectively.

We can then combine and generalize Theorems~\ref{thm:TheonUniqueness} and~\ref{thm:TheonUniquenessSecond} into
the following form that we will need below. Recall that $\Omega_1= ([0,1], \mathcal B_1, \lambda^1)$.

\begin{proposition}\label{prop:uniquenessother}
  Let $T$ be a canonical theory in a language $\mathcal{L}$, let $k=\max\{k(P) \mid P\in\mathcal{L}\}$, and
  let $\mathcal{N}$ and $\mathcal{N}'$ be two $T$-ons over $\Omega=(X,\mathcal{A},\mu)$ and $\Omega'=(X',\mathcal{A}',\mu')$, respectively, where $\Omega$ and $\Omega'$ satisfy Assumption~P. Then the following are equivalent.
  \begin{enumerate}[label={\arabic*.}, ref={\arabic*)}]
  \item We have $\phi_{\mathcal{N}}=\phi_{\mathcal{N}'}$.
    \label{it:phiequal}
  \item There exist families $f = (f_1,\ldots,f_k)$ and $g = (g_1,\ldots,g_k)$ of symmetric functions measure preserving on h.o.a., $f_d\function{\mathcal{E}_d(\Omega_1)}{\Omega}$ and $g_d\function{\mathcal{E}_d(\Omega_1)}{\Omega'}$ and a weak $T$-on $\mathcal N''$ over $\Omega_1$
      with the property
    \begin{align*}
      x\in\mathcal{N}_P'' & \equiv
      \widehat{f}_{k(P)}(x)\in\mathcal{N}_P \equiv \widehat{g}_{k(P)}(x)\in\mathcal{N}'_P,
    \end{align*}
    for every $P\in\mathcal L$ and almost every $x\in\mathcal{E}_{k(P)}(\Omega_1)$.
    \label{it:uofg}
  \item There exists a family $h = (h_1,\ldots,h_k)$ of symmetric functions measure preserving on h.o.a., $h_d\function{\mathcal{E}_d(\Omega')\times\mathcal{E}_d(\Omega')}{\Omega}$ such that
    \begin{align*}
      \widehat{h}_{k(P)}(x,\widehat{x})\in\mathcal{N}_P & \equiv x\in\mathcal{N}'_P,
    \end{align*}
    for every predicate symbol $P\in\mathcal{L}$ and for almost every $(x,\widehat{x})\in\mathcal{E}_{k(P)}(\Omega')\times\mathcal{E}_{k(P)}(\Omega')$.
    \label{it:uoh}
  \end{enumerate}
\end{proposition}

Since transferring this result is a bit more sensitive with respect to measurability, we explicitly provide the proof below.

\begin{proof}
  The implications~\ref{it:uofg}~$\Longrightarrow$~\ref{it:phiequal}
  and~\ref{it:uoh}~$\Longrightarrow$~\ref{it:phiequal} are immediate.

  Suppose then that $\phi_{\mathcal{N}}=\phi_{\mathcal{N}'}$.

  By Theorem~\ref{thm:bogachev}, there exist $B\in\mathcal{B}_1$ and $A\in\mathcal{A}$ and a measure
  isomorphism $F\function{(B,\mathcal{B}_1\vert_B,\lambda^1\vert_B)}{(A,\mathcal{A}\vert_A,\mu\vert_A)}$. Let
  $x_0\in A$ be an arbitrary point and extend $F$ to a measure preserving function
  $F\function{\Omega_1}{\Omega}$ by setting $F(z) = x_0$ for every $z\in[0,1]\setminus B$ (note that we may lose
  bijectivity in the process).

  For every $d\in[k]$, let $F_d\function{\mathcal{E}_d(\Omega_1)}{\mathcal{E}_d(\Omega)}$ be the function
  obtained by applying $F$ to each of the coordinates, that is we set
  \begin{align*}
    F_d(x)_A & = F(x_A)
  \end{align*}
  for every $A\in r(d)$. Note that $F_d$ is measurable and measure preserving when we equip the domain
  with $(\mathcal{B}_d(\Omega_1),\lambda^{r(d)})$ and the codomain with $(\mathcal{B}_d(\Omega),\mu^{r(d)})$. We define measure preserving $F'\function{\Omega_1}{\Omega'}$ and $F_d'\function{\mathcal{E}_d(\Omega_1)}{\mathcal{E}_d(\Omega')}$ by the
same process.

  Consider then the $T$-ons over $\Omega_1$ defined by
  \begin{align*}
    \widehat{\mathcal{N}}_P & \df F_{k(P)}^{-1}(\mathcal{N}_P); &
    \widehat{\mathcal{N}}'_P & \df (F_{k(P)}')^{-1}(\mathcal{N}'_P);
  \end{align*}
  and note that since $F_d$ and $F_d'$ are measure preserving, we
  have $\phi_{\widehat{\mathcal{N}}}=\phi_{\mathcal{N}}=\phi_{\mathcal{N}'}=\phi_{\widehat{\mathcal{N}}'}$.

  By Theorem~\ref{thm:TheonUniqueness}, there exist families $f = (f_1,\ldots,f_k)$ and $g = (g_1,\ldots,g_k)$
  of symmetric functions measure preserving on h.o.a., $f_d\function{\mathcal{E}_d(\Omega_1)}{\Omega_1}$ and
  $g_d\function{\mathcal{E}_d(\Omega_1)}{\Omega_1}$ such that
  there exists a $T$-on $\mathcal N''$ over $\Omega_1$ with the property
  \begin{align*}
    x\in\mathcal{N}_P'' & \equiv \widehat{f}_{k(P)}(x)\in\widehat{\mathcal{N}}_P \equiv
    \widehat{g}_{k(P)}(x)\in\widehat{\mathcal{N}}'_P,
  \end{align*}
  for almost every $x\in\mathcal{E}_{k(P)}(\Omega_1)$.

  Then item~\ref{it:uofg} holds for the families $f' = (f'_1,\ldots,f'_k)$ and $g' = (g'_1,\ldots,g'_k)$ given
  by
  \begin{align*}
    f'_d & \df F\comp f_d; & g'_d & \df F'\comp g_d
  \end{align*}
  since for these we get
  \begin{align*}
    \widehat{f}'_d & = F_d\comp\widehat{f}_d; & \widehat{g}'_d & = F_d'\comp\widehat{g}_d.
  \end{align*}
  Let us stress once more that measurability of the $f'_d$ and $g'_d$ only follows because we used Borel $\sigma$-algebras
  in the domain of $F,F'$ and in the codomain of $f_d$: the completion $\mathcal L_1$ of $\mathcal B_1$ may totally misbehave in
  this new context.

  \medskip

  Let us now prove item~\ref{it:uoh}. By Theorem~\ref{thm:measiso} there exists a measure
  isomorphism $G\function{(\mathcal{E}_1(\Omega),\mathcal{L}_1(\Omega),\mu)}{([0,1],\mathcal{L}_1,\lambda^1)}$.
  Let $G_d\function{\mathcal{E}_d(\Omega)}{\mathcal{E}_d(\Omega_1)}$ be the function
  obtained by applying $G$ to each of the coordinates. Note that $G_d$ is a measure-isomorphism if we equip
  its domain and codomain with the product $\sigma$-algebras of $\lvert r(d)\rvert$ copies
  of $\mathcal{L}_1(\Omega)$ and $\mathcal{L}_1$ respectively. By completing both measure spaces, we get
  that $G_d$ is a measure-isomorphism between $(\mathcal{E}_d(\Omega),\mathcal{L}_d(\Omega),\mu^{r(d)})$
  and $(\mathcal{E}_d(\Omega_1),\mathcal{L}_d(\Omega_1),\lambda^{r(d)})$. Define measure preserving $F'\function{\Omega_1}{\Omega'}$ and $F_d'\function{\mathcal{E}_d(\Omega_1)}{\mathcal{E}_d(\Omega')}$ as in the previous item and let
  \begin{align*}
    \widehat{\mathcal{N}}_P & \df G_{k(P)}(\mathcal{N}_P); &
    \widehat{\mathcal{N}}'_P & \df (F_{k(P)}')^{-1}(\mathcal{N}'_P).
  \end{align*}

  Finally, let $H\function{\mathcal{E}_1(\Omega)\times\mathcal{E}_1(\Omega)}{\Omega_1\times\Omega_1}$ be the
  function obtained by applying $G$ to each coordinate and let
  $H_d\function{\mathcal{E}_d(\Omega)\times\mathcal{E}_d(\Omega)}{\mathcal{E}_d(\Omega_1)\times\mathcal{E}_d(\Omega_1)}$
  be the function obtained by applying $G_d$ to each coordinate.

  By Theorem~\ref{thm:TheonUniquenessSecond}, there exists a family $h = (h_1,\ldots, h_k)$ of symmetric
  functions measure preserving on h.o.a., $h_d\function{\mathcal{E}_d(\Omega_1)\times\mathcal{E}_d(\Omega_1)}{\Omega_1}$ such that
  \begin{align*}
     x\in\widehat{\mathcal{N}}_P\equiv \widehat{h}_{k(P)}(x,\widehat x)\in\widehat{\mathcal{N}}'_P,
  \end{align*}
  for every predicate symbol $P\in\mathcal{L}$ and for almost every $(x,\widehat
  x)\in\mathcal{E}_{k(P)}\times\mathcal{E}_{k(P)}$, which implies
  \begin{align*}
    y\in\mathcal{N}_P\equiv (F'_{k(P)}\comp\widehat{h}_{k(P)}\comp H_{k(P)})(y,\widehat y)\in\mathcal{N}'_P,,
  \end{align*}
  for every predicate symbol $P\in\mathcal{L}$ and for almost every $(y,\widehat
  y)\in\mathcal{E}_{k(P)}(\Omega)\times\mathcal{E}_{k(P)}(\Omega)$. Then item~\ref{it:uoh} holds for the family $h' = (h'_1,\ldots,h'_k)$ given by
  \begin{align*}
    h'_d & \df F'\comp h_d\comp H
  \end{align*}
  since for these we get
  \begin{align*}
    \widehat{h}'_d & = F'_d\comp\widehat{h}_d\comp H_d.
  \end{align*}
  Again, measurability of the $h'_d$ only follows because we used the Borel $\sigma$-algebras for $F'$ and the completions for $G$.
\end{proof}

The requirement in Assumption~P that $\mu$ is atomless is very crucial. For example, if $X$ is finite then there are only finitely many different theons, which is certainly an undesirable feature of the theory.

\subsection{Permutons} \label{sec:permutons}
  Recall (see Example~\ref{ex:mix}) that in our formalism the theory of permutations is defined as
  $T_{\operatorname{Perm}}= T_{\operatorname{LinOrder}}\cup
  T_{\operatorname{LinOrder}}$. On the other hand, the following definition
  was made before.

  \begin{definition}[\cite{HKM*}]\label{def:permuton}
  A \emph{permuton} is a probability measure $\mu$ on $([0,1]^2,\mathcal{B}_2)$ (recall that $\mathcal{B}_2$ denotes the $\sigma$-algebra of Borel sets) such that both marginals of $\mu$ are equal to the Lebesgue measure $\lambda^1$.
  \end{definition}

  Note that the last condition simply says that each of the projections $\pi_i\function{([0,1]^2, \mathcal B_2, \mu)}{\Omega_1}$ is measure
preserving.

  For a fixed permutation $\sigma\in S_m$, we view it as an element of $\mathcal{M}_m[T_{\operatorname{Perm}}]$ and define $p(\sigma,\mu)$ by the following probabilistic experiment. We let $\rn{X_1},\ldots,\rn{X_m}$ be i.i.d.\ random variables picked according to the measure $\mu$ and define the random structure $\rn{M}$ in $\mathcal{K}_m[\{\prec_1,\prec_2\}]$ by letting
  \begin{align}\label{eq:randompermutation}
    R_{\prec_i,\rn{M}} & \df \{(a,b)\in[m]^2 \mid \pi_i(\rn{X_a}) < \pi_i(\rn{X_b})\},
  \end{align}
  for all $i\in[2]$.
Note that the marginal condition on $\mu$ guarantees that the diagonal has measure~0 and hence $\rn{M}$ is almost surely a model of $T_{\operatorname{Perm}}$. We then define $p(\sigma,\mu) \df \prob{\rn{M}\cong\sigma}$ and
define the functional $\phi_\mu \df p(\place,\mu)$. It is easy to see by a direct computation that $\phi_\mu\in \HomT{T_{\operatorname{Perm}}}$, and~\cite[Theorem~1.6]{HKM*} proved that every convergent sequence converges to
a permuton. Along with Theorem~\ref{thm:allcrypto}, this implies that permutons
  are cryptomorphic to $T_{\operatorname{Perm}}$-ons and hence to all other objects listed in its statement.
  But this detour via flag algebras is definitely unnatural, and
our purpose in this section is to give a {\em direct} translation between permutons and $T_{\operatorname{Perm}}$-ons bypassing any density counting. As an application, we will present an
alternate proof of the uniqueness of permutons~\cite[Theorem~1.7 and discussion thereafter]{HKM*}.

\medskip
In one direction, such a translation is more or less straightforward (modulo the background material we developed in Section~\ref{sec:measure_background}). Namely, the marginal conditions imply that every permuton $\mu$ satisfies $\mu(\{x\})=0,\ x\in [0,1]^2$ and hence Assumption~P. Consider the (strong Borel $\mathcal E_2^*$-measurable)
$T_{\operatorname{Perm}}$-on $\mathcal{N}(\mu)$ over the space $([0,1]^2,\mathcal{B}_2,\mu)$ given by\footnote{The second term here resolves conflicts along horizontal and vertical lines and is inserted to make sure that the theon $\mathcal N(\mu)$ is strong.}
\begin{equation}\label{eq:standardTPerm-on}
  \begin{aligned}
    \mathcal{N}(\mu)_{\prec_i}
    & \df
    \Bigl\{(x_{\{1\}},x_{\{2\}},x_{\{1,2\}})\in\mathcal{E}_2(([0,1]^2,\mathcal{B}_2,\mu)) \mid
    \pi_i(x_{\{1\}}) < \pi_i(x_{\{2\}})
    \\
    & \qquad \lor
    \bigl(\pi_i(x_{\{1\}}) = \pi_i(x_{\{2\}})\land\pi_{3-i}(x_{\{1\}}) < \pi_{3-i}(x_{\{2\}})\bigr);
  \end{aligned}
\end{equation}
we will sometimes call it the {\em standard $T_{\operatorname{Perm}}$-on associated with $\mu$}. It is straightforward to see that if $\rn{M}$ is the random permutation defined in~\eqref{eq:randompermutation}, then $p(\sigma,\mathcal{N}(\mu)) = \prob{\rn{M}\cong\sigma} =p(\sigma,\mu)$.

\medskip
In the opposite direction, we need to show how to obtain the measure $\mu$ from a (weak) $T_{\operatorname{Perm}}$-on $\mathcal N$. Let us briefly remark first that with material from flag algebras slightly more advanced
than we reviewed in Section~\ref{sec:flag}, this is also completely straightforward. Let us sketch the argument for the readers familiar with those parts of the theory (this argument will not be used in the sequel).

\medskip
First, the (easy!) part of Theorem~\ref{thm:allcrypto} gives us a homomorphism $\phi_{\mathcal N}\in\HomT{T_{\operatorname{Perm}}}$.
Now, define from it a random distribution $\rn{\phi^1}$ over $\operatorname{Hom}^+(\mathcal A^1[T_{\operatorname{Perm}}],\mathbb R)$ as in~\cite[Definition~10]{flag} and consider two elements $L^1, L^2\in \mathcal A^1_2$, where $L^i$ is the sum of the (two)
flags in $\mathcal F^1_2$ in which $v\prec_i 1$, $1$ being the labeled vertex and $v$ being the unlabeled one. Then it is easy to check that the
pushforward distribution $(\rn{\phi^1}(L_1), \rn{\phi^1}(L_2))$ defines the required permuton $\mu$.
It is worth noting that this argument already gives us an alternative proof of the existence of permutons.

As a by-side remark, the uniqueness of permutons is also quite straightforward in
this language. Indeed, let $\mu$ be a permuton, and let $\mu^\ast\sim (\rn{\phi_\mu^1}(L_1), \rn{\phi_\mu^1}(L_2))$ be the
permuton (uniquely) retrieved from $\phi_\mu$ by the process described in the previous paragraph. We have to show that $\mu^\ast=\mu$.
But this is immediate from the observation that the measure $\rn{\phi_\mu^1}$ with the (uniquely!) defining
property~\cite[Definition~10]{flag} can be geometrically constructed from {\em any}
$T_{\operatorname{Perm}}$-on $\mathcal N$ with $\phi_{\mathcal N} =\phi_\mu$. In
particular, it can be constructed from the standard $T_{\operatorname{Perm}}$-on $\mathcal N(\mu)$.
Now the fact that $(\rn{\phi_\mu^1}(L_1), \rn{\phi_\mu^1}(L_2))$ has the same distribution
as $\mu$ is straightforward.

\medskip
The above argument, while formally quite simple, entirely obscures the geometric
nature of both permutons and $T_{\operatorname{Perm}}$-ons and replaces it with
formal algebraic and measure-theoretic manipulations. While this is arguably the
whole point of the theory of flag algebras, it is certainly not the main thrust
of the current paper. Fortunately, in this particular case the geometric translation
is not very difficult to describe (and prove) explicitly.

\smallskip
So, we start with a weak $T_{\operatorname{Perm}}$-on $\mathcal N$ over a space
$\Omega = (X,\mathcal A,\mu)$ satisfying Assumption~P. For $i=1,2$, define the
function $s^i_{\mathcal N}\function X{[0,1]}$ as the measure of the corresponding
section:
\begin{align}\label{eq:permutonfunction}
  s^i_{\mathcal N}(y) & \df
  \mu^2(\{(x,z)\in X^2 \mid (x,y,z)\in \mathcal N_{\prec_i}\})
\end{align}
(by Fubini's theorem, $s_i$ is defined a.e.\ and is measurable), and let $s_{\mathcal N}\function{X}{[0,1]^2}$
be their pointwise Cartesian product: $s_{\mathcal N}(x)
\df (s^1_{\mathcal N}(x), s^2_{\mathcal N}(x))$. We claim that the pushforward measure
$\nu_{\mathcal N}\df (s_{\mathcal N})_\ast\mu$ is the desired permuton.

The most subtle part is to prove that the functions $s^i_{\mathcal N}$ are
measure preserving. Towards that end, fix $i\in \{1,2\}$ and $a\in [0,1]$, and let us abbreviate
$s\df s^i_{\mathcal N}$. We have to prove that $\mu(s^{-1}([0,a]))=a$. This clearly
follows from
\begin{align*}
  \mu(s^{-1}([0,a])) & \leq a, &
  \mu(s^{-1}([a,1])) & \leq 1-a,
\end{align*}
and by symmetry it suffices to prove the first bound.

Denote $Y\df s^{-1}([0,a])$, and assume, for the sake of contradiction, that
$\mu(Y)>a$. By~\cite[Theorem~3.15]{Oxt}, there exists a $G_\delta$-set $Z\supseteq Y$ with
$\mu(Z)=\mu(Y)$, and by~\cite[Theorem~6.1.12]{Bog}, the set $Z$ with the induced topology is a Polish space. Hence the
induced probability space $\widehat\Omega\df (Z,\mathcal A\vert_Z, \widehat\mu)$,
where
$$
\widehat\mu(A) \df\frac{\mu(A)}{\mu(Z)},
$$
satisfies Assumption~P. Endow this space with the (induced) structure of a
$T_{\operatorname{LinOrder}}$-on $\widehat{\mathcal N}$ by letting
$$
\widehat{\mathcal N} \df \set{(x,y,z)\in Z^{r(2)}}{(x,y,F(z))\in
 \mathcal N_{\prec_i}},
$$
where $F$ is an arbitrary fixed measure-isomorphism modulo~0 between $\widehat\Omega$ and $\Omega$.

Let $n>0$ and let $S_n$ be the model of $T_{\operatorname{Order}}$ with
$V(S_n) = [n]$ in which $2\prec 1,\ 3\prec 1,\ldots, n\prec 1$ and the
elements $2,3,\ldots,n$ are mutually incomparable. Since
$\mathcal M_n[T_{\operatorname{LinOrder}}]$ consists of a single element,
say, $M_n$, \eqref{eq:chain_inj} allows us to calculate
$\tinj(S_n,M_L)$, for any $L\geq n$, as follows:
$$
\tinj(S_n,M_L) =\tinj(S_n, M_n) =1/n.
$$
Taking the limit, we conclude that
\begin{equation} \label{eq:order}
\tinj(S_n,\widehat{\mathcal N})=1/n.
\end{equation}

On the other hand, this quantity can be calculated geometrically as
$$
\tinj(S_n,\widehat{\mathcal N}) = \prob{\bigwedge_{i=2}^n (\rn{x_i}, \rn{y}, \rn{z_i})\in
\widehat{\mathcal N}} = \indexpect{\rn{y}}{\indprob{\rn{x}, \rn{z}}{(\rn{x}, \rn{y}, \rn{z})\in
\widehat{\mathcal N}}^{n-1}},
$$
where $\rn y, \rn{x_2},\ldots,\rn{x_n},\rn{z_{2}},\ldots,\rn{z_{n}},\rn{x},\rn{z}$ are sampled from $Z$
i.i.d.\ with respect to the measure $\widehat\mu$.

Finally, for almost all $y\in Z$ we have $s(y)\leq a$ and hence
\begin{align*}
  \prob{(\rn x, y, \rn{z}) \in \widehat{\mathcal N}}
  & =
  \condprob{(\rn x,y,\rn z)\in \mathcal N_{\prec_i}}{\rn x\in Z}
  \\
  & \leq
  \frac{\prob{(\rn x,y,\rn z)\in \mathcal N_{\prec_i}}}{\mu(Z)}
  =
  \frac{s(y)}{\mu(Y)}
  \leq
  \frac{a}{\mu(Y)}.
\end{align*}
Putting things together, we get $\tinj(S_n,\widehat{\mathcal N})\leq (a/\mu(Y))^{n-1}$, and since $\mu(Y) > a$ this
contradicts~\eqref{eq:order} as long as $n$ is large enough.

\smallskip

Now that we know that $\nu_{\mathcal{N}}$ is a permuton, it follows that the space $\Omega'=([0,1]^2,\mathcal{B}_2,\nu_{\mathcal{N}})$ satisfies Assumption~P. By the definition of $\nu_{\mathcal{N}}$, the functions
\begin{align}\label{eq:permutonshoa}
  \begin{functiondef}
    f_1\colon & \mathcal{E}_1(\Omega) & \longrightarrow & \Omega'\\
    & x & \longmapsto & s_{\mathcal{N}}(x)
  \end{functiondef}
  & &
  \begin{functiondef}
    f_2\colon & \mathcal{E}_2(\Omega) & \longrightarrow & \Omega'\\
    & x & \longmapsto & s_{\mathcal{N}}(x_{\{1,2\}})
  \end{functiondef}
\end{align}
are symmetric and measure preserving on h.o.a., and since
\begin{align*}
  x\in\mathcal{N}_{\prec_i} & \equiv \widehat{f}_2(x)\in\mathcal{N}(\nu_{\mathcal{N}})_{\prec_i}
\end{align*}
for almost every $x\in\mathcal{E}_2(\Omega)$, by the (easy!) part of Proposition~\ref{prop:uniquenessother} and the first direction of the cryptomorphism, we get $\phi_{\mathcal{N}} = \phi_{\mathcal{N}(\nu_{\mathcal{N}})} = \phi_{\nu_{\mathcal{N}}}$.

\medskip

We end this subsection with a geometric proof of permuton uniqueness.

\begin{theorem}[{{\cite[Theorem~1.7 and discussion thereafter]{HKM*}}}]
  Let $\mu$ and $\nu$ be permutons. Then $\phi_\mu = \phi_\nu$ if and only if $\mu = \nu$ (as measures).
\end{theorem}

\begin{proof}
  The backward implication is obvious, so suppose $\phi_\mu = \phi_\nu$.

  Let $\Omega_\mu = ([0,1]^2,\mathcal{B}_2,\mu)$ and $\Omega_\nu = ([0,1]^2,\mathcal{B}_2,\nu)$.

  Using standard $T_{\operatorname{Perm}}$-ons $\mathcal N(\mu), \mathcal N(\nu)$ associated to $\mu$ and $\nu$ and by
  Proposition~\ref{prop:uniquenessother} (see also Figure~\ref{fig:permuniq} below), we know that there exists a symmetric measure preserving function\footnote{Since $\mathcal{N}(\mu)$ and $\mathcal{N}(\nu)$ are $\mathcal E_2^\ast$-measurable, we do not need $h_2$.}
  $h_1\function{\Omega_\mu\times \Omega_\mu}{\Omega_\nu}$ such that
  \begin{align}\label{eq:pih}
    \pi_i(h_1(y_{\{1\}},\widehat{y}_{\{1\}})) < \pi_i(h_1(y_{\{2\}},\widehat{y}_{\{2\}}))
    & \equiv
    \pi_i(y_{\{1\}}) < \pi_i(y_{\{2\}}),
  \end{align}
  for almost every $((y_A)_{A\in r(2)},(\widehat{y}_A)_{A\in
    r(2)})\in\mathcal{E}_2(\Omega_\mu)\times\mathcal{E}_2(\Omega_\mu)$ and every $i\in[2]$. By Fubini's Theorem, it follows that~\eqref{eq:pih} holds also for almost every $(y_{\{1\}},y_{\{2\}},\widehat{y}_{\{1\}},\widehat{y}_{\{2\}})\in\Omega_\mu^4$.

  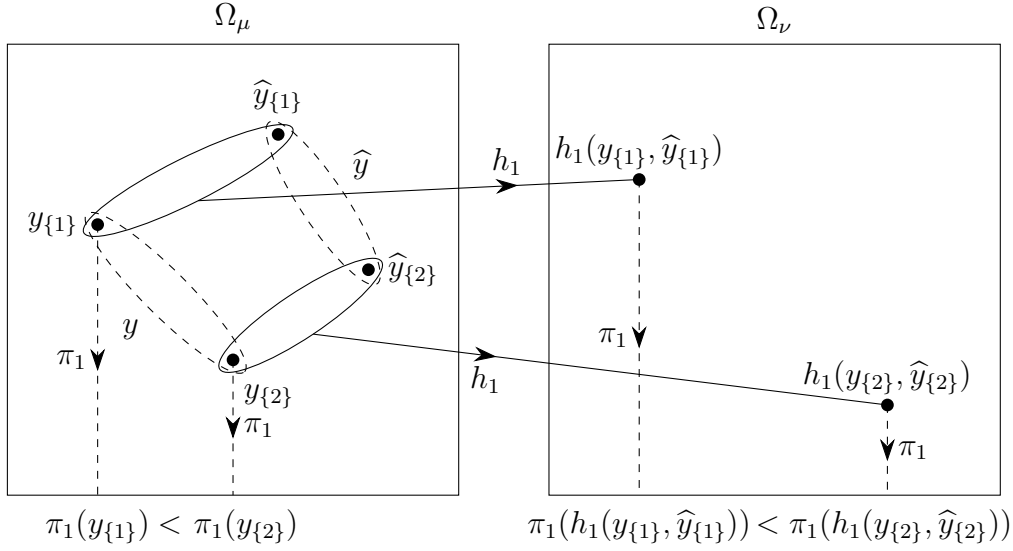
\begin{figure}[ht]
    \begin{center}
      \begingroup

\makeatletter
\def\pgfsetangleandlength#1#2#3#4{%
  \pgfmathanglebetweenpoints{\pgfpointanchor{#3}{center}}{\pgfpointanchor{#4}{center}}%
  \xdef#1{\pgfmathresult}%
  \pgfmathparse{veclen(\pgf@x,\pgf@y)}%
  \xdef#2{\pgfmathresult}%
}
\makeatother

\def\ellipseextra{1} 
\def\ellipsey{0.05 cm}
\def\al{3mm}
\def\hal{1.5mm}
\def\aw{2mm}

\begin{tikzpicture}[scale=6]
  \begin{scope}[xshift=-0.6cm]
    \coordinate (Y1) at    (0.2, 0.6);
    \coordinate (Y2) at    (0.5, 0.3);
    \coordinate (Yhat1) at (0.6, 0.8);
    \coordinate (Yhat2) at (0.8, 0.5);

    \coordinate (pY1) at ($(0,0)!(Y1)!(1,0)$);
    \coordinate (pY2) at ($(0,0)!(Y2)!(1,0)$);

    \coordinate (mY1) at ($1/2*(Y1) + 1/2*(pY1)$); 
    \coordinate (mY2) at ($1/2*(Y2) + 1/2*(pY2)$); 

    \draw (0,0) -- (1,0) -- (1,1) -- (0,1) -- cycle;

    \coordinate (Lmu) at (0.5,1);
  \end{scope}

  \begin{scope}[xshift=0.6cm]
    \coordinate (hY1) at (0.2, 0.7);
    \coordinate (hY2) at (0.75, 0.2);

    \coordinate (phY1) at ($(0,0)!(hY1)!(1,0)$);
    \coordinate (phY2) at ($(0,0)!(hY2)!(1,0)$);

    \coordinate (mhY1) at ($1/2*(hY1) + 1/2*(phY1)$);
    \coordinate (mhY2) at ($1/2*(hY2) + 1/2*(phY2)$);

    \draw (0,0) -- (1,0) -- (1,1) -- (0,1) -- cycle;

    \coordinate (Lnu) at (0.5,1);
  \end{scope}

  \node[above] at (Lmu) {$\Omega_\mu$};
  \node[above] at (Lnu) {$\Omega_\nu$};

  \foreach \p/\l/\s in {%
    Y1/$y_{\{1\}}\;$/left,
    Y2/$\strut y_{\{2\}}$/below right,
    Yhat1/$\vphantom{\Big(}\widehat{y}_{\{1\}}$/above,
    Yhat2/$\;\widehat{y}_{\{2\}}$/right,
    hY1/{$h_1(y_{\{1\}},\widehat{y}_{\{1\}})$}/above,
    hY2/{$h_1(y_{\{2\}},\widehat{y}_{\{2\}})$}/above%
  }{%
    \fill (\p) circle (0.4pt);
    \node[\s] at (\p) {\l};
  }

  \foreach \p/\l in {%
    pY1/$\strut\pi_1(y_{\{1\}})$,
    pY2/$\strut\;\;\;\pi_1(y_{\{2\}})$,
    phY1/${\strut\pi_1(h_1(y_{\{1\}},\widehat{y}_{\{1\}}))}$,
    phY2/${\strut\;\;\;\pi_1(h_1(y_{\{2\}},\widehat{y}_{\{2\}}))}$%
  }{%
    \node[below] at (\p) {\l};
  }

  \node[below] at ($1/2*(pY1) + 1/2*(pY2)$) {$\strut<$};
  \node[below] at ($1/2*(phY1) + 1/2*(phY2)$) {$\strut<$};

  \draw[dashed,arrows={-Stealth[width=\aw,length=\al]},shorten >={-\hal}] (Y1) -- (mY1);
  \draw[dashed] (mY1) -- (pY1);
  \draw[dashed,arrows={-Stealth[width=\aw,length=\al]},shorten >={-\hal}] (Y2) -- (mY2);
  \draw[dashed] (mY2) -- (pY2);
  \draw[dashed,arrows={-Stealth[width=\aw,length=\al]},shorten >={-\hal}] (hY1) -- (mhY1);
  \draw[dashed] (mhY1) -- (phY1);
  \draw[dashed,arrows={-Stealth[width=\aw,length=\al]},shorten >={-\hal}] (hY2) -- (mhY2);
  \draw[dashed] (mhY2) -- (phY2);

  \node[left] at (mhY1) {$\pi_1$};
  \node[right] at (mhY2) {$\pi_1$};
  \node[left] at (mY1) {$\pi_1$};
  \node[right] at (mY2) {$\pi_1$};

  \foreach \P/\Q/\s in {%
    Y1/Yhat1/,
    Y2/Yhat2/,
    Y1/Y2/dashed,
    Yhat2/Yhat1/dashed%
  }{%
    \pgfsetangleandlength{\angle}{\length}{\P}{\Q}
    \begin{scope}[rotate=\angle]
      \pgfmathsetmacro{\xdim}{\length/2 + \ellipseextra}
      \draw[\s]
      ($1/2*(\P) + 1/2*(\Q)$) ellipse (\xdim pt and \ellipsey);

      \coordinate (el\P\Q) at ($1/2*(\P) + 1/2*(\Q) + (0,-\ellipsey)$);
    \end{scope}
  }

  \node[below left] at (elY1Y2) {$y$};
  \node[above right] at (elYhat2Yhat1) {${\widehat{y}}$};

  \coordinate (maphY1) at ($3/10*(elY1Yhat1) + 7/10*(hY1)$);
  \coordinate (maphY2) at ($7/10*(elY2Yhat2) + 3/10*(hY2)$);

  \node[above] at (maphY1) {$h_1$};
  \node[below] at (maphY2) {$h_1$};

  \draw[arrows={-Stealth[width=\aw,length=\al]},shorten >={-\hal}] (elY1Yhat1) -- (maphY1);
  \draw (maphY1) -- (hY1);
  \draw[arrows={-Stealth[width=\aw,length=\al]},shorten >={-\hal}] (elY2Yhat2) -- (maphY2);
  \draw (maphY2) -- (hY2);

\end{tikzpicture}

\endgroup
      \caption{Function $h_1$ and property~\eqref{eq:pih}. The variables $\widehat{y}_{\{1\}}$
        and $\widehat{y}_{\{2\}}$ act as dummy variables, so their relative order does not matter.}
      \label{fig:permuniq}
    \end{center}
  \end{figure}

  Define $\chi_i \df \pi_i\comp h_1$; this function is measure preserving since $h_1$ and $\pi_i$ are so.
Our objective is to prove that $\chi_i(y,\widehat{y}) = \pi_i(y)$ for almost
  every $(y,\widehat{y})\in\Omega_\mu^2$.
We will show this for $\chi_1$ (the proof for $\chi_2$ is analogous); it might be instructive to compare this proof with those in Section~\ref{sec:constructive}.

  Let
  \begin{align*}
    C
    & \df
    \left\{\bigl(x,\widehat{x},y,\widehat{y}\bigr)\in\Omega_\mu^4 \;\middle\vert\;
    \chi_1(x,\widehat{x}) < \chi_1(y,\widehat{y}) \equiv \pi_1(x) < \pi_1(y)\right\}.
  \end{align*}
  As we have previously observed in~\eqref{eq:pih}, we have $\mu^4(C) = 1$.
For every $(y,\widehat{y})\in\Omega_\mu^2$, define the section
  \begin{align*}
    C(y,\widehat{y})
    & \df
    \left\{(x,\widehat{x})\in\Omega_\mu^2 \;\middle\vert\;
    (x,\widehat{x},y,\widehat{y})\in C\right\},
  \end{align*}
  and let $G$ be the set of all $(y,\widehat{y})\in\Omega_\mu^2$ such
  that $\mu^2(C(y,\widehat{y})) = 1$.
By Fubini's Theorem, it follows that $\mu^2(G) = 1$.

  Finally, define the set
  \begin{align*}
    L(y,\widehat{y})
    & \df
    \left\{(x,\widehat{x})\in\Omega_\mu^2 \;\middle\vert\;
    \chi_1(x,\widehat{x}) < \chi_1(y,\widehat{y})\right\}
    \\
    & =
    \chi_1^{-1}\Bigl(\bigl[0,\chi_1(y,\widehat{y})\bigr)\Bigr);
  \end{align*}
  note that $\mu^2(L(y,\widehat y)) = \chi_1(y,\widehat y)$ since $\chi_1$ is measure preserving.

  Tracking down our definitions, it follows that for every $(y,\widehat{y})\in G$, we have
  \begin{align*}
    \mu^2\Bigl(L\bigl(y,\widehat{y}\bigr)\symmdiff \bigl(\pi_1^{-1}([0,\pi_1(y)))\times\Omega_\mu\bigr)\Bigr)
    & =
    0.
  \end{align*}
  Hence (since $\pi_1\function{\Omega_\mu}{\Omega_1}$ is measure preserving)
    $\pi_1(y)
    =
    \mu^2(L(y,\widehat{y}))
  =
    \chi_1(y,\widehat{y})$
  whenever $(y,\widehat{y})\in G$.  Since $\mu^2(G) = 1$, it follows that $\chi_1(y,\widehat{y}) =
  \pi_1(y)$ for almost every $(y,\widehat{y})\in\Omega_\mu^2$ as desired.
Analogously, we get $\chi_2(y,\widehat{y}) = \pi_2(y)$ for almost every $(y,\widehat{y})\in\Omega_\mu^2$.

  Since $\chi_i = \pi_i\comp h_1$, it follows that $h_1(y,\widehat{y}) = y$ for almost
  every $(y,\widehat{y})\in\Omega_\mu^2$. Let then $I$ be the set of all $(y,\widehat{y})\in\Omega_\mu^2$ such
  that this holds ($\mu^2(I) = 1$) and note that for every $A\in\mathcal{B}_2$, we have
  \begin{align*}
    \nu(A)
    & =
    \mu^2(h_1^{-1}(A))
    =
    \mu^2(h_1^{-1}(A)\cap I)
    \\
    & =
    \mu^2((A\times\Omega_\mu)\cap I)
    =
    \mu^2(A\times\Omega_\mu)
    =
    \mu(A),
  \end{align*}
  hence $\mu = \nu$.
\end{proof}

\subsection{Posetons}
\label{sec:posetons}

Our next objective is to show how posetons from~\cite{Jan} can be identified with $T_{\operatorname{Order}}$-ons.

\begin{definition}[\cite{Jan}]\label{def:posetons}
  A \emph{poseton} is a measurable function $W\function{[0,1]^4}{[0,1]}$ such that for every $(x_1,y_1),(x_2,y_2),(x_3,y_3)\in[0,1]^2$, we have
  \begin{align*}
    x_1\geq x_2 & \implies W(x_1,y_1,x_2,y_2)=0;\\
    W(x_1,y_1,x_2,y_2) > 0\land W(x_2,y_2,x_3,y_3)>0 & \implies W(x_1,y_1,x_3,y_3)=1.
  \end{align*}
\end{definition}

For a poseton $W$ and a fixed poset $M\in\mathcal{M}_m[T_{\operatorname{Order}}]$, we define $p(M,W)$ by the following probabilistic experiment. We let $\rn{X_1},\rn{Y_1},\rn{X_2},\rn{Y_2},\ldots,\rn{X_m},\rn{Y_m}$ be i.i.d.\ random variables picked uniformly in $[0,1]$ and define the random structure $\rn{M}$ in $\mathcal{K}_m[\{\prec\}]$ by putting each $(a,b)\in[m]^2$ with $a\neq b$ in $R_{\prec,\rn{M}}$ independently with probability $W(\rn{X_a},\rn{Y_a},\rn{X_b},\rn{Y_b})$. Note that $\rn{M}$ is almost surely a model of $T_{\operatorname{Order}}$ (the first condition in Definition~\ref{def:posetons} enforces the axiom~\eqref{eq:asymmetry} while the second enforces transitivity). We then define $p(M,W) \df \prob{\rn{M}\cong M}$ and define the functional $\phi_W = p(\place,W)$. By a direct computation, we have $\phi_W\in\HomT{T_{\operatorname{Order}}}$.

Janson~\cite[Theorems~1.7 and~1.9]{Jan} proved that convergent sequences of finite posets are cryptomorphic to  posetons. In this section, we will show how this cryptomorphism looks in our framework for $T_{\operatorname{Order}}$-ons. To do so, we will need another theory: let $T_{\operatorname{ExtendedOrder}}$ be the theory obtained from $T_{\operatorname{LinOrder}}\cup T_{\operatorname{Order}}$ by adding the axiom
\begin{align}\label{eq:completionorder}
  \forall x\forall y (x\prec_2 y \to x\prec_1 y),
\end{align}
where $\prec_1$ and $\prec_2$ are the predicate symbols of $T_{\operatorname{LinOrder}}$
and $T_{\operatorname{Order}}$ respectively. In other words, $T_{\operatorname{ExtendedOrder}}$ is the theory of a partial order $\prec_2$ along with its extension $\prec_1$ to a linear order. The cryptomorphism between posetons and (strong) $T_{\operatorname{Order}}$-ons will be established via the following constructions:
\begin{align*}
  \text{posetons} & \implies
  \text{strong } T_{\operatorname{Order}}\text{-ons} \implies
  \text{strong Borel } T_{\operatorname{Order}}\text{-ons}
  \\
   & \implies
  \text{strong } T_{\operatorname{ExtendedOrder}}\text{-ons} \implies
  \text{posetons}.
\end{align*}

\medskip

The first arrow is simple: given a poseton $W$, it is easy to see that the Euclidean structure over $\Omega_2$ ($=([0,1]^2,\mathcal{B}_2,\lambda^2)$) defined by
\begin{equation}\label{eq:TOrderon}
  \begin{aligned}
    \mathcal{N}(W)_\prec \df \{z\in\mathcal{E}_2(\Omega_2) & \mid \pi_1(z_{\{1,2\}}) < W(z_{\{1\}}, z_{\{2\}})
    \\
    & \lor W(z_{\{1\}}, z_{\{2\}}) =
    \pi_1(z_{\{1,2\}}) =1\},
  \end{aligned}
\end{equation}
where $W$ is viewed as a function $W\function{[0,1]^2\times [0,1]^2}{[0,1]}$, is a strong $T_{\operatorname{Order}}$-on and satisfies $\phi_{\mathcal{N}_\prec(W)} = \phi_W$. For future reference we note that we can retrieve $W$ from $\mathcal{N}(W)$ by
\begin{align}\label{eq:TOrderonretrieve}
  W(x_1,y_1,x_2,y_2) & = \lambda^2(\{z\in[0,1]^2 \mid ((x_1,y_1),(x_2,y_2),z)\in\mathcal{N}(W)_\prec\}).
\end{align}
In other words, the mapping $W\mapsto \mathcal N(W)$ is injective.

\medskip

The second arrow follows from Theorem~\ref{thm:herl} for $T_{\operatorname{Order}}$ that (unlike $T_{\operatorname{LinOrder}}$)
{\em is} a Horn theory.

\medskip

For the third arrow, let $I\interpret{T_{\operatorname{Order}}}{T_{\operatorname{ExtendedOrder}}}$ be the structure-erasing interpretation that erases the linear order. Syntactically, it is easy to see that the mapping $\pi^\ast(I)\function{\HomT{T_{\operatorname{ExtendedOrder}}}}{\HomT{T_{\operatorname{Order}}}}$
(see Section~\ref{sec:flag}) is surjective. That is (in theonic language), given a $T_{\operatorname{Order}}$-on $\mathcal{N}$, we can find a $T_{\operatorname{ExtendedOrder}}$-on $\mathcal{N}'$ such that
$\phi_{\mathcal{N}} = \phi_{I(\mathcal{N}')}$. Indeed, since every partial order on a \emph{finite} set can be extended to a linear order, we know that for every \emph{finite} model $M$ of $T_{\operatorname{Order}}$, there
exists a model $N$ of $T_{\operatorname{ExtendedOrder}}$ such that $I(N) = M$. Fix any sequence $(M_n)_{n\in\mathbb{N}}$ of partial orders converging
to $\phi_{\mathcal N}$ and let $(N_n)_{n\in\mathbb{N}}$ be an arbitrary sequence of models of $T_{\operatorname{ExtendedOrder}}$ such that $I(N_n)=M_n$. By possibly passing to a subsequence, we may suppose that
$(N_n)_{n\in\mathbb{N}}$ converges to some $\psi\in\HomT{T_{\operatorname{ExtendedOrder}}}$, and we can let $\mathcal{N}'$ be any strong $T_{\operatorname{ExtendedOrder}}$-on such that $\phi_{\mathcal{N}'} = \psi$.
By Theorem~\ref{thm:flagpi} and Remark~\ref{rmk:theoninterpret}, we get $\phi_{I(\mathcal{N}')} = \phi_{\mathcal{N}'}\comp\pi^I = \phi_{\mathcal{N}}$ (and in fact, for this syntactic construction, we do not even need the fact that $\mathcal{N}$ is Borel).

However, since the construction above is syntactic, the theons $\mathcal{N}$ and $\mathcal{N}'$ are geometrically unrelated, that is, there are no a priori reasons why they should satisfy
$I(\mathcal{N}') = \mathcal{N}$ (not even a.e.). One naive attempt to achieve a semantic construction might be to invoke Proposition~\ref{prop:uniquenessother} and align partial orders $\mathcal N, \mathcal N_{\prec_2}'$. Then one might hope that this alignment could be used to create the linear extension on $\mathcal N$ from $\mathcal N_{\prec_1}'$.

There are no general reasons for this plan to work, however, and Example~\ref{ex:no_omega1} below shows that this in fact can fail quite badly for another simple pair of theories.
For a geometric construction of $\mathcal{N}'$, we do need the following deep measure-theoretic result.

\begin{theorem}[{{\cite[Theorem~1.10 and discussion thereafter]{HMPP}}}]\label{thm:measext}
  Every measurable partial order on a complete atomless probability space can be extended to a measurable linear order.
\end{theorem}

The geometric construction then proceeds as follows. Let $\mathcal{N}$ be a strong Borel $T_{\operatorname{Order}}$-on over $\Omega = (X,\mathcal{A},\mu)$. Our first task is to take care of vanishing and bad pairs (cf.~Lemma~\ref{lem:E2*measurable}).
Recall from Definition~\ref{def:symm} that $A_{\mathcal{N}}(x,y)\df\{z\in X\mid (x,y,z)\in\mathcal{N}\}$ and let
\begin{align*}
  \mathcal{V} & \df \{(x,y)\in\mathcal{E}_2^*(\Omega)\mid \mu(A_{\mathcal{N}}(x,y)) = 0\};\\
  \mathcal{B} & \df \{(x,y)\in\mathcal{E}_2^*(\Omega)\mid \mu(A_{\mathcal{N}}(x,y)) > 0\land\mu(A_{\mathcal{N}}(y,x)) > 0\}
\end{align*}
be the sets of \emph{vanishing pairs} and of \emph{bad pairs} respectively. By Fubini's Theorem, both these sets are Borel and $(\mathcal{V}\times X)\cap\mathcal{N}$ has zero measure.

We claim that $\mathcal{B}$ also has zero measure. Suppose not, then for some $n\in\mathbb{N}_+$, the set
\begin{align*}
  \mathcal{B}_n & \df
  \set{(x,y)\in\mathcal{E}_2^*(\Omega)}{\mu(A_{\mathcal{N}}(x,y)) > \frac{1}{n}\land
    \mu(A_{\mathcal{N}}(y,x)) > \frac{1}{n}}
\end{align*}
must have positive measure.

Pick $\rn{x_1},\rn{x_2},\rn{x_3}\in X$ uniformly at random. Then by Cauchy-Schwarz,
$$
\prob{(\rn{x_1}, \rn{x_3})\in\mathcal B_n\land (\rn{x_2}, \rn{x_3}) \in\mathcal B_n} \geq \prob{(\rn{x_1}, \rn{x_3})\in\mathcal B_n}^2>0.
$$
By the last part of Lemma~\ref{lem:atomless}, this implies that there exist {\em pairwise distinct} $x_1,x_2,x_3\in X$ such that $(x_1,x_3)\in\mathcal B_n\land (x_2,x_3)\in\mathcal B_n$. Now, the point $(x_1,x_2,x_3)$ can be
easily extended to $x\in\mathcal E_3(\Omega)$ (with $x_{\{i\}}=x_i$ for $i=1,2,3$) violating either the transitivity axiom or the asymmetry axiom (for the pair $(x_1,x_2)$). This contradicts the fact that $\mathcal{N}$ is a
strong $T_{\operatorname{Order}}$-on\footnote{The proof that bad pairs have zero measure is much simpler here in comparison to Lemma~\ref{lem:E2*measurable} due to the fact that the starting theon $\mathcal{N}$ is \emph{strong}.}, therefore $\mathcal{B}$ has zero measure.

This means that the $T_{\operatorname{Order}}$-on
\begin{align*}
  \widetilde{\mathcal{N}} & \df \mathcal{N}\setminus((\mathcal{V}\cup\mathcal{B})\times X)
\end{align*}
differs from $\mathcal{N}$ only by a zero-measure set. Note that $\widetilde{\mathcal{N}}$ is also a strong Borel $T_{\operatorname{Order}}$-on. Define now the relation
\begin{align*}
  x \preceq' y & \equiv x=y\lor \mu(A_{\widetilde{\mathcal{N}}}(x,y)) > 0,
\end{align*}
and note that since $A_{\widetilde{\mathcal{N}}}(x,y)\neq\emptyset\implies\mu(A_{\widetilde{\mathcal{N}}}(x,y)) > 0$ and since $\widetilde{\mathcal{N}}$ does not have bad pairs, $\preceq'$ is a partial order. From Fubini's Theorem, it also follows that it is a Borel set. By Theorem~\ref{thm:measext} above we can extend $\preceq'$ to a measurable linear order $\leq'$. We then let
\begin{align*}
  \mathcal{N}'_{\prec_1}
  & \df
  \{x\in\mathcal{E}_2(\Omega) \mid x_{\{1\}} \leq' x_{\{2\}}\},
  &
  \mathcal{N}'_{\prec_2}
  & \df
  \widetilde{\mathcal{N}},
\end{align*}
and it follows that $\mathcal{N}'$ is a strong $T_{\operatorname{ExtendedOrder}}$-on satisfying $I(\mathcal{N}') = \widetilde{\mathcal{N}}$ and hence also $\mu(I(\mathcal{N}')\symmdiff\mathcal{N}) = 0$.

\medskip

Thus, it remains to show how to ``naturally'' obtain a poseton $W$ satisfying $\phi_W = \phi_{I(\mathcal{N})}$ from a strong $T_{\operatorname{ExtendedOrder}}$-on $\mathcal{N}$.

The core of this construction (and of the construction in the next section) is given by the next lemma, which says that given an open interpretation $I\interpret{T_{\operatorname{LinOrder}}}{T}$, we can convert every $T$-on into a strong $T$-on over $\Omega_2$ ($=([0,1]^2,\mathcal{B}_2,\lambda^2)$) such that the linear order is given by the natural order of $[0,1]$ in the first coordinate, similarly to standard $T_{\operatorname{Perm}}$-ons in Section~\ref{sec:permutons}. As Example~\ref{ex:no_omega1} below suggests, this
statement is more subtle than it may appear at first glance.

\begin{lemma}\label{lem:LinOrderplus}
  Let $I\interpret{T_{\operatorname{LinOrder}}}{T}$ be an open interpretation. If $\mathcal{N}$ is a $T$-on over $\Omega=(X,\mathcal{A},\mu)$, then there exists a strong $T$-on $\mathcal{N}_2$ over $\Omega_2$ such that $\phi_{\mathcal{N}_2}=\phi_{\mathcal{N}}$ and
  \begin{equation}\label{eq:naturallinorder}
    \begin{aligned}
      & \!\!\!\!\!\!
      \{x\in\mathcal{E}_2(\Omega_2) \mid \pi_1(x_{\{1\}}) < \pi_1(x_{\{2\}})\}
      \subseteq
      I(\mathcal{N}_2)_\prec
      \\
      & \subseteq
      \{x\in\mathcal{E}_2(\Omega_2) \mid \pi_1(x_{\{1\}}) \leq \pi_1(x_{\{2\}})\}.
    \end{aligned}
  \end{equation}
\end{lemma}

\begin{proof}
  As in Remark~\ref{rmk:strucerase}, let $\widehat{T}$ be the theory obtained from $T_{\operatorname{LinOrder}}\cup T$ by adding the axiom $x\prec y\equiv I(\prec)(x,y)$. Then the open interpretation $\widehat{I}\interpret{T}{\widehat{T}}$ acting as identity on $T$ is an isomorphism and the diagram
  \begin{equation*}
    \begin{tikzcd}
      T_{\operatorname{LinOrder}}\arrow[r,"I"]\arrow[d,"S",swap] & T\arrow[d,"\widehat{I}"]\\
      T_{\operatorname{LinOrder}}\cup T\arrow[r,"A"] & \widehat{T}
    \end{tikzcd}
  \end{equation*}
  commutes, where $S$ and $A$ are the structure-erasing and the axiom-adding interpretations respectively.

  Recall that $\HomT{T_{\operatorname{LinOrder}}}$ has a unique element (cf.~Examples~\ref{ex:orders} and~\ref{ex:twoorderedtheons}); it is represented by both $T_{\operatorname{LinOrder}}$-ons
  \begin{align*}
    \mathcal{G}_\prec & \df \{x\in\mathcal{E}_2(\Omega_1) \mid x_{\{1\}} < x_{\{2\}}\}
  \end{align*}
  and $I(\mathcal{N})$. By Proposition~\ref{prop:uniquenessother}, there exist symmetric functions $h_1\function{\mathcal{E}_1(\Omega_1)\times\mathcal{E}_1(\Omega_1)}{\Omega}$ and $h_2\function{\mathcal{E}_2(\Omega_1)\times\mathcal{E}_2(\Omega_1)}{\Omega}$ measure preserving on h.o.a.\ such that
  \begin{align*}
    \widehat{h}_2(x,\widehat{x})\in I(\mathcal{N})_\prec
    & \equiv
    x\in\mathcal{G}_\prec
    \equiv
    x_{\{1\}} < x_{\{2\}}
  \end{align*}
  for almost every $(x,\widehat{x})\in\mathcal{E}_2(\Omega_1)\times\mathcal{E}_2(\Omega_1)$.

  Pick then arbitrary symmetric functions
  $h_d\function{\mathcal{E}_d(\Omega_1)\times\mathcal{E}_d(\Omega_1)}{\Omega}$ measure preserving on
  h.o.a.\ for every $d\geq 3$ and define
  \begin{align*}
    \mathcal{N}'_P & \df \{x\in\mathcal{E}_{k(P)}(\Omega_2) \mid \widehat{h}_{k(P)}(x)\in\mathcal{N}_P\}
  \end{align*}
  for every predicate symbol $P$ in the language of $T$, where we use the natural identification between $\mathcal{E}_d(\Omega_1)\times\mathcal{E}_d(\Omega_1)$ and $\mathcal{E}_d(\Omega_2)$ and define $\mathcal{N}'_\prec$ as the set in the left-hand side of~\eqref{eq:naturallinorder}. Since $\widehat{I}^{-1}$ acts identically on $T$ and as $I$ on $T_{\operatorname{LinOrder}}$, Proposition~\ref{prop:uniquenessother} implies that $\mathcal{N}'$ is a weak $\widehat{T}$-on satisfying $\phi_{\mathcal{N}'} = \phi_{\widehat{I}^{-1}(\mathcal{N})}$, hence $\phi_{\widehat{I}(\mathcal{N}')} = \phi_{\mathcal{N}}$.

  By Proposition~\ref{prop:ierl}, there exists a strong $\widehat{T}$-on $\mathcal{N}''$ whose peons contain all the density points of the corresponding peons of $\mathcal{N}'$ and are disjoint from the set of density points of the complements of the corresponding peons. This in particular implies that $\mathcal{N}''$ satisfies $\phi_{\mathcal{N}''} = \phi_{\mathcal{N}'}$ and that $\mathcal{N}''_\prec$ still satisfies~\eqref{eq:naturallinorder}.

  Let then $\mathcal{N}_2\df \widehat{I}(\mathcal{N}'')$. Since $I(\mathcal{N}_2) = S(A(\mathcal{N}''))$ and since both $S$ and $A$ act identically on $T_{\operatorname{LinOrder}}$, it follows that $(\mathcal{N}_2)_\prec = \mathcal{N}''_\prec$, so it also satisfies~\eqref{eq:naturallinorder}. On the other hand, we have
  \begin{align*}
    \phi_{\mathcal{N}_2} & = \phi_{\widehat{I}(\mathcal{N}'')} = \phi_{\widehat{I}(\mathcal{N}')} = \phi_{\mathcal{N}}.
  \end{align*}
  Finally, since $\mathcal{N}''$ is a strong $\widehat{T}$-on, it follows that $\mathcal{N}_2$ is a strong $T$-on.
\end{proof}

\begin{example} \label{ex:no_omega1}
Lemma~\ref{lem:LinOrderplus} is no longer true if we replace $\Omega_2$ by
$\Omega_1$.

Let $T\cong T_{\operatorname{LinOrder}}^2$ be the
extension of $T_{\operatorname{LinOrder}}$ with a unary predicate symbol $A$, and let $I\interpret{T_{\operatorname{LinOrder}}}{T}$ be the structure-erasing interpretation. Consider the (weak) $T$-on $\mathcal N$ over $\Omega_2$ in which $\mathcal N_{\prec}$ is given by the left-hand side of~\eqref{eq:naturallinorder} and $\mathcal N_A \df \{x\in\Omega_2\mid \pi_2(x)\leq 1/2\}$. Then there does not exist any (weak) $T$-on $\mathcal N'$
over $\Omega_1$ such that $\phi_{\mathcal N'} = \phi_{\mathcal N}$ and $\{ x\in \mathcal E_2(\Omega_1)\mid x_{\{1\}}< x_{\{2\}}\} \subseteq \mathcal N'_\prec \subseteq \{ x\in \mathcal E_2(\Omega_1)\mid x_{\{1\}}\leq x_{\{2\}}\}$.

Indeed, assume the contrary, i.e., that there exists a measurable
set $\mathcal N'_A\subset [0,1]$ such that $\phi_{\mathcal N} =
\phi_{\mathcal N'}$. Then the latter fact would have readily implied
that $\mathcal N'_A$ has density $1/2$ in every non-empty interval. That
is impossible by Proposition~\ref{prop:Ldensityae}.
\end{example}

We can now establish poseton cryptomorphism.  Given a $T_{\operatorname{ExtendedOrder}}$-on $\mathcal{N}$, Lemma~\ref{lem:LinOrderplus} above with the structure-erasing interpretation $I\interpret{T_{\operatorname{LinOrder}}}{T_{\operatorname{ExtendedOrder}}}$ defined by
$I(\prec)(x,y)\df x\prec_1 y$ gives us a strong $T_{\operatorname{ExtendedOrder}}$-on $\mathcal{N}_2$ satisfying~\eqref{eq:naturallinorder}  and $\phi_{\mathcal{N}_2}=\phi_{\mathcal{N}}$. Then we set
\begin{equation}\label{eq:poseton}
  \begin{aligned}
    A(x,y) & \df \{z\in[0,1]^2\mid (x,y,z)\in(\mathcal{N}_2)_{\prec_2}\};
    \\
    W(x,y) & \df
    \begin{dcases*}
      \lambda^2(A(x,y)),
      & if $A(x,y)$ is measurable and $\pi_1(x) < \pi_1(y)$;
      \\
      0,
      & otherwise.
    \end{dcases*}
  \end{aligned}
\end{equation}
It is straightforward to check that~\eqref{eq:naturallinorder} and the fact that $\mathcal{N}_2$ is strong imply that $W$ (when viewed as a function $[0,1]^4 \to [0,1]$) is a poseton and $\phi_W = \phi_{J(\mathcal{N})}$ for the structure-erasing interpretation $J\interpret{T_{\operatorname{Order}}}{T_{\operatorname{ExtendedOrder}}}$.

\medskip

We finish this subsection with poseton uniqueness. Just as in the case of graphon uniqueness, the higher order variables $x_{\{1,2\}}$ of $T_{\operatorname{Order}}$-ons get ``integrated out'' (cf.~Remark~\ref{rmk:graphon_uniqueness}). Let us also remark that the original list of equivalences for poseton uniqueness in~\cite[Theorem~7.1]{Jan} also includes other items that are not covered here.

\begin{theorem}[{{\cite[Theorem~7.1]{Jan}}}]
  Let $W_1$ and $W_2$ be posetons. The following are equivalent.
  \begin{enumerate}[label={\arabic*.}, ref={\arabic*)}]
  \item We have $\phi_{W_1} = \phi_{W_2}$.
    \label{it:posetonphi}
  \item There exist measure preserving functions $f,g\function{[0,1]}{[0,1]^2}$ such that
    \begin{align*}
      W_1(f(x),f(y)) & = W_2(g(x),g(y)),
    \end{align*}
    for almost every $(x,y)\in[0,1]^2\times[0,1]^2$.
    \label{it:posetonfg}
  \item There exists a measure preserving function $h\function{[0,1]^2\times[0,1]^2}{[0,1]^2}$ such that
    \begin{align*}
      W_1(h(x,\widehat{x}),h(y,\widehat{y})) & = W_2(x,y)
    \end{align*}
    for almost every $(x,\widehat{x},y,\widehat{y})\in([0,1]^2)^4$.
    \label{it:posetonh}
  \end{enumerate}
\end{theorem}

\begin{proof}(sketch)
  The implications~\ref{it:posetonfg}~$\Longrightarrow$~\ref{it:posetonphi} and~\ref{it:posetonh}~$\Longrightarrow$~\ref{it:posetonphi} are trivial.

  \smallskip

  For the implication~\ref{it:posetonphi}~$\Longrightarrow$~\ref{it:posetonfg}, suppose $\phi_{W_1}=\phi_{W_2}$ and consider the $T_{\operatorname{Order}}$-ons $\mathcal{N}(W_1)$ and $\mathcal{N}(W_2)$ defined in~\eqref{eq:TOrderon}. By Proposition~\ref{prop:uniquenessother}, there exist $f_1,g_1\function{\mathcal{E}_1(\Omega_1)}{\Omega_2}$ and $f_2,g_2\function{\mathcal{E}_2(\Omega_1)}{\Omega_2}$ symmetric and measure preserving on h.o.a.\ such that
  \begin{align*}
    \widehat{f}_2(x)\in\mathcal{N}(W_1) & \equiv \widehat{g}_2(x)\in\mathcal{N}(W_2)
  \end{align*}
  for almost every $x\in\mathcal{E}_2(\Omega_1)$. Then by~\eqref{eq:TOrderonretrieve}, item~\ref{it:posetonfg} follows for $f = f_1$ and $g = g_1$.

  \smallskip

  The implication~\ref{it:posetonphi}~$\Longrightarrow$~\ref{it:posetonh} follows by an analogous argument using item~\ref{it:uoh} of Proposition~\ref{prop:uniquenessother} instead.
\end{proof}

\subsection{Limits of interval graphs}

We now consider limits of interval graphs, which were first studied in~\cite{DHJ}.

\begin{definition}
  An \emph{interval graph} is a graph $G$ for which there exists a family of intervals $(I_v)_{v\in V(G)}$ in the real line\footnote{It is easy to see that we may suppose that all such intervals are of the form $[a,b]$ for $0\leq a\leq b\leq 1$ and that no two intervals have coinciding endpoints. By dilating locally, one can further assume that the left endpoints are of the form $i/n$ for some $i\in \{0,1,\ldots,n-1\}$ where $n=\lvert V(G)\rvert$.} such that $vw\in E(G)$ if and only if $I_v\cap I_w\neq\emptyset$. We let $T_{\operatorname{IntervalGraph}}$ be the theory of interval graphs.

  An \emph{interval graph limit} is a probability measure $\mu$ over $([0,1]^2,\mathcal{B}_2)$ such that $\mu(\{(a,b)\in[0,1]^2 \mid b < a\}) = 0$ and such that the first marginal $(\pi_1)_\ast(\mu)$ is equal to the Lebesgue measure.
\end{definition}

For a fixed graph $G$, we define $p(G,\mu)$ by the following probabilistic experiment. We let $\rn{X_1},\rn{X_2},\ldots,\rn{X_m}$ be i.i.d.\ random variables picked according to the measure $\mu$ and define the random structure $\rn{M}$ in $\mathcal{K}_m[\{E\}]$ by letting
\begin{align}\label{eq:randomintervalgraph}
  R_{E,\rn{M}} & = \{(v,w)\in[m]^2 \mid v\neq w\land
  [\pi_1(\rn{X_v}),\pi_2(\rn{X_v})]\cap[\pi_1(\rn{X_w}),\pi_2(\rn{X_w})]\neq\emptyset\},
\end{align}
where we let $[a,b]\df\emptyset$ when $a > b$. From the definition above, it follows that $\rn{M}$ is always an interval graph. We then define $p(G,\mu) = \prob{\rn{M}\cong G}$ and define the functional $\phi_\mu = p(\place,\mu)$.

\medskip

The first direction of the cryptomorphism between interval graph limits and $T_{\operatorname{IntervalGraph}}$-ons follows by an argument similar to the one in Section~\ref{sec:permutons}: the marginal condition implies that $\mu$ satisfies Assumption~P, so we can consider the (strong Borel) $T_{\operatorname{IntervalGraph}}$-on $\mathcal{N}(\mu)$ over $([0,1]^2,\mathcal{B}_2,\mu)$ given by
\begin{align}\label{eq:standardTIntervalGraph-on}
  \mathcal{N}(\mu)_E
  & \df
  \{x\in\mathcal{E}_2(([0,1]^2,\mathcal{B}_2,\mu)) \mid
                [\pi_1(x_{\{1\}}),\pi_2(x_{\{1\}})]\cap[\pi_1(x_{\{2\}}),\pi_2(x_{\{2\}})]\neq\emptyset\}.
\end{align}
We call $\mathcal N(\mu)$ the \emph{standard $T_{\operatorname{IntervalGraph}}$-on associated with $\mu$}. It is straightforward to see that $\phi_{\mathcal{N}(\mu)} = \phi_\mu$.

\medskip

For the other direction, we will employ an intermediate theory just as in Section~\ref{sec:posetons}. Let $T_{\operatorname{LeftIntervalGraph}}$ be the theory obtained from $T_{\operatorname{LinOrder}}\cup T_{\operatorname{IntervalGraph}}$ by adding the axiom
\begin{align}\label{eq:intervalaxiom}
  \forall x_1\forall x_2\forall x_3,
  (x_1\prec x_2\land x_2\prec x_3\land E(x_1,x_3)\to E(x_1,x_2)),
\end{align}
where $\prec$ and $E$ are the predicate symbols of $T_{\operatorname{LinOrder}}$ and $T_{\operatorname{IntervalGraph}}$ respectively, and let $I\interpret{T_{\operatorname{IntervalGraph}}}{T_{\operatorname{LeftIntervalGraph}}}$ be the structure-erasing interpretation.

The intended meaning of $T_{\operatorname{LeftIntervalGraph}}$ is that its interval graphs have the extra information about the order of the left endpoints of its intervals. With this in mind, if $G$ is an interval graph
represented by a family of intervals $([a_v,b_v])_{v\in V(G)}$, then defining the linear order $\prec$
on $V(G)$ by $v\prec w \equiv a_v < a_w$, we get a model $M$ of $T_{\operatorname{LeftIntervalGraph}}$ such that $I(M) = G$.
By the same syntactic argument as in Section~\ref{sec:posetons}, it follows that given a $T_{\operatorname{IntervalGraph}}$-on $\mathcal{N}$, there exists a $T_{\operatorname{LeftIntervalGraph}}$-on $\mathcal{N}'(\mathcal{N})$ such that $\phi_{\mathcal{N}} = \phi_{I(\mathcal{N}'(\mathcal{N}))}$. We would like to remark that we have not been able to come up with an entirely semantic argument (like Theorem~\ref{thm:measext}) in the context of interval graphs. One good approach to this problem might be to analyze existing algorithms for constructing interval representations of interval graphs and see if they are ``transferable'' to the infinite world but it does not seem to be an easy thing to do.

\smallskip

It remains to show how to construct an interval graph limit $\mu_{\mathcal{N}}$ from a
$T_{\operatorname{LeftIntervalGraph}}$-on $\mathcal{N}$ satisfying $\phi_{\mu_{\mathcal{N}}} = \phi_{I(\mathcal N)}$.
To do so, we appeal again to Lemma~\ref{lem:LinOrderplus} for $I(\prec)(x,y)\df x\prec y$. Let $\mathcal{N}_2$ be a strong $T_{\operatorname{LeftIntervalGraph}}$-on over $\Omega_2$ satisfying~\eqref{eq:naturallinorder} and, for $x,y\in [0,1]^2$, define
\begin{align*}
  A(x,y)
  & \df
  \{z\in[0,1]^2 \mid (x,y,z)\in(\mathcal{N}_2)_E\};
  \\
  \chi(x)
  & \df
  \pi_1(x) + \lambda^2(\{y\in [0,1]^2 \mid \pi_1(x) < \pi_1(y)\land \lambda^2(A(x,y))>0\});
  \\
  s(x) & \df (\pi_1(x), \chi(x)).
\end{align*}
Fubini's Theorem guarantees that the functions $\chi,s$ are defined and measurable a.e. We claim that the pushforward measure $\mu_{\mathcal{N}} \df s_*\lambda^2$ is the desired interval graph limit.

Since the first coordinate of $s$ is the projection, the first marginal of $\mu_{\mathcal{N}}$ is $\lambda$ and since $\chi(x)\geq\pi_1(x)$, we have $\mu_{\mathcal{N}}(\{(a,b)\in[0,1]^2 \mid b < a\}) = 0$, so $\mu_{\mathcal{N}}$ is an interval graph limit.

Note now that since $\mathcal{N}_2$ is a strong $T_{\operatorname{LeftIntervalGraph}}$-on satisfying~\eqref{eq:naturallinorder}, by Theorem~\ref{thm:characterization} for the axiom~\eqref{eq:intervalaxiom}, it follows that if $x,y,z\in[0,1]^2$ are such that $A(x,z)\neq\emptyset$ and $\pi_1(x) < \pi_1(y) < \pi_1(z)$, then $A(x,y)=[0,1]^2$. This implies that
\begin{equation}\label{eq:intervalchi}
  \begin{aligned}
    \chi(x) & = \sup\{t\in[\pi_1(x),1] \mid \pi_1(x) = t\lor\exists u\in[0,1] A(x,(t,u))\neq\emptyset\}
    \\
    & = \sup\{t\in[\pi_1(x),1] \mid \pi_1(x) = t\lor\exists u\in[0,1] A(x,(t,u))=[0,1]^2\}
  \end{aligned}
\end{equation}
for almost every $x\in[0,1]$.

Let then $\Omega \df ([0,1]^2,\mathcal{B}_2,\mu_{\mathcal{N}})$. By the definition of $\mu_{\mathcal{N}}$,
the functions
\begin{align*}
  \begin{functiondef}
    f_1\colon & \mathcal{E}_1(\Omega_2) & \longrightarrow & \Omega\\
    & x & \longmapsto & s(x)
  \end{functiondef}
  & &
  \begin{functiondef}
    f_2\colon & \mathcal{E}_2(\Omega_2) & \longrightarrow & \Omega\\
    & x & \longmapsto & s(x_{\{1,2\}})
  \end{functiondef}
\end{align*}
are symmetric and measure preserving on h.o.a.\ and from~\eqref{eq:intervalchi}, it follows that
\begin{align*}
  x\in(\mathcal{N}_2)_E
  & \equiv
  [\pi_1(x_{\{1\}}),\chi(x_{\{1\}})]\cap [\pi_1(x_{\{2\}}),\chi(x_{\{2\}})] \neq\emptyset
\end{align*}
for almost every $x\in\mathcal{E}_2(\Omega_1)$, which by Proposition~\ref{prop:uniquenessother} implies $\phi_{I(\mathcal{N})} = \phi_{\mathcal{N}(\mu_{\mathcal{N}})} = \phi_{\mu_{\mathcal{N}}}$.

\medskip

\subsection{Lineons} \label{sec:lineons}

Our last example is motivated by research on limits of functions on finite
vector spaces~\cite{HHH,Szeg,Yos}, and it is of somewhat different nature. As
we noticed several times before, our theory is not directly applicable to
first-order languages containing function symbols, like the language of
rings. This section illustrates that when the structure in question possesses
sufficiently high symmetries, it is still possible to salvage a significant
part of it.

Let $A\subseteq \mathbb F_p^n$ be given. For $m\leq n$, let us consider a
randomly chosen {\em linear} mapping $\rn\alpha\function{\mathbb
F_p^m}{\mathbb F_p^n}$ that, after composing it with the characteristic
function $\mathbf 1_A\function{\mathbb F_p^n}{\{0,1\}}$ of $A$ gives a
probability distribution over functions from $\mathbb F_p^m$ to $\{0,1\}$.
Removing the zero vector from the domain, we get a distribution over
functions $f\function{\mathbb F_p^m\setminus \{0^m\}}{\{0,1\}}$. Hence, just
as in Section~\ref{sec:prel}, we can define densities $p(f,A)$, converging
sequences of subsets $A_n\subseteq \mathbb F_p^n$ (for $p$ fixed and
$n\to\infty$), etc. We note that the quantities additive combinatorics is
typically interested in, like $\expect{\mathbf 1_A(\rn x)\mathbf 1_A(\rn
x+\rn h)\mathbf 1_A(\rn x+\rn y)\mathbf 1_A(\rn x+\rn y+\rn h)}\ (\rn x,\rn
y,\rn h\in_R \mathbb F_p^n)$, are retrievable from a finite number of
densities $p(f,A)$.

Upon a moment's reflection, it is clear how to formulate this in the
framework of theons (cf.~\cite[\S 1.4]{Szeg}), but we need a small
twist.

\begin{definition}
A first-order language $\mathcal L$ (as always, consisting of predicate
symbols only) is {\em locally finite} if for every $k>0$ it contains
only finitely many $k$-ary predicate symbols.
\end{definition}

Note that if $T$ is a {\em canonical} theory in a locally finite language
$\mathcal L$, then for every fixed $n>0$, there are only finitely many
non-isomorphic models of $T$ of size $n$. This is the only property that we actually
need from $T$, and it is easy to verify that our formalism extends
to this situation straightforwardly; in particular, we still can define
$T$-ons with all nice properties.

Let us now return to linear mappings and assume for a moment that $p=2$. Then we introduce the locally finite language
$\mathcal L_{\operatorname{Lin}}$ that has one $k$-ary predicate symbol $E_k$
for every $k\geq 1$, and the canonical theory $T_{\operatorname{Lin}}\df
T_{1\operatorname{-Hypergraph}}\cup T_{\operatorname{Graph}}\cup T_{3\operatorname{-Hypergraph}}\cup
T_{4\operatorname{-Hypergraph}}\cup\cdots$ in our language asserting that all
predicate symbols are symmetric. Any given $A\subseteq \mathbb
F_2^n$ can be turned into a model $M_A$ of the theory
$T_{\operatorname{Lin}}$ with vertex set $\mathbb{F}_2^n$ if we interpret the symbols $E_k$ by
$$
R_{E_k,M_A} \df \set{(x_1,\ldots,x_k) \in (\mathbb F_2^n)^k}{x_1,\ldots,x_k\ \text{pairwise distinct}\
\land x_1+\cdots+x_k\in A}.
$$

If $p>2$ the notation becomes slightly heavier: for any linear form
$a_1x_1+\cdots+ a_kx_k$ with $a_i\in \{1,\ldots,p-1\},\ a_1\geq\cdots\geq a_k$
we introduce its own predicate symbol $E_{\vec a}$, along with axioms
asserting that $E_{\vec a}$ is symmetric under those elements of $S_k$ that
stabilize $\vec a$. But this does not entail any principal changes in what
follows.

A linear mapping $\alpha\function{\mathbb F_p^m}{\mathbb F_p^n}$ is uniquely
determined by its values $\alpha(e_1),\ldots,\alpha(e_m)$, where
$B = \{e_1,\ldots,e_m\}$ is a fixed basis in $\mathbb F_p^m$. Moreover, if
$\alpha(e_1),\ldots,\alpha(e_m)$ are pairwise distinct then $f = (\mathbf
1_A\circ\alpha)\vert_{\mathbb F_p^m\setminus \{0^m\}}$ if and only if
$\alpha\vert_B$ determines an embedding of the model
$N_{f,B} \df M_{f^{-1}(1)}\vert_B$ into $M_A$ in the theory $T_{\operatorname{Lin}}$. In
particular, we have $p(f,A) = p(N_{f,B}, M_A)\pm O(p^{-n})$ (the error term
accounts for collisions that are allowed in the left-hand side but not in the
right-hand side), and hence both formulations are cryptomorphic. Following
the well-established tradition, we will call {\em lineons} both classes of
convergent sequences $A_n\subseteq \mathbb F_p^n$, as well as
$T_{\operatorname{Lin}}$-ons resulting from such sequences.

Note, however, that the situation here is very different from all other
examples considered in this section. The reason is that lineons make only a
{\em subspace} in the space of all $T_{\operatorname{Lin}}$-ons, and the
structure of this subspace is at the moment understood quite poorly. We can
only contribute the following simple remarks.
\begin{itemize}
\item We cannot add any extra axioms to the theory $T_{\operatorname{Lin}}$ without losing some lineons. Indeed, it is easy to see that
  for every model $M$ of $T_{\operatorname{Lin}}$, there is a model $M_A$ for $A\subseteq \mathbb F_2^m$ such
  that $p(M,M_A)>0$ and then $p(M,M_{A_n})$ is bounded away from $0$ in the increasing sequence $A_n\df A\times
  \mathbb F_2^{n-m}$.

\item On the other hand, lineons make a proper subspace in the space of all
    $T_{\operatorname{Lin}}$-ons. For example, the densities of all components $E_k$ must be the same in every lineon. This is certainly
    not true in an arbitrary $T_{\operatorname{Lin}}$-on.

\item We at least know that this is a closed subspace (in the standard
    product topology, see Section~\ref{sec:convergence}). This simply
    follows from the fact that lineons can be described as accumulation
    points of the set  $\set{p(\place, M_A)\in [0,1]^{\mathcal
    M[T_{\operatorname{Lin}}]}}{A\subseteq \mathbb F_2^m, m\in\mathbb{N}_+}$ (for the notation see again Section~\ref{sec:convergence}). Since
    $[0,1]^{\mathcal M[T_{\operatorname{Lin}}]}$ is metrizable, the set of
    accumulation points of any of its subsets is closed. As we will see at the end of this section, this implies that the set of all lineons can be in principle defined by countably many inequalities and equations. Giving an {\em explicit} description of these relations, however, is an enormously difficult task that, in a slightly different language, is a recurrent theme in arithmetic combinatorics.

\item Expanding on the previous item, we can show that the space of lineons
    is path-connected. Before we present the argument below, let us remark
    that it is not entirely obvious even for the space of
    $T_{\operatorname{Lin}}$-ons itself or, for that matter, even for
    graphons unless we have the whole theory of limit objects at our
    disposal. The argument below can be also used to give a simple
    elementary proof of connectedness in those cases as well.
\end{itemize}

\begin{theorem}
The space of all lineons, viewed as a subset in $[0,1]^{\mathcal
M[T_{\operatorname{Lin}}]}$, is path-connected.
\end{theorem}
\begin{proof} (sketch) Every lineon corresponds to a convergent sequence $\{A_{n_t}\subseteq \mathbb
F_p^{n_t}\},\ n_1<n_2<\cdots < n_t<\cdots$. By letting $A_n\df A_{n_t}\times
\mathbb F_p^{n-n_t}$, where $t$ is the largest index for which $n_t\leq n$,
we can assume without loss of generality that the sets $A_n$ in our convergent sequence are
defined for {\em all} positive integers $n$.

Next, let $\{A_n\subseteq \mathbb F_p^n\}$ and $\{B_n\subseteq \mathbb
F_p^n\}$ be two convergent sequences. Consider their convex combination
$(1-t)\cdot \mathbf 1_{A_n} +t\cdot \mathbf 1_{B_n}\function{\mathbb
F_p^n}{[0,1]}$ depending on the parameter $t$ (see Remark~\ref{rmk:convex} below). For every fixed $t\in [0,1]$
this sequence contains a subsequence converging to a lineon $\phi_t\in
[0,1]^{\mathcal M[T_{\operatorname{Lin}}]}$, with $\phi_0$ and $\phi_1$ being
precisely the lineons the sequences $\{A_n\}$ and $\{B_n\}$ are converging
to. One remaining technical problem is that the mapping $t\mapsto \phi_t$
need not necessarily be continuous since e.g.\ the sequences corresponding to
different $t$ may be disjoint.

In order to circumvent this problem, let us fix an arbitrary enumeration
$t_1,\ldots,t_m,\ldots$ of all {\em rationals} in $[0,1]$ and, by induction
on $m$, construct the lineons $\phi_{t_m}$ in such a way that for any $m<m'$, the lineon
$\phi_{t_{m'}}$ is derived from a subsequence of the sequence used for
defining the lineon $\phi_{t_m}$. This property implies that for any fixed
model $M$, the function $\mathbb Q\cap [0,1]\to \mathbb [0,1]$ given by
$q\mapsto (\phi_q)_M$ is Lipschitz, with the corresponding Lipschitz
constant $c_M$ depending only on $M$. We can then use a definition of distance in $[0,1]^{\mathcal
M[T_{\operatorname{Lin}}]}$ analogous to Definition~\eqref{eq:distance}, but with denominators $2^n\cdot\max\{c_{M_n}, 1\}$ instead; such distance still induces the product topology and we see that the function $\phi$ itself is {\em
uniformly} continuous on $\mathbb Q\cap[0,1]$. Hence it can be uniquely extended to a
continuous function $[0,1]\to [0,1]^{\mathcal M[T_{\operatorname{Lin}}]}$.
Finally, recalling that the set of lineons is closed in $[0,1]^{\mathcal
M[T_{\operatorname{Lin}}]}$ finishes the proof.
\end{proof}

\begin{remark}\label{rmk:convex}
  As the reader may have noticed, in the proof above we used actually densities of
  sets $A\subseteq\mathbb{F}_p^k$ in functions $f\function{\mathbb{F}_p^n}{[0,1]}$
  that are not necessarily $\{0,1\}$-valued. It is straightforward to extend
  the definitions of lineon densities to $[0,1]$-valued functions ({\em weighted} lineons).

  Furthermore, one may also argue that the proof above only shows that the space of limits of $[0,1]$-valued
  functions is path-connected, which a priori may be much larger than just the space of lineons (i.e., the
  space of limits of $\{0,1\}$-valued functions). However, these spaces turn out to be the same by the
  following standard argument. Every $[0,1]$-valued function $f\function{\mathbb{F}_p^m}{[0,1]}$ is the almost sure
  limit of the sequence of random sets $\rn{A_n}\subseteq\mathbb{F}_p^{m+n}$ where each
  element $x\in\mathbb{F}_p^{m+n}$ is independently present in $\rn{A_n}$ with probability $f(\pi(x))$,
  where $\pi\function{\mathbb{F}_p^{m+n}}{\mathbb{F}_p^m}$ is the projection to the first $m$ coordinates. Hence,
  every $p(\place,f)$ is in fact a lineon. Finally, the fact that lineons form a closed set implies that limits
  of $[0,1]$-valued functions are also lineons.
\end{remark}

Once we know that lineons form a proper subspace of $\HomT{T_{\operatorname{Lin}}}$, one natural question
that arises is: can we at least describe it with countably many polynomial inequalities? Note that $\HomT{T_{\operatorname{Lin}}}$ itself can be described like this from the explicit
description given by flag algebras (see Section~\ref{sec:flag}).

As the final result of this section, we will prove that this is indeed possible, and in fact, we can prove that this is the
case for {\em any} closed subset of the set of all limit objects (and any theory). We will need the following basic lemma
in topology/geometry (since we were not able to find it in the literature, we offer a simple proof in Appendix~\ref{sec:convex}).

\begin{lemma}\label{lem:closedconvex}
  Let $C\subseteq[0,1]^{\mathbb{N}}$ be a closed convex set. Then $C$ can be described as the set of points $p\in[0,1]^{\mathbb{N}}$ that satisfy some system of countably many linear inequalities, each depending on finitely many coordinates of $p$.
\end{lemma}

\begin{theorem}\label{thm:closeddescribe}
  Every closed set $F\subseteq\HomT{T}$ can be described with countably many polynomial inequalities (each in a finite number of coordinates).
\end{theorem}

\begin{proof}
  Since $\HomT{T}$ itself can be described with countably many polynomial inequalities (see Section~\ref{sec:flag}) and in
  light of Lemma~\ref{lem:closedconvex} above, it is enough to show that there exists a closed convex set $C\subseteq
  [0,1]^{\mathcal M[T]}$ such that $F = C\cap\HomT{T}$. Let $C$ be the closed convex hull of $F$. Since both $F$ and $C$ are
  compact, by~\cite[Theorem~3.28]{Rud}, every element of $C$ is of the form $\expect{\rn{\phi}}$ for some random variable
  $\rn{\phi}$ supported on $F$. Thus, it is enough to show that if $\rn{\phi}$ is supported on $F$ and $\expect{\rn{\phi}}\in\HomT{T}$, then we must have $\prob{\rn{\phi}=\expect{\rn\phi}}=1$. This is immediate from item~\ref{it:crypto} of Theorem~\ref{thm:allcrypto} and the observation that the isomorphism between distributions on $\HomT T$ and exchangeable random structures is linear and hence preserves extreme points.
\end{proof}

But all in all, these pieces of information about the space of lineons
generate more questions than they answer, and we defer further discussion to
the concluding section.

\section{Conclusion and open problems} \label{sec:conclusion}

One topic that we have touched only very briefly is the distance and topology
on the space of all $T$-ons for a given theory $T$. Both can be defined via
densities (see Section~\ref{sec:prel}), but to the best of our knowledge, no
alternative ``intrinsic'' description bypassing statistical sampling is known
even for the case of 3-hypergraphons. This is of course in sharp contrast
with the case of ordinary graphons where the characterization in terms of
cut-distance~\cite{BCL*} is an inherent part of a fruitful and beautiful
theory~\cite[Part~3]{Lov4}. Is it possible to give an analogous
characterization for an arbitrary theory $T$, presumably based on a suitable
generalization of the cut-distance to higher dimensions?

Expanding on the previous question, in the sparse graph setting, convergence in (normalized) cut-distance
yields the theory of $L^p$-graphons~\cite{BCCZ14,BCCZ18}, which forms a semantic limit theory just as ordinary
graphons. A natural question is whether it is possible to generalize such notions of cut-distance convergence in
the sparse setting to arbitrary combinatorial objects. While one stepping stone in answering this would be a generalization of cut-distance to higher dimensions, another would be to
provide some sort of syntactic limit for sparse graphs capturing cut-distance convergence.

On the other hand, for convergence of densities there have been successful adaptations of flag algebras to
some theories in the sparse setting~\cite{Bab11,BHLL14}, providing a syntactic limit for them (recall that
convergence in cut-distance and convergence of densities are not equivalent in the sparse setting). Can one
then also provide a semantic limit for these theories? One possible approach would be to substitute the
exchangeability notion used in this work with the notion of partial exchangeability (see~\cite{DiFr84} for a
definition and further references on the topic).

As we have mentioned before, our proof of the Induced Euclidean Removal Lemma (Theorem~\ref{thm:ierl}) heavily depends on the axiom of choice. A natural question is whether one can prove it without using the axiom of choice or whether this theorem is equivalent to some weak form of the axiom of choice. As we saw in Section~\ref{sec:removal}, this question is ``morally similar'' to the question of whether strong theons can be always made Borel (using whatever methods).
The simplest theory for which we do not know the answer is the theory of
graphs with forbidden {\em induced} cycles $C_4$.

The undecidability result by Hatami and Norine~\cite{HaNo} says that given an element $f\in\mathcal A[T_{\operatorname{Graph}}]$
with rational coefficients, it is algorithmically undecidable whether the
inequality $f\geq 0$ holds in the limit for all graphs. It is very natural to ask for which other universal theories $T$ in a relational language this is true; let us provisionally call such theories ``statistically undecidable''. One obvious example of a statistically decidable theory is $T_{\operatorname{LinOrder}}$, and this notion is clearly invariant under isomorphisms as defined in Section~\ref{sec:interpretations}. More generally,
if $I\interpret{T_1}{T_2}$ is an open interpretation {\em such that $\pi^\ast(I)$ is surjective} (see Example~\ref{ex:erasing}) and $T_2$ is statistically decidable then $T_1$ is also statistically decidable. But other than these simple remarks we have little to say on the subject. In particular, we do not know how statistical decidability compares with ordinary (in the full first-order logic) decidability, either way.

Let us finally ask several questions about lineons (Section~\ref{sec:lineons}). On the semantical side, it would be very interesting to
give a theon-free description of limit objects that better takes into
account the specific $\mathbb F_p$-linear structure. A natural test for the
validity of such a description should be a (two-sided) cryptomorphism between
these hypothetical objects and lineons, as described syntactically in Section~\ref{sec:lineons}. While a progress in this direction has been reported in~\cite{HHH,Szeg,Yos}, the question still remains open.

We concluded Section~\ref{sec:lineons} by observing that the set of lineons is a proper closed subset of
$\HomT{T_{\operatorname{Lin}}}\subseteq [0,1]^{\mathcal
M[T_{\operatorname{Lin}}]}$ that can be described by countably many polynomial inequalities. However, our proof does not provide these inequalities explicitly, so it remains an open question whether the set of lineons has an explicit description (i.e., computably enumerable) by countably many polynomial inequalities.

Finally, the hypergraph interpretation of lineons allows us to transfer to
this framework the whole host of questions asked in the asymptotical extremal
combinatorics about {\em concrete} relations $f\geq 0$, and some of them will
look quite natural in this setting. For example, inspired by the concept of
{\em commonality} in graph theory, we might ask the following. Is it true
that for any coloring $c\function{\mathbb F_2^n}{\mathbb F_2}$ the density of
monochromatic affine triangles $(x,y,x+y)$ is always at least $1/4-o(1)$? If
the answer is yes, what is the structure of extremal lineons, are there any
other examples besides affine mappings and quasi-random ones?

\section*{Acknowledgment}

We are grateful to Alexander Shen and Caroline Terry for pointing out to us
the references~\cite{AuT} and~\cite{ArC}, respectively. Our special thanks are due to Persi Diaconis for many insightful remarks made on the first version of the paper, as
well as some important references to earlier work. 


\bibliographystyle{alpha}
\bibliography{refs}

\appendix

\section{M\"{o}bius Inversion}
\label{sec:Mobius}

In this section we present the M\"{o}bius Inversion used in our particular application in a lightweight ad hoc manner. For a more thorough introduction to the topic, we refer the interested reader to~\cite{Spi}.

\begin{definition}
  Let $(P,\preceq)$ be a finite poset. A \emph{(closed and bounded) interval} in $P$ is a set of the form
  \begin{align*}
    [a,b] & \df \{c\in P \mid a\preceq c\preceq b\},
  \end{align*}
  for some $a,b\in P$ with $a\preceq b$. The \emph{length} of the interval $[a,b]$, denoted by $\ell([a,b])$ is defined as the cardinality of the largest chain contained in $[a,b]$, that is, we have
  \begin{align*}
    \ell([a,b]) & \df \max\{\ell\in\mathbb{N} \mid \exists a_1,a_2,\ldots,a_\ell\in[a,b],(a=a_1\prec a_2\prec\cdots\prec a_\ell=b)\}.
  \end{align*}

  The \emph{M\"{o}bius Function} of the poset $P$ is the function $\mu\function{P\times P}{\mathbb{R}}$ defined by induction on the length of intervals by
  \begin{align*}
    \mu(a,b) & =
    \begin{dcases*}
      1, & if $a=b$;\\
      -\sum_{\substack{c\in [a,b]\\c\neq b}} \mu(a,c), & if $a\prec b$;\\
      0, & if $a\npreceq b$.
    \end{dcases*}
  \end{align*}
\end{definition}

\begin{theorem}[M\"{o}bius Inversion]
  Let $(P,\preceq)$ be a finite poset and $f$ and $g$ be real-valued functions defined on $P$. If
  \begin{align}\label{eq:mobiuscondition}
    \forall x\in P, f(x) & = \sum_{y\succeq x} g(y),
  \end{align}
  then
  \begin{align*}
    \forall x\in P, g(x) & = \sum_{y\succeq x} \mu(x,y)f(y).
  \end{align*}
\end{theorem}

\begin{proof}
  Follows directly from the calculation below.
  \begin{align*}
    \sum_{y\succeq x}\mu(x,y)f(y)
    & =
    \sum_{y\succeq x}\mu(x,y)\sum_{z\succeq y}g(z)
    \\
    & =
    \sum_{z\succeq x}g(z)\sum_{y\in [x,z]}\mu(x,y)
    \\
    & =
    g(x)\mu(x,x) + \sum_{z\succ x}g(z)\biggl(\mu(x,z)+\sum_{\substack{y\in [x,z]\\y\neq z}}\mu(x,y)\biggr)
    \\
    & =
    g(x).
  \end{align*}
\end{proof}

As a corollary, if we consider a model $M$ of a theory $T$ and form the poset $P_M$ of (labeled) models of $T$ whose set of
vertices is $V(M)$ and with the partial order $\subseteq$, the equation~\eqref{eq:inj-ind} is exactly of the
form~\eqref{eq:mobiuscondition}. Hence we get
\begin{align}\label{eq:ind-inj}
  \tind(M,N) & = \sum_{M'\supseteq M}\mu(M,M')\tinj(M',N).
\end{align}

\begin{example}[uniform hypergraphs]
  In the theory of $k$-uniform hypergraphs $T_{k\operatorname{-Hypergraph}}$, we have
  \begin{align*}
    \mu(M,N) & = (-1)^{\lvert E(N)\setminus E(M)\rvert},
  \end{align*}
  for every $M\subseteq N$ with $V(M)=V(N)$, where $E(M)$ denotes the set of hyperedges of $M$, that is, we have
  \begin{align*}
    E(M) & = \{\{v_1,v_2,\ldots,v_k\}\subseteq V(M) \mid R_{E,M}(v_1,v_2,\ldots,v_k)\}.
  \end{align*}

  Therefore, we have
  \begin{align*}
    \tind(M,N) & = \sum_{M'\supseteq M}(-1)^{\lvert E(M')\setminus E(M)\rvert}\tinj(M',N).
  \end{align*}
\end{example}

\begin{example}[uniform directed hypergraphs]
  In the theory of directed $k$-uniform hypergraphs (that is, the canonical theory in the language with a single predicate whose arity is $k$ and with only the canonicity axiom~\eqref{eq:canonical}), we have
  \begin{align*}
    \mu(M,N) & = (-1)^{\lvert R_{E,N}\setminus R_{E,M}\rvert},
  \end{align*}
  for every $M\subseteq N$ with $V(M)=V(N)$.

  Therefore, we have
  \begin{align*}
    \tind(M,N) & = \sum_{M'\supseteq M}(-1)^{\lvert R_{E,M'}\setminus R_{E,M}\rvert}\tinj(M',N).
  \end{align*}
\end{example}

\begin{example}[permutations and tournaments]
  Since in the theory of permutations and in the theory of tournaments we have
  \begin{align*}
    V(M) = V(N)\land M\subseteq N & \implies M=N,
  \end{align*}
  then the M\"{o}bius Function for these theories is trivial:
  \begin{align*}
    \mu(M,N) & =
    \begin{dcases*}
      1, & if $M = N$;\\
      0, & if $M\neq N$.
    \end{dcases*}
  \end{align*}
\end{example}

Which is another way of saying that induced and non-induced settings for these theories are the same.

\section{Measure theory proofs}
\label{sec:measure}

In this section, we prove Lemma~\ref{lem:atomless} and Theorem~\ref{thm:measiso}.

\begin{proofof}{Lemma~\ref{lem:atomless}}
  The forward implication follows directly from the fact that if $\mu(\{x\}) > 0$ for some $x\in X$, then $\{x\}$ is an atom of $(X,\mathcal{A},\mu)$.

  \medskip

  For the backward implication, fix arbitrarily a metric $\operatorname{dist}(x,y)$ leading to the topology and let $\{p_n \mid n\in\mathbb{N}\}$ be a countable dense set in $X$. Let us denote by $B(x,r)$ the open ball of radius $r$ centered at $x$ and denote by $\operatorname{diam}(B)\df\sup\{\operatorname{dist}(x,y) \mid x,y\in B\}$ the diameter of a set $B$.

  Suppose that $A\in\mathcal{A}$ is an atom of $\Omega$ and let us show that there exists $x\in X$ such that $\mu(\{x\}) > 0$.

  We construct inductively a sequence $(A_m)_{m\in\mathbb{N}}$ of measurable sets satisfying
  \begin{itemize}
  \item $A_m$ is an atom of $\Omega$;
  \item $\operatorname{diam}(A_m)\leq 2^{-m}$;
  \item $A_{m+1}\subseteq A_{m}$;
  \item $\mu(A_m) = \mu(A)$.
  \end{itemize}

  As an initial step, we set $A_{-1} \df A$ (note that we do not claim that $A_{-1}$ has finite diameter). Given $A_{m-1}$ for some $m\in\mathbb{N}$, since $A_{m-1} = \bigcup_{n\in\mathbb{N}} (B(p_n,2^{-m-1})\cap A_{m-1})$, we know that $\mu(B(p_{n_m},2^{-m-1})\cap A_{m-1}) > 0$ for some $n_m$ and since $A_{m-1}$ is an atom and $\mu(A_{m-1})=\mu(A)$, this measure must be $\mu(A)$. Set then $A_m \df B(p_{n_m},2^{-m-1})\cap A_{m-1}$ and note that all required properties are satisfied.

  Consider now the set $B \df \bigcap_{m\in\mathbb{N}} A_m$. Since $\mu(B) = \lim_{m\to\infty} \mu(A_m) = \mu(A) > 0$, it follows that $B$ is non-empty. On the other hand, we know that $\operatorname{diam}(B) = 0$, so $B$ must be of the form $\{x\}$ for some $x\in X$ and we get $\mu(\{x\}) > 0$.

  \medskip

  Finally, let us prove that in this case the diagonal $D = \set{(x,x)\in X^2}{x\in X}$ has measure~0
  w.r.t.\ the product measure $\mu^2$. First note that since the topology is metrizable (and hence Hausdorff), the diagonal is a
  closed set, hence measurable. On the other hand, by Fubini's Theorem, it is enough to show that for every
  $x\in X$, the section $D(x) = \{y\in X \mid (x,y)\in D\}$ has measure~0. But this is indeed the case as
  $D(x) = \{x\}$ for every $x\in X$ and $\mu$ is atomless.
\end{proofof}

To prove Theorem~\ref{thm:measiso}, we first need a technical lemma.

\begin{lemma}\label{lem:zerocontinuum}
  If $A\subseteq [0,1]$ is a Lebesgue measurable set with $\lambda(A) > 0$, then there exists a set $K\subseteq A$ of zero Lebesgue measure and cardinality of the continuum.
\end{lemma}

\begin{proof}
 By possibly replacing $A$ with a closed set $F\subseteq A$ such that $\lambda(F)>\lambda(A)/2>0$, we may suppose that $A$ is closed.
We construct a Cantor subset in $A$.

More precisely, let us define closed sets $F_n$ inductively with $F_{n+1}\subseteq F_{n}$ and such that $F_n$ is a union
  of $2^n$ disjoint closed intervals contained in $[0,1]$, each such interval $I$ satisfying $\lambda(I\cap A)
  = \lambda(A)/3^n$. It will be convenient to index the intervals by finite strings over $\{0,1\}$, with the ones indexed by $\{0,1\}^n$ corresponding to $F_n$.

We start with $F_0 = I_\epsilon = [0,1]$, where $\epsilon$ is the empty string. Suppose by induction that we have constructed $F_n = \mathop{\stackrel\cdot\bigcup}_{\alpha\in\{0,1\}^n} I_\alpha$. For $\alpha\in\{0,1\}^n$, we let
  \begin{align*}
    \ell_\alpha
    & \df
    \sup\left\{p\in I_\alpha\cap A \middle\vert
    \lambda([0,p]\cap I_\alpha\cap A) \leq \frac{\lambda(A)}{3^{n+1}}\right\};
    \\
    r_\alpha
    & \df
    \sup\left\{p\in I_\alpha\cap A \middle\vert
    \lambda([0,p]\cap I_\alpha\cap A) \leq \frac{2\lambda(A)}{3^{n+1}}\right\};
    \\
    I_{\alpha 0} & \df \{p\in I_\alpha \mid p \leq \ell_\alpha\};
    \\
    I_{\alpha 1} & \df \{p\in I_\alpha \mid p \geq r_\alpha\}.
  \end{align*}
Clearly $I_{\alpha 0}$ and $I_{\alpha 1}$ are disjoint closed intervals contained in $I_\alpha$ and since
  all singletons have Lebesgue measure zero, it follows that $\lambda(I_{\alpha 0}\cap A) = \lambda(I_{\alpha
    1}\cap A) = \lambda(A)/3^{n+1}$.
Setting $F_{n+1} \df \bigcup_{\alpha\in\{0,1\}^{n+1}} I_\alpha$ concludes the construction.

  Let then $K \df \bigcap_{n\in\mathbb{N}} (F_n\cap A)$ and note that since $F_{n+1}\subseteq F_{n}$ and $\lambda(F_n\cap A) = \lambda(A)\cdot (2/3)^n$, we have $\lambda(K) = 0$.
For every infinite string $\alpha\in\{0,1\}^{\mathbb{N}_+}$, the decreasing family of bounded closed sets
  \begin{align*}
    \{A\cap I_{\alpha_1\alpha_2\cdots\alpha_n} \mid n\in\mathbb{N}_+\}
  \end{align*}
contains at least one point $p_\alpha\in K$. Moreover, since $I_\alpha\cap I_\beta=\emptyset$ for any pair $\alpha\neq\beta$ of the same length, all these points are pairwise different and hence $K$ has the cardinality of the continuum.
\end{proof}

\begin{proofof}{Theorem~\ref{thm:measiso}}
  For the backward implication, $([0,1],\mathcal{L}_1,\lambda^1)$ is the completion of the space $([0,1], \mathcal B,\lambda^1)$
  satisfying Assumption~P, and Assumption~P is clearly invariant under measure-isomorphisms.

  For the forward implication, suppose $\Omega=(X,\mathcal{A},\mu)$ is a space satisfying Assumption~P and
  let $\overline{\Omega}=(X,\overline{\mathcal{A}},\overline{\mu})$ be its completion. By
  Theorem~\ref{thm:bogachev}, we know that $\Omega$ is measure-isomorphic modulo~0
  to $([0,1],\mathcal{B}_1,\lambda^1)$, that is, there exist $B\in\mathcal{B}_1$ and $A\in\mathcal{A}$
  with $\lambda^1([0,1]\setminus B)=\mu(X\setminus A)=0$ and there exists a
  measure-isomorphism $f\function{(B,\mathcal{B}_1\vert_B,\lambda^1\vert_B)}{(A,\mathcal{A}\vert_A,\mu\vert_A)}$.

  Since every measurable set $C$ in $\overline{\mathcal{A}}$ is of the form $C = C'\cup C''$ for
  some $C'\in\mathcal{A}$ and some $C''$ contained in a zero $\mu$-measure set and the same holds
  for $\mathcal{L}_1$, it follows that $f$ is also a measure-isomorphism
  between $(B,\mathcal{L}_1\vert_B,\lambda^1\vert_B)$
  and $(A,\overline{\mathcal{A}}\vert_A,\overline{\mu}\vert_A)$.

  By possibly replacing $B$ with $B\setminus Y$ for a set $Y\subseteq B$ such
  that $\lambda^1(Y)=0$ and $Y$ has cardinality of the continuum, whose existence is guaranteed by
  Lemma~\ref{lem:zerocontinuum}, we may suppose that $[0,1]\setminus B$ (and hence also $X\setminus A$) has
  cardinality of the continuum. Note that this conclusion does {\em not} use either the axiom of choice or
  the continuum hypothesis, only the Cantor--Schr\"oder--Bernstein Theorem.

  Then we can extend $f$ to a bijection between $[0,1]$ and $X$ arbitrarily and the resulting function is a
  measure-isomorphism between $([0,1],\mathcal{L}_1,\lambda^1)$
  and $(X,\overline{\mathcal{A}},\overline{\mu})$ since these measure spaces are complete.
\end{proofof}

\section{Closed convex sets in the Hilbert Cube}
\label{sec:convex}

\begin{proofof}{Lemma~\ref{lem:closedconvex}}
  For every $n\in\mathbb{N}$, let $\pi_n\function{[0,1]^{\mathbb{N}}}{[0,1]^n}$ be the projection to the first $n$ coordinates. Note that since $[0,1]^{\mathbb{N}}$ is compact, the projection $\pi_n$ is a closed map.

  Clearly we have $C\subseteq\bigcap_{n\in\mathbb{N}}\pi_n^{-1}(\pi_n(C))$. On the other hand, if $x\in\bigcap_{n\in\mathbb{N}}\pi_n^{-1}(\pi_n(C))$, then since $\{\pi_n^{-1}(U) \mid n\in\mathbb{N}\land\pi_n(x)\in U\land U\subseteq[0,1]^n\text{ open}\}$ is a basis of neighborhoods of $x$ in $[0,1]^{\mathbb{N}}$ and every such $\pi_n^{-1}(U)$ must have at least one point $y$ of $C$ satisfying $\pi_n(y)=\pi_n(x)$, we get $x\in\overline{C}=C$.

  Therefore, we get
  \begin{align}\label{eq:Ccapproj}
    C & = \bigcap_{n\in\mathbb{N}}\pi_n^{-1}(\pi_n(C)).
  \end{align}
  But (since $\pi_n$ is a closed map) $\pi_n(C)$ is a closed convex set in $[0,1]^n$, hence it can be described as the set of points $p\in[0,1]^n$ satisfying (possibly uncountably many) linear inequalities involving the coordinates of $p$, that is, we have
  \begin{align*}
    \pi_n(C)
    & =
    \left\{p = (p_j)_{j=1}^n\in[0,1]^n \middle\vert \forall i\in I, a_{i,0} + \sum_{j=1}^n a_{i,j} p_j \geq 0\right\},
  \end{align*}
  where $a_{i,j}\in\mathbb{R}$ for every $i\in I$ and $j\in\{0,1,\ldots,n\}$ and $I$ is some (possibly uncountable) set.

  Now, for each $i\in I$ and each $j\in\{0,1,\ldots,n\}$, we let $(q_{i,j,m})_{m\in\mathbb{N}}$ be a sequence of rational numbers such that $q_{i,j,m}\geq a_{i,j}$ and $\lim_{m\to\infty}q_{i,j,m} = a_{i,j}$. Then we have
  \begin{align*}
    \pi_n(C)
    & =
    \left\{p = (p_j)_{j=1}^n\in[0,1]^n \middle\vert
    \forall i\in I, \forall m\in\mathbb{N}, q_{i,0,m} + \sum_{j=1}^n q_{i,j,m} p_j \geq 0\right\}.
  \end{align*}
  But note now that there are at most $\lvert\mathbb{Q}^{n+1}\rvert$ distinct $(q_{i,j,m})_{j=0}^n$, i.e., countably many. Therefore $\pi_n(C)$ can be described by countably many linear inequalities. Equation~\eqref{eq:Ccapproj} then finishes the proof as the intersection is countable.
\end{proofof}

\end{document}